\documentclass[leqno]{amsart}

\usepackage[centertags]{amsmath}
\usepackage[colorlinks]{hyperref}
\usepackage{amsfonts}
\usepackage{MnSymbol}
\usepackage{amsthm}
\usepackage{newlfont}
\usepackage{amscd}
\usepackage{enumerate}
\usepackage{enumitem}

\usepackage[all,2cell]{xy}
\UseAllTwocells
\input xy
\xyoption{2cell}
\xyoption{all}
\usepackage{verbatim}
\usepackage{eucal}
\usepackage{stackrel}
\newcommand{\DD}{\mathscr{D}}

\newcommand{\NN}{\mathbb{N}}

\newcommand{\CC}{\mathbb{C}}
\renewcommand{\DD}{\mathbb{D}}
\newcommand{\ZZ}{\mathbb{Z}}
\newcommand{\LL}{\mathbb{L}}
\newcommand{\RR}{\mathbb{R}}
\newcommand{\BB}{\mathbb{B}}
\newcommand{\EE}{\mathbb{E}}

\def\A{\mathcal{A}}
\def\B{\mathcal{B}}

\def\C{\mathcal{C}}
\def\D{\mathcal{D}}
\def\G{\mathcal{G}}
\def\H{\mathcal{H}}
\def\F{\mathcal{F}}
\def\SS{\mathcal{S}}
\def\M{\mathcal{M}}
\def\N{\mathcal{N}}

\def\I{\mathcal{I}}
\def\cS{\mathcal{S}}
\def\T{\mathcal{T}}
\def\P{\mathcal{P}}

\def\R{\mathcal{R}}
\renewcommand{\L}{\mathcal{L}}

\def\wtl{\widetilde}
\def\HH{\mathcal{H}}
\def\E{\mathcal{E}}
\def\bS{\mathbf{S}}
\def\bT{\mathbf{T}}

\def\fC{\mathfrak{C}}
\def\rN{\mathrm{N}}
\def\rH{\mathrm{H}}
\def\K{\mathcal{K}}

\def\SS{\mathbb{S}}

\def\alp{{\alpha}}
\def\bet{{\beta}}

\def\eps{{\varepsilon}}

\def\sig{{\sigma}}

\def\Om{{\Omega}}

\def\Del{{\Delta}}
\def\Sig{{\Sigma}}

\def\Lam{{\Lambda}}
\def\vphi{\varphi}
\def\Beta{\mathfrak{P}}

\def\BDel{\Del\mkern-14mu\Del}
\newcommand{\HSwarrow}{\kern0.05ex\vcenter{\hbox{\Huge\ensuremath{\Swarrow}}}\kern0.05ex}
\newcommand{\hSwarrow}{\kern0.05ex\vcenter{\hbox{\huge\ensuremath{\Swarrow}}}\kern0.05ex}
\newcommand{\LLSwarrow}{\kern0.05ex\vcenter{\hbox{\LARGE\ensuremath{\Swarrow}}}\kern0.05ex}
\newcommand{\LSwarrow}{\kern0.05ex\vcenter{\hbox{\Large\ensuremath{\Swarrow}}}\kern0.05ex}

\newcommand{\HSearrow}{\kern0.05ex\vcenter{\hbox{\Huge\ensuremath{\Searrow}}}\kern0.05ex}
\newcommand{\hSearrow}{\kern0.05ex\vcenter{\hbox{\huge\ensuremath{\Searrow}}}\kern0.05ex}
\newcommand{\LLSearrow}{\kern0.05ex\vcenter{\hbox{\LARGE\ensuremath{\Searrow}}}\kern0.05ex}
\newcommand{\LSearrow}{\kern0.05ex\vcenter{\hbox{\Large\ensuremath{\Searrow}}}\kern0.05ex}

\newcommand{\HDownarrow}{\kern0.05ex\vcenter{\hbox{\Huge\ensuremath{\Downarrow}}}\kern0.05ex}
\newcommand{\hDownarrow}{\kern0.05ex\vcenter{\hbox{\huge\ensuremath{\Downarrow}}}\kern0.05ex}
\newcommand{\LLDownarrow}{\kern0.05ex\vcenter{\hbox{\LARGE\ensuremath{\Downarrow}}}\kern0.05ex}
\newcommand{\LDownarrow}{\kern0.05ex\vcenter{\hbox{\Large\ensuremath{\Downarrow}}}\kern0.05ex}

\newcommand{\HUparrow}{\kern0.05ex\vcenter{\hbox{\Huge\ensuremath{\Uparrow}}}\kern0.05ex}
\newcommand{\hUparrow}{\kern0.05ex\vcenter{\hbox{\huge\ensuremath{\Uparrow}}}\kern0.05ex}
\newcommand{\LLUparrow}{\kern0.05ex\vcenter{\hbox{\LARGE\ensuremath{\Uparrow}}}\kern0.05ex}
\newcommand{\LUparrow}{\kern0.05ex\vcenter{\hbox{\Large\ensuremath{\Uparrow}}}\kern0.05ex}
\numberwithin{equation}{section}

\newtheorem{thm}[equation]{Theorem}
\newtheorem{cor}[equation]{Corollary}
\newtheorem{lem}[equation]{Lemma}
\newtheorem{pro}[equation]{Proposition}

\newtheorem{obs}[equation]{Observation}

\theoremstyle{definition}
\newtheorem{define}[equation]{Definition}

\newtheorem{example}[equation]{Example}

\newtheorem{defn}[equation]{Definition}
\newtheorem{cons}[equation]{Construction}

\newtheorem{notn}[equation]{Notation}

\theoremstyle{remark}
\newtheorem{rem}[equation]{Remark}

\newtheorem{warning}[equation]{Warning}

\DeclareMathOperator{\hocolim}{hocolim}

\DeclareMathOperator{\Id}{Id}
\DeclareMathOperator{\id}{id}
\DeclareFontFamily{OT1}{pzc}{}
\DeclareFontShape{OT1}{pzc}{m}{it}{<-> s * [1.10] pzcmi7t}{}
\DeclareMathAlphabet{\mathpzc}{OT1}{pzc}{m}{it}

\DeclareMathOperator{\fib}{fib}
\DeclareMathOperator{\ad}{ad}

\DeclareMathOperator{\op}{op}
\DeclareMathOperator{\Map}{Map}

\DeclareMathOperator{\proj}{proj}
\DeclareMathOperator{\df}{def}
\DeclareMathOperator{\Set}{Set}

\DeclareMathOperator{\Cat}{Cat}
\DeclareMathOperator{\Adj}{Adj}

\DeclareMathOperator{\Fun}{Fun}

\DeclareMathOperator{\Obj}{Obj}

\DeclareMathOperator{\Un}{Un}
\DeclareMathOperator{\Ob}{Ob}
\DeclareMathOperator{\Sp}{Sp}
\DeclareMathOperator{\Ho}{Ho}

\DeclareMathOperator{\ev}{ev}

\DeclareMathOperator{\Lan}{Lan}

\DeclareMathOperator{\Ab}{Ab}
\DeclareMathOperator{\Env}{Env}

\DeclareMathOperator{\Tw}{Tw}
\DeclareMathOperator{\der}{h}

\renewcommand{\H}{H}

\DeclareMathOperator{\St}{St}

\DeclareMathOperator{\Reedy}{Reedy}

\DeclareMathOperator{\cov}{cov}

\DeclareMathOperator{\Lift}{Lift}
\DeclareMathOperator{\sca}{sc}
\DeclareMathOperator{\lcc}{lcc}
\DeclareMathOperator{\Joy}{Joy}
\DeclareMathOperator{\KQ}{KQ}
\DeclareMathOperator{\spl}{sp}

\DeclareMathOperator{\gp}{gp}
\DeclareMathOperator{\pt}{pt}
\DeclareMathOperator{\coop}{coop}

\DeclareMathOperator{\Torsors}{Torsors}
\DeclareMathOperator{\Free}{\mathcal{F}ree}
\DeclareMathOperator{\mar}{mark}
\DeclareMathOperator{\thin}{thin}
\DeclareMathOperator{\forget}{\mathcal{F}orget}
\DeclareMathOperator{\LaxCone}{LaxCone}
\DeclareMathOperator{\Cone}{Cone}

\DeclareMathOperator{\LFib}{LFib}

\def\x{\stackrel}

\def\Hom{\textrm{Hom}}

\def\<{\left<}
\def\>{\right>}
\def\({\left(}
\def\){\right)}

\newcommand{\tgpd}{\kern0.05ex\vcenter{\hbox{\footnotesize\ensuremath{2}}}\kern0.05ex\mathcal{G}pd} 

\def\lrar{\longrightarrow}
\def\llar{\longleftarrow}
\def\hrar{\hookrightarrow}

\newcommand{\adj}{\mathrel{\substack{\longrightarrow \\[-.6ex] \x{\upvdash}{\longleftarrow}}}}

\def\ovl{\overline}

\def\uline{\underline}

\newcommand{\pair}[4][10pt]{\ar@/^#1/[#2]^-{#3} \ar@/_#1/[#2]_-{#4}}
\renewcommand{\twocell}[5]{\ar@<#4ex>@{}[#1] \ar@<#4ex>@{=>}?(#3)+/dl #5cm/;?(#3)+/ur #5cm/^{#2}}

\setlist[itemize]{leftmargin=*}
\setlist[enumerate]{leftmargin=20pt}

\title{Quillen cohomology of $(\infty,2)$-categories}

\author{Yonatan Harpaz}
\email{harpaz@math.univ-paris13.fr}
\address{Institut Galil\'ee\\ Universit\'e Paris 13\\ 99 avenue J.B. Cl\'ement\\ 93430 Villetaneuse\\ France.}
\author{Joost Nuiten}
\email{j.j.nuiten@uu.nl}
\address{Mathematical Institute\\ Utrecht University\\ 3508 TA Utrecht\\ The
Netherlands.}
\author{Matan Prasma}
\email{mtnprsm@gmail.com}
\address{Faculty of Mathematics\\ University of Regensburg\\ Regensburg\\93040 \\Germany.}

\date{}
\begin{document}
\begin{abstract}
In this paper we study the homotopy theory of parameterized spectrum objects in the $\infty$-category of $(\infty, 2)$-categories, as well as the Quillen cohomology of an $(\infty, 2)$-category with coefficients in such a parameterized spectrum. More precisely, we construct an analogue of the twisted arrow category for an $(\infty,2)$-category $\CC$, which we call its twisted $2$-cell $\infty$-category. We then establish an equivalence between parameterized spectrum objects over $\CC$, and diagrams of spectra indexed by the twisted $2$-cell $\infty$-category of $\CC$. Under this equivalence, the Quillen cohomology of $\CC$ with values in such a diagram of spectra is identified with the two-fold suspension of its inverse limit spectrum. As an application, we provide an alternative, obstruction-theoretic proof of the fact that adjunctions between $(\infty,1)$-categories are uniquely determined at the level of the homotopy $(3, 2)$-category of $\Cat_\infty$. 
\end{abstract}

\maketitle
\tableofcontents

\section{Introduction}\label{s:intro}

This paper is part of an ongoing project whose goal is to understand the cohomology theory of higher categories. Our approach follows the framework developed by Quillen (\cite{Qui67}), and refined by Lurie (\cite{Lur14}), which gives a general recipe for defining cohomology in an abstract setting. In the case of spaces, this approach recovers generalized cohomology with coefficients in a local system of spectra. When spaces are replaced with $\infty$-categories, previous work of the authors \cite{part2} identifies the corresponding Quillen cohomology as the \textbf{functor cohomology} of diagrams of spectra, indexed by the associated twisted arrow category.

In this paper we take these ideas a step further by studying the Quillen cohomology of $(\infty,2)$-categories. Recall that in the abstract setting of Quillen and Lurie, if $\D$ is a presentable $\infty$-category and $X \in \D$ is an object, then the coefficients for the Quillen cohomology of $X$ are given by $\Om$-spectrum objects in the slice $\infty$-category $\D_{/X}$, which we call \textbf{parameterized spectra} over $X$. The Quillen cohomology groups of $X$ with coefficients in such a parameterized spectrum $M$ are given by the homotopy groups of the mapping spectrum
$$ \Map_{\Sp(\D_{/X})}(L_X,M) ,$$
where $L_X := \Sig^{\infty}_+(\Id_X)$ is the suspension spectrum of $\Id_X \in \D_{/X}$. The parameterized spectrum $L_X$ is also known in this general context as the \textbf{cotangent complex} of $X$. There is hence in principle no obstacle to defining Quillen cohomology of an $(\infty,2)$-category by considering the presentable $\infty$-category $\D = \Cat_{(\infty,2)}$ and following the above formalism. However, this will only yield a tractable theory if one can describe parameterized spectra over an $(\infty,2)$-category in a reasonably concrete way. 

When $\D = \Cat_{\infty}$ the main result of the previous paper~\cite{part2} identifies the $\infty$-category $\Sp((\Cat_{\infty})_{/\C})$ of parameterized spectra over an $\infty$-category $\C$ with the $\infty$-category of functors $\Tw(\C) \lrar \Sp$ from the twisted arrow category to spectra, and the cotangent complex $L_\C$ with the constant functor whose value is the $1$-shifted sphere spectrum $\SS[-1]$. This allows one to access and compute Quillen cohomology of $\infty$-categories in rather explicit terms.

Our goal in this paper is to give a similar description in the case of $(\infty,2)$-categories by constructing a suitable analogue of the twisted arrow category, which we call the \textbf{twisted $2$-cell $\infty$-category} of $\CC$. Informally speaking, the objects of the twisted $2$-cell $\infty$-category can be identified with the $2$-cells of $\CC$, and the morphisms are given via suitable factorizations of $2$-cells. To make this precise we use the scaled unstraightening construction of~\cite{goodwillie}, which allows one to present diagrams of $\infty$-categories indexed by an $(\infty,2)$-category by a suitable fibration of $(\infty,2)$-categories. More precisely, we first encode $\CC$ as a category enriched in marked simplicial sets and consider the $(\infty,2)$-category $\CC_{\Tw}$ obtained from $\CC$ by replacing each mapping object by its (marked) twisted arrow category. We then construct the twisted $2$-cell $\infty$-category of $\CC$ by applying the scaled unstraightening construction to the mapping category functor $\Map: \CC_{\Tw}^{\op} \times \CC_{\Tw} \lrar \Set^+_\Del$. This procedure yields a \textbf{scaled simplicial set} $\Tw_2(\CC)$, which we refer to as the \textbf{twisted $2$-cell $\infty$-bicategory} of $\CC$. Finally, the twisted $2$-cell $\infty$-category $\uline{\Tw}_2(\CC)$ is defined to be the $\infty$-category freely generated by $\Tw_2(\CC)$. 

This approach requires us to work simultaneously with two models for $(\infty,2)$-categories, namely, categories enriched in marked simplicial sets on the one hand, and scaled simplicial sets on the other. We recall the relevant preliminaries in~\S\ref{s:scaledsset} and~\S\ref{s:straightening}, while the construction itself is carried out in~\S\ref{s:twisted-2-cells}. Some concrete examples of interest are described in~\S\ref{s:examples}. 
In the case where $\CC$ is a strict $2$-category we can describe the twisted $2$-cell $\infty$-category more explicitly by replacing the scaled unstraightening procedure with the $2$-categorical Grothendieck construction. The equivalence of these two operations, which may be of independent interest, is proven in~\S\ref{s:discrete}. 
Finally, we use the construction of the twisted $2$-cell $\infty$-category in~\S\ref{s:main} to order to prove our main theorem (see Theorem~\ref{t:main}):
\begin{thm}\label{t:main-intro}
Let $\CC$ be an $(\infty,2)$-category. 
Then there is a natural equivalence of $\infty$-categories
$$ \Sp((\Cat_{(\infty,2)})_{/\CC}) \x{\simeq}{\lrar} \Fun(\uline{\Tw}_2(\CC),\Sp(\cS_\ast)) $$
from the $\infty$-category of parameterized spectrum objects over $\CC$ to the $\infty$-category of functors from $\Tw_2(\CC)$ to spectra. 
Furthermore, this equivalence identifies the cotangent complex $L_{\CC}$ with the constant functor whose value is the twice desuspended sphere spectrum $\SS[-2]$. 
\end{thm}

Theorem~\ref{t:main-intro} identifies the abstract notion of a parameterized spectrum object over an $(\infty,2)$-category $\CC$ with a concrete one: a diagram of spectra indexed by an $\infty$-category $\uline{\Tw}_2(\CC)$. A direct consequence of this is that the associated notion of Quillen cohomology becomes much more accessible:
\begin{cor}\label{c:twisted-arrow-quillen-intro}
Let $\F: \uline{\Tw}_2(\CC) \lrar \Sp$ be a diagram of spectra and let $M_{\F} \in \Sp((\Cat_{(\infty,2)})_{/\CC})$ be the corresponding parameterized spectrum object under the equivalence of Theorem~\ref{t:main-intro}. Then the Quillen cohomology group $\rH^n_Q(\CC;M_{\F})$ is naturally isomorphic to the $(-n-2)$'th homotopy group of the limit spectrum $\lim_{\uline{\Tw}_2(\CC)}\F$. 
\end{cor}

Quillen cohomology, and mostly its relative version (see \S\ref{s:recall}), is naturally suited to support an obstruction theory for the existence of lifts against a certain class of maps, known as \textbf{small extensions}. In the realm of spaces, a natural source of small extensions is given by the consecutive maps $P_{n+1}(X) \lrar P_n(X)$ in the \textbf{Postnikov tower} of $X$, for $n \geq 1$. This leads to the classical obstruction theory for spaces which is based on relative ordinary cohomology with local coefficients (a particular case of relative Quillen cohomology for spaces). The case of $(\infty,1)$-categories was studied by Dwyer and Kan in~\cite{DKS86} (in the framework of simplicial categories) who developed a similar obstruction theory based on the Postnikov filtration of mapping spaces, using a version of relative Quillen cohomology with coefficients in abelian group objects. A possible extension to $(\infty,n)$-categories using the Postnikov filtration of the spaces of $n$-morphisms was first suggested by Lurie in~\cite[\S 3.5]{tft}. We formally establish the existence of such a tower of small extensions in a companion paper~\cite{part4}, see also~\cite{Ng17}. This leads to an obstruction theory for $(\infty,n)$-categories which is based on relative Quillen cohomology. 

When $n=2$ this obstruction theory can be made explicit using our description of Quillen cohomology via the twisted $2$-cell $\infty$-category. In particular, the equivalence of Theorem~\ref{t:main-intro} leads to an explicit criteria for when all the relative Quillen cohomology groups of a map $\CC \lrar \DD$ of $(\infty,2)$-categories \textbf{vanish}, in terms of weak contractibility of certain comma categories. In \S\ref{s:adj} we apply this idea to the problem of \textbf{classification of adjunctions}. In particular, we show that the inclusion of $2$-categories $[1] \lrar \Adj$ from the walking arrow to the walking adjunction has trivial relative Quillen cohomology groups. 
The obstruction theory for $(\infty,2)$-categories then implies that a $1$-arrow $f$ in an $(\infty,2)$-category $\CC$ extends to an adjunction if and only if it extends to an adjunction in the truncated $(3,2)$-category $\Ho_{\leq 3}(\CC)$. In fact, the space of lifts in the square
$$ \xymatrix{
[1] \ar[r]\ar[d] & \CC \ar[d] \\
\Adj \ar[r]\ar@{-->}[ur] & \Ho_{\leq 3}(\CC) \\
}$$
is weakly contractible. This leads to a classification of adjunctions in terms of explicit low dimensional data. We note that the analogous contractibility statement for lifts of $[1] \lrar \Adj$ against $\CC \lrar \Ho_{\leq 2}(\CC)$ was established in~\cite{RV16}, by using a somewhat elaborate combinatorial argument and an explicit cell decomposition of $\Adj$. 
While we hope to convince the reader that the obstruction theoretic proof is simpler in comparison, it should be noted that it only applies to the tower of small extensions $\CC \lrar \Ho_{\leq 3}(\CC)$, yet leaves open the problem of classifying lifts of $[1] \lrar \Adj$ against $\Ho_{\leq 3}(\CC) \lrar \Ho_{\leq 2}(\CC)$. This particular piece of the puzzle can be done by hand, or by using the approach of~\cite{RV16}, but in any case only requires understanding the $3$-skeleton of $\Adj$. It also seems plausible that a suitable non-abelian cohomology approach can be applied in this case. This reflects the typical situation in Postnikov type obstruction theories: the cohomological argument can be used to reduce a homotopical problem (potentially involving an infinite web of coherence issues) to a finite dimensional problem, whose coherence constraints are bounded in complexity.

\subsection{Acknowledgments}
A significant part of this paper was written while the first author was a postdoctoral researcher at the Institut des Hautes \'Etudes Scientifiques, which graciously accommodated all three authors during several work sessions. The authors would like to thank the IHES for its hospitality and superb work environment which allowed for this paper to be born. The second author was supported by the NWO. The third author was supported by grant SFB 1085.

\section{Recollections}\label{s:scaled}
In this section we recall various preliminaries which we require in later parts of the paper. We begin in~\S\ref{s:scaledsset} by recalling various aspects of the theory of $(\infty, 2)$-categories, mostly using the models of scaled simplicial sets (as developed in~\cite{goodwillie}), and categories enriched in marked simplicial sets. In~\S\ref{s:straightening} we recall the straightening and unstraightening operations which allow one to encode a diagram of $\infty$-categories indexed by an $(\infty,2)$-category as a suitable fibration of $(\infty,2)$-categories. The particular case where the diagram takes its values in $\infty$-groupoids leads to the notion of a marked left fibration, which we spell out in~\S\ref{s:marked-left}. Finally, in~\S\ref{s:recall} we recall the notions of stabilization, abstract parameterized spectra and Quillen cohomology, whose specialization to the case of $(\infty,2)$-categories is our main interest in this paper. As in the previous papers,~\cite{part0},~\cite{part1} and~\cite{part2} we adopt the formalism of \textbf{tangent categories} and \textbf{tangent bundles}, which follow Lurie's abstract cotangent complex formalism developed in~\cite[\S 7.3]{Lur14}.

\subsection{Scaled simplicial sets}\label{s:scaledsset}
The homotopy theory of $(\infty, 1)$-categories admits various model-categorical presentations, e.g.\ in terms of the Bergner-Dwyer-Kan model structure on simplicial categories, the Joyal model structure on simplicial sets (with quasicategories as fibrant objects), or the categorical model structure on marked simplicial sets (with fibrant objects the quasicategories, marked by their equivalences). These model categories are related by Quillen equivalences
\begin{equation}\label{e:modelsfor1cat}\xymatrix{
\fC: \Set_\Del\ar@<1ex>[r]^-\sim & \Cat_\Del:\rN \ar@<1ex>[l]_-{\upvdash} & (-)^\flat: \Set^{\Joy}_\Del\ar@<1ex>[r]^-\sim & \Set_\Del^+:\forget,\ar@<1ex>[l]_-{\upvdash}
}\end{equation}
with right adjoints taking the coherent nerve, resp.\ forgetting the marked edges. Let us mention that the categorical model structure on marked simplicial sets is related to the usual Kan-Quillen model structure on simplicial sets by two Quillen adjunctions
\begin{equation}\label{e:gpdcomp}\xymatrix@C=1.8pc{
(-)^\sharp: \Set^{\KQ}_\Del \ar@<1ex>[r] & \Set_\Del^+:(-)^{\mar}\ar@<1ex>[l]_-{\upvdash} & |-|: \Set_\Del^+\ar@<1ex>[r] & \Set^{\KQ}_\Del:(-)^{\sharp}.\ar@<1ex>[l]_-{\upvdash}
}\end{equation}
Here $X^\sharp=(X, X_1)$ is the simplicial set $X$ with all edges marked, $|-|$ simply forgets marked edges and $(X,E_X)^{\mar}$ is the largest simplicial subset of $X$ whose edges are all in $E_X$. Since $|-|$ is a left adjoint, the object $X = |(X,E_X)| \in \Set_\Del^{\KQ}$ is a model for the free $\infty$-groupoid generated by the $\infty$-category $(X,E_X)$, or equivalently, a model for its \textbf{classifying space}.

In this paper we will use two analogous models for the theory of $(\infty,2)$-categories: the model category $\Cat^+_\Del$ of categories enriched in marked simplicial sets, which we will refer to as \textbf{marked-simplicial categories}, and the model category $\Set_\Del^{\sca}$ of \textbf{scaled simplicial sets}. Recall that a scaled simplicial set is a pair $(X,T_X)$ where $X$ is a simplicial set and $T_X$ is a collection of $2$-simplices in $X$ which contains all degenerate $2$-simplices. The $2$-simplices in $T_X$ are refered to as the \textbf{thin triangles}. In~\cite{goodwillie}, Lurie constructs a model structure on the category $\Set^{\sca}_{\Del}$ of scaled simplicial sets which is a model for the theory of $(\infty,2)$-categories. In particular, a scaled version of the coherent nerve construction yields a Quillen equivalence
$$ 
\fC^{\sca}: \Set^{\sca}_\Del \x{\simeq}{\adj} \Cat^+_{\Del}: \rN^{\sca} 
$$
between scaled simplicial sets and marked-simplicial categories (see~\cite[Theorem 4.2.7]{goodwillie}). Following~\cite{goodwillie} we will refer to weak equivalences in $\Set^{\sca}_\Del$ as \textbf{bicategorical equivalences}, and to fibrant objects in $\Set^{\sca}_{\Del}$ as \textbf{$\infty$-bicategories}.

Recall that a scaled simplicial set is called a \textbf{weak $\infty$-bicategory} if it satisfies the extension property with respect to the class of scaled anodyne maps described in~\cite[Definition 3.1.3]{goodwillie}. In particular, every $\infty$-bicategory is a weak $\infty$-bicategory. These extension conditions can be considered as analogous to the inner horn filling conditions of the Joyal model structure. For instance, an inner horn $\Lambda^2_1\lrar X$ admits a thin filler and an inner horn $\Lambda^n_i\lrar X$ with $n\geq 3$ admits a filler as soon as the 2-simplex $\Delta^{\{i-1, i, i+1\}}$ is thin. 

Just as $(\infty, 1)$-categories are related to $\infty$-groupoids via \eqref{e:gpdcomp}, $(\infty, 2)$-categories are related to $(\infty, 1)$-categories via the Quillen adjunctions
\begin{equation}\label{e:catcomp}\xymatrix{
(-)_\sharp: \Set^{\Joy}_\Del \ar@<1ex>[r] & \Set^{\sca}_\Del:(-)^{\thin}\ar@<1ex>[l]_-{\upvdash} & |-|_1: \Set^{\sca}_\Del\ar@<1ex>[r] & \Set^{\Joy}_\Del:(-)_{\sharp},\ar@<1ex>[l]_-{\upvdash}
}\end{equation}
where $X_\sharp = (X, X_2)$ is $X$ with all triangles being thin, $(X, T_X)^{\thin}$ is the maximal simplicial subset of $X$ whose triangles all belong to $T_X$ and $|-|_1$ forgets the thin triangles. Since $|-|_1$ is a left adjoint, the object $X = |(X,T_X)|_1 \in \Set_\Del^{\Joy}$ is a model for the $\infty$-category freely generated by an $\infty$-bicategory $(X,T_X)$.

\begin{rem}\label{r:freeoo1}
Let $\CC$ be a marked-simplicial category and let $\CC_{\mar}$ and $\CC_{|-|}$ be the simplicial categories obtained by applying the product-preserving functors $(-)^{\mar}$ and $|-|$ from \eqref{e:gpdcomp} to all mapping objects. Unraveling the definition of the scaled nerve \cite[Definition 3.1.10]{goodwillie}, one sees that there are natural isomorphisms
$$
\rN\big(\CC_{\mar}\big) \cong \big(\rN^{\sca}(\CC)\big)^{\thin} \qquad \qquad \rN\big(\CC_{|-|}\big)\cong |\rN^{\sca}(\CC)|_1.
$$
Informally speaking we may summarize the above isomorphisms as follows: the $\infty$-category freely generated from $\CC$ has as mapping spaces the $\infty$-groupoids freely generated from the mapping categories of $\CC$, and the maximal sub $\infty$-category of $\CC$ has as mapping $\infty$-groupoids the maximal sub $\infty$-groupoids of the mapping categories of $\CC$.
\end{rem}
A particularly important class of $(\infty,2)$-categories is given by the $(2,2)$-categories, namely, those $(\infty,2)$-categories whose spaces of $2$-cells are all discrete. It is well-known that every $(2,2)$-category can be represented by a (strict) $2$-category, i.e., a category enriched in categories. Given such a (strict) 2-category $\CC$, we can apply the marked nerve construction $\rN^+:\Cat \lrar \Set^+_\Del$ to every mapping category in $\CC$ to obtain a marked-simplicial category $\CC_{\rN^+}$. The scaled nerve of this marked-simplicial category is an $\infty$-bicategory, which can be described as follows.

Let $\BDel^n$ be the $2$-category whose objects are $0,...,n$ and where $\Map_{\BDel^n}(i,j)$ is the poset of subsets of $[n]$ whose minimal element is $i$ and maximal element is $j$. Given a $2$-category $\CC$ we define its \textbf{$2$-nerve} $\rN_2(\CC) \in \Set^{\sca}_{\Del}$ by the formula
$$ \rN_2(\CC)_n = \Fun_2(\BDel^n,\CC). $$
A triangle $\sig \in \rN_2(\CC)$ is thin if and only if the corresponding $2$-functor $\BDel^2 \lrar \CC$ sends the non-identity arrow of $\Map_{\BDel^2}(0,2)$ to an isomorphism.
\begin{rem}\label{r:nerve}
There is a natural isomorphism between the marked-simplicial categories $\BDel^n_{\rN^+}$ and $\fC^{\sca}(\Del^n_\flat)$, where $\Del^n_{\flat}$ is $\Del^n$ with thin triangles only the degenerate ones. It follows that there is a natural isomorphism $\rN_2(\CC) \cong \rN^{\sca}(\CC_{\rN^+})$. We also note that for completely general reasons $\rN_2$ admits a left adjoint $\fC_2: \Set^{\sca}_{\Del} \lrar \Cat_2$ whose value on the $n$-simplices is given by $\fC_2(\Del^n) = \BDel^n$. 
\end{rem}

\subsection{Scaled straightening and unstraightening}\label{s:straightening}
A key property of the model of scaled simplicial sets is that it admits a notion of \textbf{unstraightening}: diagrams of $\infty$-categories indexed by an $\infty$-bicategory $\CC$ can be modeled by certain fibrations $\DD \lrar \CC$. 

\begin{define}\label{d:T-lcc-fib}
Let $(S, T_S)$ be a scaled simplicial set and let $f: X\lrar S$ be a map of simplicial sets. We will say that $f$ is a \textbf{$T_S$-locally coCartesian fibration} if it is an inner fibration and for every thin triangle $\sig: \Del^2 \lrar S$, the base change $\sig^*f: X \times_S \Del^2 \lrar \Del^2$ is a coCartesian fibration.
\end{define}
\begin{define}\label{d:2-fib-a}
For $f: (X,T_X) \lrar (S,T_S)$ a map of scaled simplicial sets, we will say that $f$ is a \textbf{scaled coCartesian fibration} if the underlying map $X \lrar S$ is a $T_S$-locally coCartesian fibration in the sense of Definition~\ref{d:T-lcc-fib} and $T_X = f^{-1}(T_S)$. 
\end{define}

\begin{lem}\label{l:unstisweakbicat}
If $f: (X, T_X)\lrar (S, T_S)$ is a scaled coCartesian fibration and $(S, T_S)$ is a weak $\infty$-bicategory, then $(X, T_X)$ is a weak $\infty$-bicategory.
\end{lem}
\begin{proof}
It will suffice to show that if $f$ is a scaled coCartesian fibration then it satisfies the right lifting property with respect to scaled anodyne maps. To see this, observe that since $f$ is an inner fibration and $T_X = f^{-1}(T_S)$ the right lifting property with respect to maps of type (A) and (B) of~\cite[Definition 3.1.3]{goodwillie} is immediate, and the lifting property with respect to maps of type (C) follows from~\cite[Lemma 3.2.28]{goodwillie} since any degenerate edge of $X$ is locally $f$-coCartesian. 
\end{proof}

To study scaled coCartesian fibrations efficiently it is useful to employ the language of \textbf{categorical patterns} (see~\cite[Appendix B]{Lur14}). Let $S$ be a simplicial set, $E_S$ a collection of edges in $S$ containing all degenerate edges, and $T_S$ a collection of triangles in $S$ containing all degenerate triangles. The tuple $\Beta := (S,E_S,T_S)$ then determines a categorical pattern on $S$, to which one may associate a \textbf{model structure} on the category $(\Set^+_{\Del})_{/(S,E_S)}$ of marked simplicial sets over $(S,E_S)$ (see~\cite[Theorem B.0.20]{Lur14}). The cofibrations of this model structure are the monomorphisms and its fibrant objects are the so called \textbf{$\Beta$-fibered} objects (see~\cite[Definition B.0.19]{Lur14}).
Explicitly, an object $p:(X,E_X) \lrar (S,E_S)$ of $(\Set^+_{\Del})_{/(S,E_S)}$ is $\Beta$-fibered if it satisfies the following conditions:
\begin{enumerate}
\item
The map $p:X \lrar S$ is an inner fibration of simplicial sets.
\item
For every edge $e:\Del^1 \lrar S$ which belongs to $E_S$ the map $e^*p:X \times_S \Del^1 \lrar \Del^1$ is a coCartesian fibration, and the marked edges of $X$ which lie above $e$ are exactly the $e^*p$-coCartesian edges.
\item
Given a commutative diagram
$$ \xymatrix{
\Del^{\{0,1\}} \ar^{e}[r]\ar[d] & X \ar[d] \\
\Del^2 \ar^{\sig}[r] & S \\
}$$
if $e \in E_X$ and $\sig \in T_S$ then $e$ determines a $\sig^*p$-coCartesian edge of $X \times_S \Del^2$.
\end{enumerate}
As in \cite[Appendix B]{Lur14}, we will denote the resulting model category by $(\Set^+_\Del)_{/\Beta}$.
\begin{lem}\label{l:all-the-same}
Let $(S, T_S)\in \Set_\Del^{\sca}$, let $f: X\lrar S$ be an inner fibration and let $\Beta_{T_S}=(S,S_1,T_S)$. Let $E_X$ denote the collection of locally $f$-coCartesian edges and let $T_X = f^{-1}(T_S)$ denote the collection of triangles whose image in $S$ is thin. Then the following are equivalent:
\begin{enumerate}
\item
$f$ is a $T_S$-locally coCartesian fibration.
\item
$(X,E_X)$ is $\Beta_{T_S}$-fibered.
\item
$f: (X,T_X) \lrar (S,T_S)$ is a scaled coCartesian fibration.
\end{enumerate}
\end{lem}
\begin{proof}
The equivalence of (i) and (iii) and the implication (ii) $\Rightarrow$ (i) are immediate. The implication (i) $\Rightarrow$ (ii) follows from~\cite[Remark 2.4.2.13]{Lur09}.
\end{proof}
In light of Lemma~\ref{l:all-the-same} we will denote
$$
(\Set^+_{\Del})^{\lcc}_{/(S,T_S)}:=(\Set^+_{\Del})_{/\Beta_{T_S}}.
$$
The following lemma makes sure that the passage from a $T_S$-locally coCartesian fibration to the associated scaled coCartesian fibration is homotopically sound.
\begin{lem}\label{l:same}
Let $f: X \lrar Y$ be a weak equivalence between fibrant objects in $(\Set^+_{\Del})^{\lcc}_{/(S,T_S)}$ and let $T_X \subseteq X_2$ and $T_Y \subseteq Y_2$ be the subsets of triangles whose images in $S$ belong to $T_S$. Then the map of scaled simplicial sets $(X,T_X) \lrar (Y,T_Y)$ is a bicategorical equivalence.
\end{lem}
\begin{proof}
We first note that the model category $(\Set^+_{\Del})^{\lcc}_{/(S,T)}$ is tensored over $\Set^+_\Del$ (see~\cite[Remark B.2.5]{Lur14}), where the action of $K \in \Set^+_\Del$ is given by $K \otimes (X \lrar S) = K \times X \lrar S$. Since the functor $K \mapsto K^{\flat}$ of \eqref{e:modelsfor1cat} is a product preserving left Quillen functor from $\Set^{\Joy}_\Del$ to $\Set^+_\Del$ we obtain an induced tensoring of $(\Set^+_{\Del})^{\lcc}_{/(S,T)}$ over $\Set^{\Joy}_\Del$.
In particular, if $f: X \lrar Y$ is a weak equivalence between fibrant (and automatically cofibrant) objects, then there exists an inverse map $g: Y \lrar X$ such that $f \circ g$ and $g \circ f$ are homotopic to the respective identities via homotopies of the form $J^{\flat} \times X \lrar X$ and $J^{\flat} \times Y \lrar Y$, where $J$ is a cylinder object for $\Del^0$ in $\Set^{\Joy}_{\Del}$. On the other hand, the model category $\Set^{\sca}_\Del$ is also tensored over $\Set^{\Joy}_\Del$; the action of $K \in \Set^{\Joy}_\Del$ is given by $K \otimes (X,T_X) = (K \times X, K_2 \times T_X)$. We conclude that if $f: X \lrar Y$ is a weak equivalence between fibrant objects in $(\Set^+_{\Del})^{\lcc}_{/(S,T)}$, then the induced map $(X,T_X) \lrar (Y,T_Y)$ of scaled simplicial set has an inverse up to homotopy and is therefore a bicategorical equivalence. 
\end{proof}

Given a map $\vphi: \fC(S,T_S) \lrar \CC$ of marked-simplicial categories, Lurie constructs in~\cite[\S 3.5]{goodwillie} a straightening-unstraightening Quillen adjunction
$$ \St^{\sca}_{\vphi}: (\Set^+_{\Del})^{\lcc}_{/(S,T_S)} \adj \Fun^+(\CC,\Set^+_\Del): \Un^{\sca}_{\vphi} $$
which is a Quillen equivalence when $\vphi$ is a weak equivalence (\cite[Theorem 3.8.1]{goodwillie}). Here the right hand side is the category of $\Set^+_\Del$-enriched functors with the projective model structure. In light of Lemma~\ref{l:all-the-same} one can therefore consider scaled coCartesian fibrations over $(S,T_S)$ as an unstraightened model for an $(\infty,2)$-functor $(S,T_S) \lrar \Cat_{\infty}$. 
\begin{notn}\label{n:variants}
Let $\F \in \Fun^+(\fC^{\sca}(S,T_S),\Set^+_\Del)$ be a functor. We will use the following variants of $\Un^{\sca}_{\vphi}(\F)$:
\begin{itemize}
\item
We will denote by $\uline{\Un}^{\sca}_{\vphi}(\F)$ the simplicial set underlying the marked simplicial set $\Un^{\sca}_{\vphi}(\F)$.
\item
We will denote by $\wtl{\Un}^{\sca}_{\vphi}(\F)$ the scaled simplicial set whose underlying simplicial set is $\uline{\Un}^{\sca}_{\vphi}(\F)$ and whose thin triangles are exactly those whose image in $S$ is thin.
\end{itemize}
\end{notn}
\begin{rem}\label{r:underlying}
When $\F: \CC \lrar \Set^+_\Del$ is a fibrant diagram, the object $\Un^{\sca}_{\vphi}(\F)$ is $\Beta_{T_S}$-fibered over $S$. It then follows from Lemma~\ref{l:all-the-same} that
$$ 
\uline{\Un}^{\sca}_{\vphi}(\F) \lrar S \qquad \text{and}\qquad  \wtl{\Un}^{\sca}_{\vphi}(\F) \lrar S
$$ 
are a $T_S$-locally coCartesian fibration and a scaled coCartesian fibration, respectively. In particular, if $(S,T_S)$ is a weak $\infty$-bicategory then $\wtl{\Un}^{\sca}_{\vphi}(\F)$ is a weak $\infty$-bicategory (see Lemma~\ref{l:unstisweakbicat}).
\end{rem}
\begin{notn}\label{n:counit}
When $\CC$ is fibrant and $\vphi: \fC^{\sca}(\rN^{\sca}(\CC)) \x{\simeq}{\lrar} \CC$ is the counit map we will omit $\vphi$ from the notation and denote $\St^{\sca}_{\vphi}$ and $\Un^{\sca}_{\vphi}$ simply by $\St^{\sca}$ and $\Un^{\sca}$. We will employ the same convention for the variants of Notation~\ref{n:variants}.
\end{notn}

The scaled unstraightening of a diagram of (ordinary) categories indexed by a (strict) 2-category can be understood in more concrete terms, using the 2-categorical \textbf{Grothendieck construction} (see, e.g., ~\cite{Buc14}). Explicitly, given a strict 2-functor $\F: \CC \lrar \Cat_1$, its Grothendieck construction $\int_\CC\F$ is the $2$-category whose
\begin{itemize}
\item objects are pairs $(A, X)$ with $A\in \CC$ and $X\in \F(A)$.
\item 1-morphisms $(A, X)\lrar (B, Y)$ are pairs $(f, \vphi)$, with $f: A\lrar B$ a morphism in $\CC$ and $\vphi: f_!X\lrar Y$ a morphism in $\F(B)$. Here $f_!: \F(A) \lrar \F(B)$ is the functor associated to $f$.
\item given two 1-morphisms $(f,\vphi)$ and $(g,\psi)$ from $(A,X)$ to $(B,Y)$, a $2$-morphism $(f,\vphi)\Rightarrow (g,\psi)$ is a 2-morphism $\sigma: f\Rightarrow g$ in $\CC$ such that $\vphi = \psi\circ \sigma_!(X): f_!X\lrar Y$, where $\sig_!: f_! \Rightarrow g_!$ is the natural transformation associated to $\sig$.
\end{itemize}
We then have the following result, whose proof will be deferred to \S\ref{s:discrete}:

\begin{pro}\label{p:grothendieck}
Let $\CC$ be a 2-category and let $\F: \CC\lrar \Cat_1$ be a 2-functor. Let $\rN^+\F: \CC_{\rN^+} \lrar \Set^+_\Del$ be the $\Set^+_\Del$-enriched functor given by $A \mapsto \rN^+(\F(A))$. Then there is a natural map of scaled coCartesian fibrations over $\rN_2(\CC)$
$$
\Theta_{\CC}(\F): \rN_2\left(\int_{\CC}\F\right) \lrar \wtl{\Un}^{\sca}\big(\rN^+\F\big)
$$
which is a bicategorical equivalence of scaled simplicial sets.
\end{pro}

\subsection{Marked left fibrations}\label{s:marked-left}
Any $(\infty, 2)$-functor $(S, T_S)\lrar \cS$ with values in spaces can be considered as a functor with values in $\infty$-categories. Under unstraightening, such functors correspond to left fibrations over $S$. For technical reasons (see \S\ref{s:main}), it will be convenient to use the following marked variant of a left fibration:
\begin{define}\label{d:marked-left-1}
Let $p: (X,E_X) \lrar (S,T_S)$ be a map of marked simplicial sets. We will say that $p$ is a \textbf{marked left fibration} if it satisfies the following properties:
\begin{enumerate}
\item
The map $p: X \lrar S$ is a left fibration of simplicial sets.
\item
An edge of $X$ is marked if and only if its image in $S$ is marked.
\end{enumerate}
\end{define}
\begin{warning}\label{w:markedleft}
A marked simplicial set $(S,E_S)$ can be considered as representing an $\infty$-category via the categorical model structure on $\Set^+_\Del$. However, marked left fibrations in the above sense do \textbf{not} correspond to functors of the form $(S, E_S)\lrar \cS$. Instead, they corresponds to functors of the form $S \lrar \cS$, see Lemma~\ref{l:lcc-left} below.
\end{warning}
\begin{rem}\label{r:left-is-loc}
Let $S$ be a simplicial set equipped with a marking $E_S$ and a scaling $T_S$, and set $\Beta = (S,E_S,T_S)$ as above. Then any marked left fibration $p: (X,E_X) \lrar (S,E_S)$ constitutes a $\Beta$-fibered object of $(\Set^+_\Del)_{/(S,E_S)}$ (see \S\ref{s:scaledsset}): indeed, any left fibration is a coCartesian fibration and any edge in $X$ is $p$-coCartesian.
\end{rem}
Now let $(S,E_S)$ be a marked simplicial set. We will say that a map $f: (Y, E_Y)\lrar (X, E_X)$ in $(\Set^+_{\Del})_{/(S,E_S)}$ is a \textbf{marked covariant weak equivalence} if $Y\lrar X$ is a covariant weak equivalence in $(\Set_{\Del})_{/S}$. We will say that $f$ is a \textbf{marked covariant fibration} if $f: Y \lrar X$ is a covariant fibration in $(\Set_{\Del})_{/S}$ and $E_Y=f^{-1}(E_X)$.
\begin{lem}\label{l:lcc-left}
There exists a model structure on $(\Set^+_\Del)_{/(S,E_S)}$ whose weak equivalences are the marked covariant weak equivalence, whose fibrations are the marked covariant fibrations and whose cofibrations are the monomorphisms. Furthermore, the adjoint pair 
\begin{equation}\label{e:transferadj}
(-)^\flat: (\Set_{\Del})_{/S} \adj (\Set^+_{\Del})_{/(S,E_S)} : \forget
\end{equation}
whose right adjoint forgets the marking and left adjoint introduces trivial marking, yields a Quillen equivalence between this model structure and the covariant model structure on $(\Set_{\Del})_{/S}$.
\end{lem}
\begin{proof}
It is straightforward to verify that these classes of maps form a model structure: indeed, the lifting and factorization axioms all follow from the corresponding axioms for the covariant model structure on $(\Set_{\Del})_{/S}$. 
Furthermore, the adjunction \eqref{e:transferadj} is a Quillen pair by construction in which the right adjoint preserves and detects weak equivalences. To see that it is a Quillen equivalence, it therefore suffices to verify that the (underived) unit map is a weak equivalence. But this unit map is an isomorphism since the underlying simplicial set of $X^\flat$ is simply $X$.
\end{proof}

\begin{define}\label{d:lcc-left}
We will refer to the model category of Lemma~\ref{l:lcc-left} as the \textbf{marked covariant model structure} and denote it by $(\Set^+_{\Del})^{\cov}_{/(S,E_S)}$. 
\end{define}
\begin{rem}\label{r:fibrations}
Let $p: (X,E_X) \lrar (S,E_S)$ and $q:(Y,E_Y) \lrar (S,E_S)$ be two marked left fibrations. By~\cite[Corollary 2.2.3.14]{Lur09} a map $f:(X,E_X) \lrar (Y,E_Y)$ is a fibration in the marked covariant model structure if and only if it is a marked left fibration. In particular, the fibrant objects of $(\Set^+_{\Del})^{\cov}_{/(S,E_S)}$ are precisely the marked left fibrations. 
\end{rem}

\begin{rem}\label{r:local}
Let $\Beta:=(S,E_S,T_S)$ be a simplicial set $S$ equipped with a marking $E_S$ and a scaling $T_S$. By Remark~\ref{r:left-is-loc} and Remark~\ref{r:fibrations} every fibrant object of $(\Set^+_{\Del})^{\cov}_{/(S,E_S)}$ is also fibrant when considered as an object of $(\Set^+_\Del)_{/\Beta}$. Since these model structures have the same class of cofibrations we may deduce that the marked covariant model structure is a simplicial left Bousfield localization of the $\Beta$-fibered model structure. In this case, it is not hard to exhibit an explicit set $\bS$ of maps which induce the desired left Bousfield localization. Indeed, take $\bS$ to be the set of left horn inclusions $\Lam^n_i \subseteq \Del^n$ for every $0 \leq i < n$ and every $\Del^n \lrar S$, together with the maps $(\Del^1)^\flat\lrar (\Del^1)^\sharp$ for every marked edge of $S$. Then all the maps in $\bS$ are marked covariant weak equivalences and hence every marked left fibration is $\bS$-local. On the other hand, if a $\Beta$-fibered object is $\bS$-local, then certainly it has the right lifting property with respect to $\bS$, which consists of cofibrations. This means that it is a marked left fibration.
\end{rem}

\begin{rem}\label{r:local-2}
If $(S,E_S)$ is a fibrant marked simplicial set, then Lemma \ref{l:joyal} below asserts that the slice model structure on $(\Set^+_\Del)_{/(S, E_S)}$ arises from a certain categorical pattern $\Beta$. Remark~\ref{r:local} now shows that the marked covariant model structure is a simplicial left Bousfield localization of the slice model structure with respect to the set of maps $\bS$. In particular, any marked left fibration over a fibrant marked simplicial set is a categorical fibration of marked simplicial sets.
\end{rem}

\subsection{Stabilization and tangent bundles}\label{s:recall}
In this section we will recall the notion of \textbf{stabilization} and the closely related construction of \textbf{tangent bundles}. Recall that a model category is called \textbf{stable} if its homotopy category is pointed and the loop-suspension adjunction $\Sig: \Ho(\M) \adj \Ho(\M): \Om$ is an equivalence (equivalently, the underlying $\infty$-category of $\M$ is stable in the sense of~\cite[\S 1]{Lur14}). Given a model category $\M$ one may look for a universal stable model category $\M'$ related to $\M$ via a Quillen adjunction $\M \adj \M'$. When $\M$ is combinatorial the underlying $\infty$-category $\M_{\infty}$ is presentable, in which case a universal stable presentable $\infty$-category $\Sp(\M_{\infty})$ admitting a left functor from $\M_{\infty}$ indeed exists. When $\M$ is furthermore pointed and left proper there are various ways to realize $\Sp(\M_{\infty})$ as a certain model category of spectrum objects in $\M$ (see~\cite{Hov}). One such construction, which is particularly convenient for the applications in the current paper, was developed in \cite{part0}, based on ideas of Heller (\cite{Hel}) and Lurie (\cite{Lur06}): for a pointed, left proper combinatorial model category $\M$ we consider the left Bousfield localization $\Sp(\M)$ of the category of $(\NN\times \NN)$-diagrams in $\M$ whose fibrant objects are those diagrams $X: \NN\times\NN\lrar \M$ for which $X_{m, n}$ is weakly contractible when $m\neq n$ and for which each diagonal square
\begin{equation}\label{e:square_nn}
\vcenter{\xymatrix@R=1.5pc@C=1.5pc{
X_{n, n}\ar[r]\ar[d] & X_{n, n+1}\ar[d]\\
X_{n+1, n}\ar[r] & X_{n+1, n+1}
}}
\end{equation}
is homotopy Cartesian. The diagonal squares then determine equivalences $X_{n,n} \x{\simeq}{\lrar} \Om X_{n+1,n+1}$, and so we may view fibrant objects of $\Sp(\M)$ as $\Om$-spectrum objects. 
There is a canonical Quillen adjunction
$$ \Sig^{\infty}: \M \adj \Sp(\M): \Om^{\infty}, $$
where $\Om^{\infty}$ sends an $(\NN \times \NN)$-diagram $X_{\bullet\bullet}$ to $X_{0,0}$ and $\Sig^{\infty}$ sends an object $X$ to the constant $(\NN \times \NN)$-diagram with value $X$. 

When $\M$ is not pointed, its stabilization is the model category $\Sp(\M_\ast)$ of spectrum objects in its pointification $\M_\ast=\M_{\ast/}$. We then denote by $\Sig^{\infty}_+: \M \lrar \Sp(\M_{\ast})$ the composite left Quillen functor 
$$
\M \x{(-) \coprod \ast}{\lrar} \M_{\ast} \x{\Sig^{\infty}}{\lrar} \Sp(\M_{\ast}).
$$
Given an object $A \in \M$, we will denote by $\M_{A//A}:=\left(\M_{/A}\right)_\ast$ the category of pointed objects in the over-category $\M_{/A}$, endowed with its induced model structure. The stabilization of $\M_{/A}$ is the model category of spectrum objects in $\M_{A//A}$, which we will denote (as in \cite{part0}) by 
$$ \T_A\M \x{\df}{=} \Sp(\M_{A//A}) $$
and refer to as the \textbf{tangent model category} to $\M$ at $A$. We will refer to fibrant objects in $\T_A\M$ as \textbf{parameterized spectrum objects} over $A$. By~\cite[Lemma 3.20 and Proposition 3.21]{part0}, the $\infty$-category associated to the model category $\T_A\M$ is equivalent to the \textbf{tangent $\infty$-category} $\T_A(\M_\infty)$ defined in~\cite[\S 7.3]{Lur14}, at least if $A$ is fibrant or if $\M$ is right proper (so that $\M_{/A}$ models the slice $\infty$-category $(\M_\infty)_{/A}$).
 
\begin{define}[{(cf.~\cite[\S 7.3]{Lur14})}]\label{d:cotangent}
Let $\M$ be a left proper combinatorial model category. We will denote by
$$ L_A = \LL\Sig^{\infty}_+(A) \in \T_A\M $$
the derived suspension spectrum of $A$ and will refer to $L_A$ as the \textbf{cotangent complex} of $A$. The \textbf{relative cotangent complex} $L_{B/A}$ of a map $f:A \lrar B$ is the homotopy cofiber in $\T_B\M$
$$ \LL\Sig^{\infty}_+(f) \lrar L_B\lrar L_{B/A}. 
$$ 
\end{define}
As in~\cite[\S 2.2]{part2}, we will consider the following form of Quillen cohomology, which is based on the cotangent complex above:
\begin{defn}\label{d:Quillen-coh}
Let $\M$ be a left proper combinatorial model category and let $f:A \lrar X$ be a map in $\M$ with fibrant codomain. For $n \in \ZZ$ we define the \textbf{relative $n$'th Quillen cohomology} group of $X$ with coefficients in a parameterized spectrum object $M\in \T_X\M$ by the formula
$$ \rH_Q^n(X,A;M):=\pi_0 \Map^{\der}(L_{X/A},\Sig^n M) .$$  
where $L_{X/A}$ is the relative cotangent complex of the map $f$ (see Definition~\ref{d:cotangent}). When $f: \emptyset \lrar X$ is the initial map we also denote $\rH_Q^n(X;M) := \rH_Q^n(X,\emptyset;M)$ and refer to it simply as the Quillen cohomology $X$.
\end{defn} 
If $\C$ is a presentable $\infty$-category, then the functor $\C \lrar \Cat_{\infty}$ sending $A \in \C$ to $\T_A\C$ classifies a (co)Cartesian fibration $\T\C\lrar \C$ known as the \textbf{tangent bundle} of $\C$. 
A simple variation of the above model-categorical constructions can be used to give a model for the tangent bundle of a model category $\M$ as well, which furthermore enjoys the type of favorable formal properties one might expect (see \cite{part0}). 
More precisely, if $(\NN \times \NN)_\ast$ denotes the category obtained from $\NN \times \NN$ by \textbf{freely adding a zero object} and $\M$ is a left proper combinatorial model category, then one can define $\T\M$ as a left Bousfield localization of the Reedy model category $\M^{(\NN \times \NN)_{\ast}}_{\Reedy}$, where a Reedy fibrant object $X: (\NN \times \NN)_\ast\lrar \M$ is fibrant in $\T\M$ if and only if the map $X_{n,m} \lrar X_{\ast}$ is a weak equivalence for every $n \neq m$ and the square~\eqref{e:square_nn} is homotopy Cartesian for every $n \geq 0$. 

The projection $\ev_*: \T\M \lrar \M$ is then a (co)Cartesian fibration which exhibits $\T\M$ as a \textbf{relative model category} over $\M$ in the sense of~\cite{HP}: $\T\M$ has relative limits and colimits over $\M$ and factorization (resp.\ lifting) problems in $\T\M$ with a solution in $\M$ admit a compatible solution in $\T\M$. In particular, it follows that the projection is a left and right Quillen functor and that each fiber is a model category. When $A\in \M$ is a fibrant object, the fiber $(\T\M)_A$ can be identified with the tangent model category $\T_A\M$. Furthermore, the underlying map of $\infty$-categories $\T\M_\infty \lrar \M_{\infty}$ exhibits $\T\M_{\infty}$ as the tangent bundle of $\M_\infty$ (see~\cite[Proposition 3.25]{part0}). We refer the reader to~\cite{part0} for further details.

\section{The twisted $2$-cell $\infty$-category}\label{s:twisted-2-cells}
In this section we will introduce the notion of the \textbf{twisted $2$-cell $\infty$-category}, which plays a central role in this paper. This $\infty$-category will actually be derived from a suitable $\infty$-bicategory, which we will refer to as the \textbf{twisted $2$-cell $\infty$-bicategory}. To begin, let us recall the $(\infty,1)$-categorical counterpart of our construction, namely the twisted arrow category.

Let $F:\Del\lrar \Del$ be the functor given by $[n]\mapsto [n]^{\op} \ast [n]$, where $\ast$ denotes concatenation of finite ordered sets.  
When $\C\in \Set_{\Del}$ is an $\infty$-category, the simplicial set $\Tw(\C) := F^*\C$ is also an $\infty$-category, which is known as the \textbf{twisted arrow category} of $\C$.  
By definition the objects of $\Tw(\C)$ are the arrows of $\C$ and a morphism in $\Tw(\C)$ from $f: X \lrar Y$ to $g: Z \lrar W$ is given by a diagram in $\C$ of the form
\begin{equation}\label{e:arrowintw}\vcenter{\xymatrix@R=1.5pc@C=1.5pc{X\ar^{f}[r]& Y\ar[d] \\ Z\ar[u]\ar^{g}[r] & W.
}}\end{equation}
Note that the above convention regarding the direction of arrows is opposite to that of~\cite[\S 5.2.1]{Lur14}). When $\C$ is an ordinary category $\Tw(\C)$ is an ordinary category as well, and was studied in a variety of contexts. 
In fact, in this case one can write $\Tw(\C)$ using the classical Grothendieck construction as 
$$
\Tw(\C):=\int_{(x,y)\in\C^{\op}\times \C}\Map_\C(x,y).
$$
This property has an analogue in the $\infty$-categorical setting: by \cite[\S 5.2.1]{Lur14}, restriction along the inclusions $[n]\hrar [n]^{\op}\ast [n]$ and $[n]^{\op}\hrar [n]^{\op} \ast [n]$ induces a left fibration of $\infty$-categories $\Tw(\C) \lrar \C^{\op} \times \C$, which classifies the mapping space functor $\Map_{\C}: \C^{\op} \times \C \lrar \cS$ (where $\cS$ denotes the $\infty$-category of spaces). In particular, it follows that $\Tw(-)$ preserves equivalences between $\infty$-categories.

\begin{rem}\label{r:kan}
If $\C$ is a Kan complex then $\Tw(\C)$ is a Kan complex as well and the codomain projection $\Tw(\C) \lrar \C$ is a trivial Kan fibration.
\end{rem}

It will be useful to have a marked variant  $\Tw^+:\Set_{\Del}^+\lrar \Set_{\Del}^+$ of the twisted arrow category. Let $\C$ be a marked simplicial set. We define $\Tw^+(\C)$ to be the marked simplicial set whose underlying simplicial set is $\Tw(\C)$ and where a $1$-simplex \eqref{e:arrowintw} 
is marked if both $Z\lrar X$ and $Y\lrar W$ are marked in $\C$. When $\C$ is a fibrant marked simplicial set the map $\Tw^+(\C) \lrar \C^{\op} \times \C$ is a marked left fibration and in particular $\Tw^+(\C)$ is fibrant in $\Set^+_\Del$.

Let us now introduce an analogue of the above construction for $(\infty,2)$-categories. Let $\CC\in \Cat_{\Set_{\Del}^+}$ be a fibrant marked-simplicial category. We denote by $\CC_{\Tw}$ the marked-simplicial category with the same objects and mapping objects defined by $\CC_{\Tw}(x,y)=\Tw^+(\CC(x,y))$. 

\begin{defn}\label{d:twisted-2}
Let $\CC$ be a fibrant marked-simplicial category and let $\Map_{\Tw}: \CC_{\Tw}^{\op} \times \CC_{\Tw} \lrar \Set^+_\Del$ be the mapping space functor.
We define the \textbf{twisted $2$-cell $\infty$-bicategory} as
$$ \Tw_2(\CC) := \wtl{\Un}^{\sca}(\Map_{\Tw}), $$
where $\wtl{\Un}^{\sca}(-)$ is as in Notation~\ref{n:variants}.
We will also denote by $\uline{\Tw}_2(\CC) \in (\Set_\Del)^{\Joy}$ the underlying unscaled simplicial set of $\Tw_2(\CC)$. We will refer to any Joyal fibrant model of $\uline{\Tw}_2(\CC)$ as the \textbf{twisted $2$-cell $\infty$-category}. 
\end{defn}

\begin{rem}
By Lemma~\ref{l:unstisweakbicat} the scaled simplicial set $\Tw_2(\CC)$ is a weak $\infty$-bicategory. In fact, by a recent result of~\cite{higher} any weak $\infty$-bicategory is fibrant, i.e., an $\infty$-bicategory. In particular, $\Tw_2(\CC)$ is an $\infty$-bicategory.
\end{rem}

\begin{warning}
The simplicial set $\uline{\Tw}_2(\CC)$ is \textbf{not} Joyal fibrant in general. \end{warning}

\begin{rem}
As explained in \S\ref{s:scaledsset} we may consider $\uline{\Tw}_2(\CC) \simeq |\Tw_2(\CC)|_1$ as a model for the $\infty$-category freely generated from the $\infty$-bicategory $\Tw_2(\CC)$. This can be used to give a more explicit description of $\uline{\Tw}_2(\CC)$ in terms of $\Tw_2(\CC)$: indeed, the objects of $\uline{\Tw}_2(\CC)$ can be taken to be the same as the objects of $\Tw_2(\CC)$, and for each pair of objects $x,y$ the mapping space from $x$ to $y$ in any Joyal fibrant model for $\uline{\Tw}_2(\CC)$ is the classifying space of the $\infty$-category $\Map_{\Tw_2(\CC)}(x,y)$ (see Remark \ref{r:freeoo1}).
\end{rem}

\begin{example}\label{e:1-cat}
Let $\CC$ be a simplicial category in which every mapping object is a Kan complex and let $\CC_{\Tw}$ be the simplicial category obtained by applying the functor $\Tw$ to every mapping object. Let $\CC'$ be the marked-simplicial category obtained from $\CC$ by applying the functor $(-)^{\sharp}$ to all mapping objects and let $\CC'_{\Tw}$ be as above. Then all the triangles in $\rN^{\sca}(\CC'_{\Tw})$ and $\Tw_2(\CC')$ are thin and the underlying map of simplicial sets $\uline{\Tw}_2(\CC') \lrar \rN(\CC^{\op}) \times \rN(\CC)$ reduces to the left fibration classifying the Kan complex valued functor $(x,y) \mapsto \Tw(\Map_\C(x,y))$ (see Remark~\ref{r:kan}). On the other hand, the map $\CC_{\Tw} \lrar \CC$ induced by the codomain projection is a trivial fibration of simplicial categories by Remark~\ref{r:kan}, so we obtain a pair of equivalences $\uline{\Tw}_2(\CC') \simeq \Tw(\CC'_{\Tw}) \simeq \Tw(\CC)$. We may summarize the above discussion as follows: for an $(\infty,1)$-category the twisted $2$-cell $\infty$-bicategory $\Tw_2(\CC)$ is actually an $(\infty,1)$-category which is equivalent to the corresponding twisted \textbf{arrow} category. Similarly, if $\rN(\CC)$ is an $\infty$-groupoid then the twisted $2$-cell $\infty$-category of $\CC$ is equivalent to $\rN(\CC)$ itself. 
\end{example}

\begin{rem}\label{r:product}
If $\F: \DD \lrar \Set^+_\Del$ and $\G: \EE \lrar \Set^+_\Del$ are two $\Set^+_\Del$-enriched functors, then 
$$ \Un^{\sca}(p_\DD^*\F \times p_{\EE}^*\G) \cong (p_{\rN^{\sca}\DD}^*\Un^{\sca}(\F)) \times_{\rN^{\sca}(\DD) \times \rN^{\sca}(\EE)} (p_{\rN^{\sca}\EE}^*\Un^{\sca}(\G)) \cong \Un^{\sca}(\F) \times \Un^{\sca}(\G) $$
where $p_\DD: \DD \times \EE \lrar \DD$ and $p_{\EE}: \DD \times \EE \lrar \EE$ are the two projections and similarly for $p_{\rN^{\sca}\DD}$ and $p_{\rN^{\sca}\EE}$. This is because $\Un^{\sca}$ is right Quillen and is compatible with base change. Consequently, if  
$\CC,\CC'$ are two marked simplicial categories then
$$ 
\Tw_2(\CC \times \CC') \simeq \Tw_2(\CC) \times \Tw_2(\CC') \quad\quad \text{and} \quad\quad  \uline{\Tw}_2(\CC \times \CC') \simeq \uline{\Tw}_2(\CC) \times \uline{\Tw}_2(\CC') .$$
\end{rem}

When $\CC$ is a (strict) 2-category, Proposition \ref{p:grothendieck} shows that its twisted 2-cell bicategory is a strict 2-category as well:
\begin{pro}\label{p:2-cat}
For a 2-category $\CC$, there is a natural equivalence of $\infty$-bicategories
$$ \Tw_2(\CC_{\rN^+}) \simeq \rN_2\left(\int_{\CC_{\Tw}^{\op} \times \CC_{\Tw}} \Map_{\CC_{\Tw}}(-,-)\right). $$ 
\end{pro}

\subsection{Examples}\label{s:examples}

Let $(A, \cdot)$ be an abelian monoid (in sets) and let $\BB^2A$ be the strict 2-category with a single object, a single 1-morphism and $A$ as 2-morphisms. Then the strict 2-category $(\BB^2A)_{\Tw}$ has a single object whose endomorphism category is the category 
$\Tw(\BB A)=A\backslash A /A$ whose objects are elements $a\in A$ and whose morphisms are given by $b_\pm = (b_-, b_+): a\lrar b_-ab_+$ for $b_-,b_+ \in A$. The composition is given by $b_{\pm} \circ b_{\pm}' = (bb')_{\pm}=(b_-b'_-,b_+b_+')$ and 
the multiplication in $A$ makes this a monoidal category. Using Proposition~\ref{p:2-cat} we may identify the twisted $2$-cell $\infty$-bicategory $\Tw_2(\BB^2A)$ as the strict 2-category with
\begin{enumerate}
 \item[(0)] objects $a\in A$.
 \item[(1)] morphisms $(b, c, d_\pm): a\lrar d_-(bac)d_+$, where $b,c \in A\backslash A/A$ 
and $d_\pm:bac \lrar d_-(bac)d_+$ is a morphism in $A\backslash A/A$. 
 \item[(2)] 2-morphisms $(e_\pm, f_\pm): (b, c, d_\pm)\lrar (e_-be_+, f_-cf_+, d'_\pm)$, where $e_\pm: b \lrar e_-be_+$ and $f_\pm: c \lrar f_-cf_+$ are morphisms in $A \backslash A / A$ such that
$$
 d_\pm = e_\pm d'_\pm f_\pm.
$$
We may suggestively depict a 2-morphism $(e_\pm, f_\pm): (b, c, d_\pm)\lrar (e_-be_+, f_-cf_+, d'_\pm) =: (b',c',d'_\pm)$ as a commuting diagram
$$\xymatrix{
\ar[r]^{d_+}\ar[d]_{e_+} & \ar[r]^{c} & \ar[d]^{f_-} \ar[r]^{d_-} & \ar[r]^{b} &\ar[d]^{e_+}\\
\ar[r]_{d'_+} & \ar[r]_{c'}\ar[u]^{f_+} & \ar[r]_{d'_-} & \ar[r]_{b'}\ar[u]_{e_-} &
}$$
\end{enumerate}
The twisted 2-cell $\infty$-category $\uline{\Tw}_2(\BB^2A)$ of $\BB^2 A$ is then the $\infty$-category freely generated by the above $2$-category $\Tw_2(\BB^2 A)$, i.e., its objects are the elements $a \in A$ and 
$$ \Map_{\uline{\Tw}_2(\BB^2 A)}(a,a') \simeq |\Map_{\Tw_2(\BB^2 A)}(a,a')| $$ 
is the classifying space of the mapping category from $a$ to $a'$ described above (see Remark~\ref{r:freeoo1}). To obtain a somewhat simpler description of $\Tw_2(\BB^2 A)$, let us consider the following construction:

\begin{cons}\label{con:D}
Let $\E$ be the category whose objects are pairs $(b,x) \in A^2$ and morphisms $(b, x)\lrar (b', x')$ are tuples $e_\pm\in A^2$ such that $b'=e_-be_+$ and $x=e_-x'e_+$. The product in $A$ endows $\E$ with the structure of a monoidal category. Let $\BB\E$ be the $2$-category with one object whose endomorphism category is $\E$ and consider the projection
$$ \pi: \DD_{A} := \int_{\BB\E}\F_A \lrar \BB\E $$
where $\F_A:\BB\E \lrar \Set \subseteq \Cat$ is the $2$-functor which sends the unique object of $\BB\E$ to the underlying set of $A$ and the morphism $(b,x)$ to the map $m_{bx}: A \lrar A$ sending $a \mapsto bax$. Unwinding the definition of the Grothendieck construction (see \S\ref{s:straightening}) we see that the $2$-category $\DD_{A}$ admits the following description:
the objects of $\DD_{A}$ are the elements $a\in A$ and the mapping category $\Map_{\DD_{A}}(a, a')$ has
\begin{enumerate}
\item[(0)] objects given by tuples $(b, x)\in A^2$ such that $bax=a'$.
\item[(1)] morphisms $(b, x)\lrar (b', x')$ given by tuples $e_\pm\in A^2$ such that $b'=e_-be_+$ and $x=e_-x'e_+$.
\end{enumerate}
All compositions are given by multiplication in $A$. We will use a commuting diagram
$$\xymatrix{
\ar[r]^x\ar[d]_{e_+} & \ar[r]^b & \ar[d]^{e_+}\\
\ar[r]_{x'} & \ar[u]_{e_-} \ar[r]_{b'} &
}$$
to depict a morphism $e_\pm:(b, x)\lrar (b', x')$ in $\Map_{\DD_{A}}(a, a')$.
\end{cons}
Let $\pi: \Tw_2(\BB^2 A)\lrar \DD_{A}$ be the $2$-functor which is the identity on objects and is given on mapping categories by the functors 
$$\xymatrix{
\pi_{a,a'}: \Map_{\Tw_2(\BB^2A)}(a, a')\ar[r] & \Map_{\DD_{A}}(a, a'); \hspace{4pt} (b, c, d_\pm)\ar@{|->}[r] & (b, d_-cd_+)
}$$
whose value on an arrow $(e_\pm, f_\pm): (b, c, d_\pm)\lrar (b', c', d'_\pm)$ is $e_\pm: (b, d_-cd_+)\lrar (b', d'_-c'd'_+)$.  We can depict the behavior on morphisms diagrammatically as
$$\xymatrix{
\ar[r]^{d_+}\ar[d]_{e_+} & \ar[r]^{c} & \ar[d]^{f_-} \ar[r]^{d_-} & \ar[r]^{b} &\ar[d]^{e_+}="s" & & \ar[r]^{d_-cd_+}\ar[d]_{e_+}="t" & \ar[r]^b & \ar[d]^{e_+}\\
\ar[r]_{d'_+} & \ar[r]_{c'}\ar[u]^{f_+} & \ar[r]_{d'_-} & \ar[r]_{b'}\ar[u]_{e_-} & & & \ar[r]_{d'_-c'd'_+} & \ar[u]_{e_-} \ar[r]_{b'} & \ar@{|->}"s"+<10pt, 0pt>;"t"-<10pt, 0pt>^{\pi_{a, a'}}
}$$
We claim that each $\pi_{a, a'}$ is cofinal. Indeed, observe that the functor $\pi_{a, a'}$ is a Cartesian fibration: given a tuple $(b', c', d'_\pm)$ and a map $e_\pm: (b, x)\lrar (b', d'_-c'd'_+)$, a Cartesian lift is given by the following picture:
$$\xymatrix{
\ar[r]^{d'_+e_+}\ar[d]_{e_+} & \ar[r]^{c'} & \ar[d]^{1} \ar[r]^{e_-d'_-} & \ar[r]^{b} &\ar[d]^{e_+}="s" & & \ar[r]^{x}\ar[d]_{e_+}="t" & \ar[r]^b & \ar[d]^{e_+}\\
\ar[r]_{d'_+} & \ar[r]_{c'}\ar[u]^{1} & \ar[r]_{d'_-} & \ar[r]_{b'}\ar[u]_{e_-} & & & \ar[r]_{d'_-c'd'_+} & \ar[u]_{e_-} \ar[r]_{b'} &.\ar@{|->}"s"+<10pt, 0pt>;"t"-<10pt, 0pt>^{\pi_{a, a'}}
}$$
It therefore suffices to show that the fiber of $\pi_{a, a'}$ over each $(b, x)\in \Map_{\DD_{A}}(a, a')$ has a weakly contractible classifying space. Unraveling the definitions, the fiber over $(b, x)$ is the category with 
\begin{enumerate}
\item[(0)] objects given by tuples $(c, d_\pm)\in A^{\times 3}$ such that $d_-cd_+=x$
\item[(1)] morphisms given by $f_\pm: (c, d_\pm)\lrar (c', d'_\pm)$ such that $d_+=f_+d'_+, c'= f_-cf_+$ and $d_-=d'_-f_-$.
\end{enumerate}
This category has a terminal object, given by $(c, d_-, d_+) = (x, 1, 1)$. We conclude that the fibers of $\pi_{a, a'}$ are weakly contractible, so that $\pi_{a, a'}$ in indeed cofinal.

We may now conclude that the twisted arrow category $\uline{\Tw}_2(\BB^2A)$ is equivalent to the $\infty$-category freely generated from the $2$-category $\DD_{A}$, i.e., the $\infty$-category whose objects are the elements $a \in A$ and whose mapping spaces are the classifying spaces $|\Map_{\DD_{A}}(a,a')|$ of the mapping categories of $\DD_{A}$. We note that the functor $F_A: \BB\E \lrar \Set$ used to construct $\DD_{A}$ clearly factors through the $\infty$-category $|\BB\E|_1 = \BB
|\E|$ freely generated from $\BB\E$, so that the twisted $2$-cell category admits a left fibration
$$ \uline{\Tw}_2(\BB^2A) \simeq |\DD_{A}|_1 \x{|\pi|_1}{\lrar} \BB|\E| $$
which is classified by the induced functor $\ovl{\F}_A:\BB|\E|\lrar \Set$.

\begin{rem}\label{r:env}
The monoid in spaces $|\E|$ and the functor $\ovl{\F}_A$ both admit conceptual descriptions. Indeed, the nerve of the category $\E$ is naturally isomorphic to the two-sided bar construction $\textrm{Bar}_{A^{\op}\times A}(A,A)$ which computes the Hochschild homology space $\int_{S^1}A = A \otimes_{A^{op} \times A} A$ of $A$ (see also~\cite[\S 5.5.3]{Lur14}). Since $A$ is commutative, we can consider it as an $\EE_2$-monoid in spaces. In this case, $\int_{S^1}A$ inherits a monoid structure and by~\cite[Theorem 3.16]{Fr13} we may identify $|\E| \simeq \int_{S^1}A$ with the \textbf{enveloping monoid} $\Env_{\EE_2}(A)$ of $A$. From this point of view the functor $\ovl{\F}_A: \BB|\E| = \BB\Env_{\EE_2}(A) \lrar \Set$ admits a very simple description: it is the functor which encodes the canonical action of $\Env_{\EE_2}(A)$ on $A$.
\end{rem}
 
\begin{example}
Suppose that $A$ is an abelian group. Then for every $a,a'\in A$, an element in $\Map_{\DD_{A}}(a, a')$ is determined uniquely by an (arbitrary) element $b\in A$. It follows that $\Map_{\DD_{A}}(a,a') = A \backslash A / A \simeq \BB A$ for every $a,a'\in A$ and hence $\Tw_2(\BB^2 A) \simeq \uline{\Tw}_2(\BB^2 A) \simeq \BB^2 A$ (see also Example~\ref{e:1-cat}).
\end{example}

\begin{example}\label{e:natural}
Consider the case where $(A, \cdot) = (\NN, +)$. We claim that the twisted 2-cell category of $\BB^2\NN$ can be identified with the $\infty$-category whose objects are elements $n\in\NN$ and whose mapping spaces are
$$
\Map_{\uline{\Tw}_2(\BB^2\NN)}(m, n) = \left\{\begin{array}{cc} 
                                      \emptyset & m>n\\
                                      \ast & m=n\\
                                      S^1=B\ZZ & m<n
                                      \end{array}\right.
$$
where all compositions arise from the multiplication of $S^1$. To see this, let $\DD_{\NN}$ be as the $2$-category constructed above for the monoid $A = \NN$, so that we can identify $\uline{\Tw}_2(\BB^2\NN)$ with the $\infty$-category obtained by replacing the mapping categories of $\DD_{\NN}$ by their classifying spaces. Now the mapping category $\Map_{\DD_{\NN}}(m, n)$ has
\begin{enumerate}
\item[(0)] objects $b\in\NN$ with $0\leq b\leq n-m$ (encoding the pair $(b, n-m-b)$ in Construction \ref{con:D}).
\item[(1)] morphism $b\lrar b'$ given by $e\in\NN$ with $0\leq e\leq b'-b$ (encoding the pair $e_{\pm}=(e,b'-b-e)$ in Construction \ref{con:D}), with composition given by addition.
\end{enumerate}
It is then clear that $\Map_{\DD_{\NN}}(m,n)$ is empty when $m > n$ and a point when $m=n$. Now consider the functor
$$\xymatrix{
\F :\Map_{\DD}(m, n)\ar[r] & \ZZ-\Torsors; \hspace{4pt} \Big(b\stackrel{e}{\lrar} b'\Big) \ar@{|->}[r] & \Big(\ZZ\stackrel{+e}{\lrar} \ZZ\Big).
}$$
Then $\F$ induces a map on classifying spaces $|\F|:|\Map_{\DD}(m, n)| \lrar |\ZZ-\Torsors| \simeq S^1$. We claim that $|\F|$ is a weak equivalence as soon as $m < n$. To see this, consider the corresponding principal $\ZZ$-bundle
$$\xymatrix{
\C:=\int_{\Map_{\DD}(m, n)}\F\ar[r] & \Map_{\DD}(m, n)
}$$
To show that $|\F|$ is a weak equivalence it will suffice to show that $|\C|$ is weakly contractible. Unraveling the definitions, one finds that $\C$ is the poset with
\begin{enumerate}
 \item[(0)] objects $(b, z)$ with $0\leq b\leq n-m$ and $z\in \ZZ$.
 \item[(1)] $(b, z)\leq (b', z')$ if and only if $0\leq (z'-z)\leq (b'-b)$.
\end{enumerate}
The projection $\C \lrar \Map_{\DD}(m, n)$ sends $(b, z)\leq (b', z')$ to the arrow $z'-z: b\lrar b'$. The functor
$$\xymatrix{
\C\ar@{^{(}->}[r] & (\ZZ, \leq)^{\times 2}; \hspace{4pt} (b, z)\ar@{|->}[r] & (b-z, z)
}$$
identifies $\C$ with the subposet of $\ZZ\times \ZZ$ of tuples $(p, q)$ with $0\leq p+q\leq n-m$.

Let $\C'$ be the subposet of tuples $(p, q)$ with $0\leq p+q\leq 1$, which is just an infinite zig-zag of spans
$$\xymatrix@R=-0.2pc{
\dots & & (0, 1) & & (1, 0) &  & \dots \\
& (-1, 1)\ar[ru]\ar[lu] & & (0, 0)\ar[lu]\ar[ru] & & (1, -1)\ar[lu]\ar[ru] & 
}$$
In particular, $\C'$ is weakly contractible. On the other hand, the inclusion $\C'\subseteq \C$ is coinitial: indeed, for every $(p, q)$, the comma category $\C'/(p, q)$ is a subposet of $\C'$, given by a finite composition of zig-zags
$$\xymatrix@R=-0.2pc{
& (1-q, q) & & \dots &  \\
(-q, q)\ar[ru] & & (1-q, q-1)\ar[lu]\ar[ru] & & (p, -p)\ar[lu]
}$$
which are weakly contractible posets. We may then conclude that $\C$ is weakly contractible and hence that $|\F|:|\Map_{\DD}(m,n)| \lrar S^1$ is a weak equivalence, as desired.
\end{example}

\begin{example}
Combining Example~\ref{e:natural} with Remark~\ref{r:product} we get that
the twisted 2-cell category of $\BB^2\NN^k$ can be identified with the $\infty$-category whose objects are elements $(n_1,...,n_k)\in\NN^k$ and whose mapping spaces are
$$
\Map_{\uline{\Tw}_2(\BB^2\NN)}((m_1,...,m_k), (n_1,...,n_k)) = 
\left\{\begin{array}{cc} 
                    \emptyset & \exists i | m_i>n_i \\
                    (S^1)^{\{i=1,...,k | m_i < n_i\}} & \forall i, m_i \leq n_i \\
\end{array}\right.
$$
\end{example}

\section{Quillen cohomology of $(\infty,2)$-categories}\label{s:main}
In this section we will prove the main theorem of this paper: given an $(\infty,2)$-category $\CC$ (see~\S\ref{s:recall}), we identify the $\infty$-category $\T_{\CC}\Cat_{(\infty,2)}$ of parameterized spectrum objects over $\CC$ 
with the $\infty$-category of functors $\uline{\Tw}_2(\CC) \lrar \Sp$ from the twisted $2$-cell $\infty$-category of $\CC$ to spectra. 

\begin{thm}\label{t:main}
Let $\CC$ be an $(\infty,2)$-category. Then there is a natural equivalence of $\infty$-categories
$$ \T_{\CC}(\Cat_{(\infty,2)}) \x{\simeq}{\lrar} \Fun(\uline{\Tw}_2(\CC),\Sp) $$
from the tangent $\infty$-category to $\Cat_{(\infty,2)}$ at $\CC$ to the $\infty$-category of functors from $\uline{\Tw}_2(\CC)$ to spectra.
\end{thm}
\begin{example}
Let $A$ be a discrete commutative monoid considered as an $\EE_2$-monoid in spaces and let $\Env_{\EE_2}(A)$ be its associated enveloping monoid (which is usually no longer discrete). As explained in Remark~\ref{r:env}, the twisted $2$-cell category $\uline{\Tw}_2(\BB^2A)$ is equivalent to the unstraightening of the functor $\BB \Env_{\EE_2}(A) \lrar \Set$ which encodes the canonical action of $\Env_{\EE_2}(A)$ on itself, or, alternatively, the canonical $\EE_2$-action of $A$ on itself. We may hence identify functors $\uline{\Tw}_2(\BB^2A)\lrar \Sp$ with $A$-indexed families $\{X_a\}_{a \in A}$ of spectra which are \textbf{$\Env_{\EE_2}(A)$-equivariant} with respect to the action of $\Env_{\EE_2}(A)$ on $A$ (or equivalently, which are \textbf{$A$-equivariant} with respect to the \textbf{$\EE_2$-action} of $A$ on itself).
\end{example}
Theorem~\ref{t:main} will be deduced from a more concrete statement, involving the model categorical presentations of abstract parameterized spectra discussed in \S\ref{s:recall}. We will present the $\infty$-category $\Cat_{(\infty,2)}$ by the model category $\Cat^+_\Del$ of marked-simplicial categories and the $\infty$-category $\Fun(\uline{\Tw}_2(\CC),\cS)$ in terms of the covariant model structure (see~\cite[\S 2]{Lur09}). To simplify the expressions appearing throughout this section, let us introduce the following notation:
\begin{notn}\label{n:diagrams}
Let $X$ be a marked simplicial set. We will denote by
$$
\Set_\Del^{X} := \big(\Set_\Del^+\big)_{/X}^{\cov} \qquad \text{and} \qquad \Sp^{X} := \Sp\big((\Set_\Del^X)_*\big)=\Sp\big((\Set_\Del^+)_{X//X}^{\cov}\big)
$$ 
the marked covariant model structure on marked simplicial sets (Definition \ref{d:lcc-left}) and the model category of spectrum objects therein, respectively. When $X$ is an unmarked simplicial set, we will use $\Set_\Del^X$ and $\Sp^X$ to denote $(\Set_\Del)_{/X}^{\cov}$ and the model category of spectrum objects therein.
\end{notn}
The above notation is meant to be suggestive of the fact that $\Sp^X$ is a model categorical presentation of the $\infty$-category of functors $X\lrar \Sp$, when $X$ is a simplicial set or a fibrant marked simplicial set (see also Warning \ref{w:markedleft}).
\begin{rem}\label{r:forgetmarking}
Let $X$ be a marked simplicial set and let $\uline{X}$ be the underlying simplicial set. Lemma \ref{l:lcc-left} provides Quillen equivalences $\Set_\Del^X\simeq \Set_\Del^{\uline{X}}$ and $\Sp^X\simeq \Sp^{\uline{X}}$.
\end{rem}
We will prove the following model-categorical reformulation of Theorem~\ref{t:main}:
\begin{thm}\label{t:main2}
For every fibrant marked-simplicial category $\CC$ there is a Quillen equivalence
$$ 
\xymatrix{
\F_{\CC}:\Sp^{\uline{\Tw}_2(\CC)}\ar@<1ex>[r]  &
 \T_\CC\Cat^+_\Del\ar@<1ex>[l]_-{\upvdash} : U_{\CC}\\
}
$$
which in natural in $\CC$ in the following sense: for every map $f: \CC \lrar \DD$ of fibrant marked-simplicial categories with induced map $\vphi: \uline{\Tw}_2(\CC) \lrar \uline{\Tw}_2(\DD)$ on twisted arrow $\infty$-categories there is a commuting square of right Quillen functors
\begin{equation}\label{e:square-tangent-2}
\vcenter{\xymatrix{
\T_\DD\Cat^+_\Del\ar[d]_-{f^*}\ar[r]^-{U_{\DD}} & \Sp^{\uline{\Tw}_2(\DD)} \ar[d]^-{\vphi^*} \\
\T_\CC\Cat^+_\Del\ar[r]_-{U_{\CC}} & \Sp^{\uline{\Tw}_2(\CC)}.
}}
\end{equation}
Here the functor $f^*$ takes the pullback of a parameterized spectrum object over $\DD$ along $f$ and $\vphi^*$ takes the pullback of a spectrum of left fibrations $S_{\bullet\bullet}\lrar \uline{\Tw}_2(\DD)$ along $\vphi$.
\end{thm}
Theorem~\ref{t:main} arises from a two-stage reduction: we first identify the tangent $\infty$-category $\T_{\CC}\Cat_{(\infty, 2)}$ in terms of the tangent $\infty$-categories to $\Cat_{(\infty, 1)}$, and then identify these further in terms of the tangent $\infty$-categories to $\Cat_{(\infty, 0)}\simeq \cS$. More precisely, given a fibrant marked-simplicial category $\CC$, we will produce the Quillen equivalence of Theorem \ref{t:main2} in several steps, as follows:
\begin{enumerate}
\item[(0)] By \cite[Corollary 3.1.16]{part2}, the tangent category $\T_{\CC}\Cat^+_\Del$ is Quillen equivalent to the model category of $\Set^+_\Del$-enriched lifts of the form
\begin{equation}\label{e:enriched lifts-1}
\vcenter{\xymatrix{&& \T\Set^+_\Del\ar[d] \\ \CC^{\op}\times \CC \ar@{-->}[urr]\ar[rr]_{\Map_{\CC}(-,-)} && \Set^+_\Del}}
\end{equation}
where $\T\Set^+_\Del \lrar \Set^+_\Del$ is the tangent bundle fibration of $\Set^+_\Del$.

\item[(1)] For each fibrant simplicial set $X$, the tangent category $\T_{X}\Set^+_\Del$ is Quillen equivalent to $\Sp^{\Tw^+(X)}$ by the results of \cite[\S 3.3]{part2} and Lemma \ref{l:lcc-left}. In \S\ref{s:marked-tangent}, we will describe a direct right Quillen functor $\R_X^{\Sp}: \T_{X}\Set^+_\Del\lrar \Sp^{\Tw^+(X)}$ exhibiting this equivalence and we will show that these Quillen functors assemble into a global right Quillen functor $\R^{\Sp}:\T\Set^+_\Del \lrar \int_{X}\Sp^{X}$.
\item[(2)]
In \S\ref{s:lifts} we show that postcomposition with the functor $\R^{\Sp}$ induces a Quillen equivalence between the model category of lifts as in~\eqref{e:enriched lifts-1} and the model category of enriched lifts of the form
\begin{equation}\label{e:lift-tw-1}
\vcenter{\xymatrix{&& \int_{X\in\Set^+_\Del} \Sp^X\ar[d] \\ \CC_{\Tw}^{\op}\times \CC_{\Tw}^{} \ar@{-->}[urr]\ar[rr]_-{\Map_{\CC_{\Tw}}(-,-)} && \Set^+_\Del.}}
\end{equation}
\item[(3)]
Finally, in \S\ref{s:proof-main} we identify the model category of enriched lifts as in~\eqref{e:lift-tw-1} with the stabilization of a certain model structure on marked-simplicially enriched functors $\CC_{\Tw}^{\op} \times \CC_{\Tw}\lrar \Set^+_\Del$ over the mapping space functor $\Map_{\CC_{\Tw}}$. In turn, this model category is equivalent (already before stabilization) to the model category $\Set_\Del^{\uline{\Tw}_2(\CC)}$ (Proposition~\ref{p:bousfield}), from which the result follows.
\end{enumerate}

\subsection{The tangent bundle of marked simplicial sets}\label{s:marked-tangent}
Our goal in this section is to prove Proposition \ref{p:fiberwise equivalence}, identifying the tangent bundle of the category $\Set^+_\Del$ of marked simplicial sets endowed with the categorical model structure. 

\begin{cons}\label{c:diagrams}
Consider the (co)Cartesian fibrations
$$
\ev_1: \left(\Set^+_\Del\right)^{[1]}\lrar \Set^+_\Del \qquad\text{and}\qquad \ev_*: \left(\Set^+_\Del\right)^{(\NN\times\NN)_\ast}\lrar \Set^+_\Del
$$
which classify the functors $X\mapsto (\Set^+_\Del)_{/X}$ and $X\mapsto ((\Set^+_\Del)_{X//X})^{\NN\times\NN}$. By \cite[Lemma 3.11]{part0}, these functors have the structure of relative model categories, where the domain carries the Reedy model structure induced by the categorical model structure on $\Set^+_\Del$ (here $[1]$ has the Reedy structure with only decreasing maps). Let us consider the following two left Bousfield localizations of these Reedy model structures:
\begin{itemize}
 \item Let $\LFib$ be the localization of $\left(\Set^+_\Del\right)^{[1]}$ whose local objects are the marked left fibrations $Y\lrar X$, where $X$ is a fibrant marked simplicial set. By Remark~\ref{r:local-2}, this can be obtained by localizing with respect to the set of maps
 $$
 h_{1}\times L\coprod_{h_1\times K} h_0 \times K\lrar h_0\times L
 $$
 where $h_i=\Map(i, -)$ is the corepresentable functor and $K\lrar L$ is either $(\Lambda^n_0)^\flat\lrar (\Delta^n)^\flat$ or $(\Del^1)^\flat\lrar (\Del^1)^\sharp$.
 \item Let $\LFib_{\Sp}$ be the localization of $\left(\Set^+_\Del\right)^{(\NN\times\NN)_\ast}$ whose local objects are the parameterized $\Om$-spectrum objects $X_{\bullet\bullet}\lrar X_*$ over a fibrant object $X_*$, where each $X_{m,n}\lrar X_*$ is a marked left fibration. Explicitly, this can be obtained by first localizing to get the model category $\T\Set^+_\Del$ (see \S\ref{s:recall} and \cite[Theorem 3.10]{part0}), and then localizing further at the maps
 $$
 h_{*}\times K\coprod_{h_*\times K} h_{m, n} \times K\lrar h_{m,n}\times L
 $$
 where $K\lrar L$ is as above.
\end{itemize}
It follows from \cite[Proposition 3.12]{part0} that the (co)Cartesian fibrations 
$$
\LFib\lrar \Set_\Del^+\qquad\qquad \text{and} \qquad\qquad \LFib_{\Sp}\lrar \Set_\Del^+
$$
are both relative model categories. The fibers over a fibrant object $\C\in \Set_\Del^+$ are the model categories $\Set_\Del^{\C}$ and $\Sp^{\C}$ of Notation \ref{n:diagrams}.
\end{cons}

\begin{pro}\label{p:fiberwise equivalence}
There is a commuting square of right Quillen functors 
\begin{equation}\label{e:fiberwise}
\vcenter{\xymatrix{
\T\Set^+_\Del \ar[r]^-{\R^{\Sp}}\ar[d] & \LFib_{\Sp}\ar[d]\\ \Set^+_{\Del}\ar[r]_{\Tw^+} & \Set^+_{\Del} }}
\end{equation}
where the top functor induces a Quillen equivalence $\T_\C\Set^+_\Del\lrar \Sp^{\Tw^+(\C)}$ between the fibers, for each fibrant marked simplicial set $\C$.
\end{pro}

The remainder of this section is devoted to the proof of Proposition~\ref{p:fiberwise equivalence}. Let us start by proving that the bottom horizontal arrow of~\eqref{e:fiberwise} is a right Quillen functor.
\begin{pro}\label{p:twisted is right Quillen}
The functor 
$$\Tw^+:\Set^+_{\Del}\lrar \Set^+_{\Del}$$ 
is a right Quillen functor with respect to the categorical model structure.
\end{pro}

\begin{lem}\label{l:twistedisrightquillen}
Let $p: X\lrar Y$ be a map of marked simplicial sets and let 
$$
\R_{X}^+(Y) := \Tw^+(Y) \times_{Y^{\op} \times Y} X^{\op} \times X
$$
equipped with the natural maps $q: \Tw^+(X) \lrar \R_{X}^+(Y)$ and $q':\R_X^+(Y) \lrar \Tw^+(Y)$. Then the following assertions hold:
\begin{enumerate}
\item\label{i:1} If $p$ is a trivial fibration in $\Set_{\Del}^+$, then $q$ and $q'$ are trivial fibrations in $\Set_{\Del}^+$. 
\item\label{i:2} If $p$ is a fibration in $\Set_{\Del}^+$, then $q$ is a marked left fibration and $q'$ is a fibration in $\Set^+_\Del$. 
\end{enumerate}
\end{lem}
\begin{proof}
We first note that $q'$ is a base change of $X^{\op} \times X \lrar Y^{\op} \times Y$, so the claims concerning $q'$ are obvious. Furthermore, by construction the marked edges of $\Tw^+(X)$ are exactly the edges whose image in $\R_{X}^+(Y)$ is marked. Let $\ovl{p}$ and $\ovl{q}$ be the maps of simplicial sets underlying $p$ and $q$ respectively. It will hence suffice to show that (1), if $\ovl{p}$ is a trivial Kan fibration then $\ovl{q}$ is a trivial Kan fibration and that (2), if $\ovl{p}$ is a Joyal fibration then $\ovl{q}$ is a left fibration.
 
By construction the functor $\Tw: \Set_\Del \lrar \Set_\Del$ admits a left adjoint $F: \Set_\Del \lrar \Set_\Del$, given on simplices by $F(\Del^n) = (\Del^n)^{\op} \ast \Del^n$. Let $G: \Set_\Del \lrar \Set_\Del$ be the functor $G(X) = X^{\op} \coprod X$. Then the functor $F$ receives a natural transformation $G(X)\Rightarrow F(X)$ which is adjoint to the natural transformation $\Tw(X) \lrar X^{\op} \times X$. Claim (1) about $\ovl{q}$ is now equivalent to $F(\partial \Del^n) \coprod_{G(\partial \Del^n)}G(\Del^n) \lrar F(\Del^n)$ being a cofibration, which can be directly verified. Similarly, to prove Claim (2) about $\ovl{q}$ it suffices to show that $F(\Lam^n_i) \coprod_{G(\Lam^n_i)}G(\Del^n) \lrar F(\Del^n)$ is an inner fibration for $0 \leq i < n$. This part is indeed verified in the proof of~\cite[Proposition 5.2.1.3]{Lur14} (where the map in question is denoted $K \lrar \Del^{2n+1}$). 
\end{proof}

\begin{proof}[Proof of Proposition~\ref{p:twisted is right Quillen}]
By Lemma~\ref{l:twistedisrightquillen}(i) $\Tw^+$ preserves trivial fibrations and by Lemma~\ref{l:twistedisrightquillen}(ii) and Remark~\ref{r:local-2} it preserves fibrations between fibrant objects. The result then follows from \cite[Proposition 8.5.4]{Hir}.
\end{proof}
Given a marked simplicial set $X$, the construction of Lemma \ref{l:twistedisrightquillen} defines a functor 
$$\xymatrix@R=0pc{
\R^+_{X}:\left(\Set^+_{\Del}\right)_{X//X} \ar[r] & \Set_{\Del}^{\Tw^+(X)}; \hspace{8pt} \R_X^+(Y)= \Tw^+(Y)\times_{Y^{\op}\times Y} X^{\op}\times X.
}$$  
The map $\R_X^+(Y)\lrar \Tw^+(X)$ is induced by the structure map $Y\lrar X$.
\begin{pro}\label{p:R_C is right Quillen}
For any $X\in\Set^+_{\Del}$, the functor 
$$
\R^+_X:\left(\Set^+_{\Del}\right)_{X//X} \lrar \Set_{\Del}^{\Tw^+(X)}$$
is a right Quillen functor. 
\end{pro}
\begin{proof}
Unwinding the definitions, one sees that for any map $X\lrar Y\lrar Z\lrar X$ in $\left(\Set^+_{\Del}\right)_{X//X}$, there is a pullback square of marked simplicial sets (over $\Tw^+(X)$)
$$\xymatrix{
\R_{X}^+(Y)\ar[d]\ar[r] & \Tw^+(Y)\ar[d]\\
\R_{X}^+(Z)\ar[r] & \R_Y^+(Z).
}$$
It then follows from Lemma~\ref{l:twistedisrightquillen} and Remark~\ref{r:fibrations} that $\R^+_{X}$ preserves trivial fibrations and fibrations between fibrant objects, so that the result follows from \cite[Proposition 8.5.4]{Hir}. 
\end{proof}
\begin{rem}\label{r:Rnatural}
Let $f: Y\lrar X$ be a map in $\Set^+_\Del$ and let $\vphi: \Tw^+(Y)\lrar \Tw^+(X)$ be the induced map. For any retractive object $X\lrar Z\lrar X$, there is a natural isomorphism
$$
\R^+_{Y}(Z\times_X Y) \cong \R^+_X(Z)\times_{\Tw^+(X)} \Tw^+(Y).
$$
In other words, there is a natural isomorphism $\R^+_{Y}\circ f^* \cong \phi^*\circ \R^+_{X}$.
\end{rem}
Let us now consider the functors $\R_X^{\Sp}=\Sp(\R_X^+): \Sp((\Set^+_\Del)_{X//X})\lrar\Sp^{\Tw^+(X)}$ arising from Proposition \ref{p:R_C is right Quillen}.
\begin{pro}\label{p:global R}
The functors $\R_X^{\Sp}$ assemble to a right Quillen functor 
$$
\R^{\Sp}:\T\Set_{\Del}^+\lrar \LFib_{\Sp};\qquad \R^{\Sp}(X)_{n,m} = \R^+_{X_\ast}(X_{n,m})
$$
covering the right Quillen functor $\Tw^+:\Set^+_{\Del}\lrar \Set^+_{\Del}$, where $\LFib_{\Sp}$ is as in Construction \ref{c:diagrams}.
\end{pro}
\begin{proof}
Let us first verify that $\R^{\Sp}$ is a right Quillen functor for the Reedy model structures, of which both $\T\Set^+_\Del$ and $\LFib_{\Sp}$ are left Bousfield localizations. Recall that a map $f: Y\lrar X$ of $(\NN\times \NN)_\ast$-diagrams is a (trivial) Reedy fibration if $Y_*\lrar X_*$ is a (trivial) fibration and each matching map $M_{(m, n)}(f): Y_{m,n}\lrar X_{m, n}\times_{X_\ast} Y_\ast$ is a (trivial) fibration in $\Set^+_\Del$. If this is the case, then the map
$$
\R^{\Sp}(Y)_* = \Tw^+(Y_*)\lrar \Tw^+(X_*)=\R^{\Sp}(X)_*
$$
is a (trivial) fibration in $\Set^+_\Del$ by Proposition \ref{p:twisted is right Quillen}. Furthermore, for each $(m, n)$ we can use Remark \ref{r:Rnatural} to identify the matching map $\R^{\Sp}(Y)_{m, n}\lrar \R^{\Sp}(X)_{m, n}\times_{\R^{\Sp}(X)_*} \R^{\Sp}(Y)_*$ with the map
\begin{equation}\label{e:matching}\xymatrix@C=6pc{
\R^+_{Y_*}(Y_{m,n}) \ar[r]^-{\R^+_{Y_*}(M_{(m, n)}(f))} & \R^{\Sp}_{Y_*}(X_{m,n}\times_{X_*} Y_*).
}\end{equation}
This map is a (trivial) marked left fibration in $\Set^+_\Del$ by Proposition \ref{p:R_C is right Quillen}. By Remark \ref{r:local-2}, this marked left fibration \eqref{e:matching} is a categorical fibration in $\Set^+_\Del$ when $X$ and $Y$ are Reedy fibrant, so $\R^{\Sp}$ preserves trivial fibrations and fibrations between fibrant objects. This means that it is right Quillen for the Reedy model structure by \cite[Proposition 8.5.4]{Hir}.

To see that $\R^{\Sp}$ is right Quillen for the localized model structures, it remains to be shown (by \cite[Proposition 8.5.4]{Hir}) that it preserves local objects. Suppose that $X$ is a Reedy fibrant object which is local in $\T\Set^+_\Del$, i.e.\ $X_{\bullet\bullet}\lrar X_*$ is a parameterized $\Omega$-spectrum object. Since $\R^+_{X_*}$ is right Quillen by Proposition \ref{p:R_C is right Quillen}, its image $\R^+_{X_*}(X_{\bullet\bullet})\lrar \R^+_{X_*}(X_*)=\Tw^+(X_*)$ is an $\Omega$-spectrum $\Set_\Del^{\Tw^+(X_*)}$. By Remark \ref{r:local-2}, this is precisely a parameterized $\Omega$-spectrum of marked simplicial sets, each left fibered over $\Tw(X_\ast)$, i.e.\ a local object is $\LFib_{\Sp}$.
\end{proof} 

\begin{pro}\label{p:R-equiv}
Let $\C$ be a fibrant marked simplicial set. Then the right Quillen functor $\R^+_\C$ of Proposition \ref{p:R_C is right Quillen} induces a right Quillen equivalence
$$\R^{\Sp}_\C:=\Sp(\R^+_\C): \Sp\left(\left(\Set^+_{\Del}\right)_{\C//\C}\right)\lrar \Sp^{\Tw^+(\C)}=\Sp\big((\Set_\Del^{\Tw^+(\C)})_\ast\big).$$
\end{pro}

\begin{proof}
Let $\uline{\C}$ be the Joyal fibrant simplicial set underlying $\C$. Since forgetting the marking gives right Quillen equivalences (see Remark \ref{r:forgetmarking})
$$
\left(\Set^+_{\Del}\right)_{\C//\C}\lrar \left(\Set_{\Del}\right)^{\Joy}_{\uline{\C}//\uline{\C}}\qquad\quad \text{and}\quad \qquad \Sp^{\Tw^+(\C)}\lrar \Sp^{\Tw(\uline{\C})}
$$
it suffices to show that the unmarked analogue of $\R_{\C}^+$
$$
\R_{\uline{\C}}: \left(\Set_{\Del}\right)^{\Joy}_{\uline{\C}//\uline{\C}}\lrar \Set_{\Del}^{\Tw(\uline{\C})}; \hspace{8pt}\R_{\uline{\C}}(Y)= \Tw(Y)\times_{Y^{\op}\times Y} \uline{\C}^{\op}\times \uline{\C}
$$
induces a right Quillen equivalence after stabilization. Since the covariant (resp.\ slice-coslice) model structures over weakly equivalent quasicategories are Quillen equivalent, we may replace $\uline{\C}$ by an equivalent quasicategory and assume that $\uline{\C}=\rN(\A)$ for some fibrant simplicial category $\A$. It then suffices to show that the composite with the nerve (which is a Quillen equivalence)
\begin{equation}\label{e:twistnerve}\xymatrix{
(\Cat_\Del)_{\A//\A}\ar[r]^-{\rN}_-\sim & \left(\Set_{\Del}\right)^{\Joy}_{\rN(\A)//\rN(\A)}\ar[r]^-{\R_{\rN(\A)}} & \Set_{\Del}^{\Tw(\rN(\A))}
}\end{equation}
induces a right Quillen equivalence on stabilization. The right Quillen functor \eqref{e:twistnerve} is naturally equivalent (over fibrant objects) to a somewhat more accessible Quillen functor. To see this, recall the following construction from the proof of~\cite[Proposition 5.2.1.11]{Lur14}: for every simplicial category $\B$, there is a map of simplicial sets over $\rN(\B)\times\rN(\B^{\op})$
$$
\beta_{\B}: \Tw(\rN(\B))\lrar \Un(\Map_\B)
$$
from the twisted arrow category of $\rN(\B)$ to the unstraightening of the mapping space functor $\Map_\B: \B\times \B^{\op}\lrar \Set_\Del$. Furthermore, $\beta_\B$ is an equivalence of left fibrations over $\rN(\B)\times\rN(\B^{\op})$ whenever $\B$ is fibrant. Now $\beta_{\B}$ depends naturally on $\B$ and so for every retract diagram $\A\lrar \B\lrar \A$ there is a commuting square of simplicial sets over $\rN(\A)^{\op}\times\rN(\A)$ of the form
\begin{equation}\label{e:betas}\vcenter{\xymatrix{
\Tw(\rN(\B))\times_{\rN(\B^{\op})\times\rN(\B)} \rN(\A^{\op})\times \rN(\A)\ar[r]^-{\beta'_{\B}}\ar[d] & \Un(\Map_\B)\times_{\times_{\rN(\B^{\op})\times\rN(\B)}} \rN(\A^{\op})\times \rN(\A)\ar[d]\\
\Tw(\rN(\A))\ar[r]_{\beta_{\A}} & \Un(\Map_\A)
}}\end{equation}
where $\beta'_{\B}$ is simply the base change of $\beta_{\B}$. When $\B\lrar \A$ is a fibration the horizontal maps are equivalences of left fibrations over $\rN(\A)^{\op}\times \rN(\A)$. 

Note that that the left vertical map in \eqref{e:betas} is the map $\R_{\rN(\A)}(\rN(\B))\lrar \R_{\rN(\A)}(\rN(\A))$ obtained by applying $\R_{\rN(\A)}$ to $\rN(\B)\lrar \rN(\A)$. Furthermore, the naturality of the unstraightening \cite[Proposition 2.2.1.1]{Lur09} implies that the top right corner is naturally isomorphic to $\Un(\G^{\A}(\B))$, where $\G^{\A}(\B): \A\times\A^{\op}\lrar \Set_\Del$ is the restriction of $\Map_{\B}$ to $\A\times\A^{\op}$. The right vertical map is then obtained by applying $\Un$ to the projection $\G^{\A}(\B)\lrar \Map_\A$. In particular, we deduce that both vertical maps are fibrations when $\B\lrar \A$ is a fibration of simplicial categories.

The map into the pullback of \eqref{e:betas} therefore yields a map of simplicial sets over $\Tw(\rN(\A))$
$$
\gamma_{\B}: \R_{\rN(\A)}(\rN(\B))\lrar \beta_{\A}^*\big(\Un(\G^{\A}(\B))\big)
$$
which depends functorially on $\B\in (\Cat_\Del)_{\A//\A}$ and is a weak equivalence when $\B$ is fibrant over $\A$. In other words, $\gamma_\B$ determines a right Quillen homotopy from \eqref{e:twistnerve} to the composite right Quillen functor
$$
\xymatrix{
(\Cat_\Del)_{\A//\A} \ar^-{\G^{\A}}[r] & \Fun(\A^{\op} \times \A,\Set_\Del)_{/\Map_{\A}} \ar[r]^-\simeq & \big(\Set_{\Del}^{\rN(\A^{\op})\times\rN(A)}\big)_{/\Tw(\rN(\A))}\ar[r]^-{\simeq} & \Set_{\Del}^{\Tw(\rN(\A))}.
}
$$
The second functor takes the unstraightening over $\A^{\op}\times \A$ and pulls back along $\beta_{\A}:\Tw(\rN(\A)) \lrar \Un(\Map_\A)$ and the last right Quillen equivalence is given by the identity functor on the underlying categories (see the discussion in~\cite[\S 3.3]{part2}).
It therefore suffices to verify that this composite right Quillen functor induces a Quillen equivalence after stabilization. But this is precisely the content of~\cite[\S 3.3]{part2}, using~\cite[Theorem 3.1.14]{part2}.
\end{proof}

\begin{proof}[Proof of Proposition~\ref{p:fiberwise equivalence}]
Combine Proposition~\ref{p:twisted is right Quillen}, Proposition~\ref{p:global R} and Proposition~\ref{p:R-equiv}.
\end{proof}

\subsection{Categories of lifts}\label{s:lifts}
If $\CC$ is a marked-simplicial category, then \cite[Corollary 3.1.16]{part2} identifies the tangent category $\T_{\CC}\Cat^+_\Del$ with the model category of marked-simplicially enriched lifts of the form
\begin{equation}\label{e:enriched lifts-2}
\vcenter{\xymatrix{&& \T\Set^+_\Del\ar[d] \\ \CC^{\op}\times \CC \ar@{-->}[urr]\ar[rr]_{\Map_{\CC}(-,-)} && \Set^+_\Del}}
\end{equation}
At the same time, Proposition~\ref{p:fiberwise equivalence} identifies the tangent bundle projection $\T\Set^+_\Del\lrar \Set^+_\Del$ with the `homotopy pullback' of the projection
\begin{equation}\label{e:covproj}\xymatrix{
\LFib_{\Sp}\ar[r] & \Set^+_\Del
}\end{equation}
along the functor $\Tw^+: \Set^+_\Del\lrar \Set^+_\Del$: for every fibrant marked simplicial set $\C$, the fiber $\T_\C\Set^+_\Del$ is Quillen equivalent to the fiber of \eqref{e:covproj} over $\Tw^+(\C)$. However, since the functor $\R^{\Sp}: \T\Set^+_\Del\lrar \LFib_{\Sp}$ is \textbf{not} $\Set^+_\Del$-enriched, the image of an enriched lift as in \eqref{e:enriched lifts-2}  will not yield an enriched lift against \eqref{e:covproj} over $\CC^{\op}\times\CC$. Instead, a lift as in \eqref{e:enriched lifts-2} yields an enriched lift against \eqref{e:covproj} over the marked-simplicial category $\CC_{\Tw}$ obtained by applying $\Tw^+$ to the mapping objects of $\CC$ (see~\S\ref{s:twisted-2-cells}). Our goal is then to prove the following:
\begin{pro}\label{p:not-there-yet}
Let $\CC$ be a fibrant marked-simplicial category. Then postcomposition with the functor $\R^{\Sp}$ of \eqref{e:fiberwise} induces a right adjoint functor
\begin{equation}\label{e:compare-lift}\Lift_{\Map_{\CC}}\left(\CC^{\op} \times \CC, \T\Set^+_{\Del}\right) \lrar
\Lift_{\Map_{\CC_{\Tw}}}\left(\CC_{\Tw}^{\op} \times \CC^{}_{\Tw}, \LFib_{\Sp}\right) \end{equation}
between the categories of $\Set^+_\Del$-enriched lifts \eqref{e:enriched lifts-2} and of $\Set^+_\Del$-enriched lifts
\begin{equation}\label{e:lift-tw}
\vcenter{\xymatrix{&& \LFib_{\Sp}\ar[d] \\ \CC_{\Tw}^{\op}\times \CC_{\Tw}^{} \ar@{-->}[urr]\ar[rr]_-{\Map_{\CC_{\Tw}}(-,-)} && \Set^+_\Del.}}
\end{equation}
This right adjoint is a right Quillen equivalence when both categories of lifts are endowed with the projective model structure. In particular, the right hand side of \eqref{e:compare-lift} is a model for $\T_{\CC}\Cat^+_\Del$.
\end{pro}
It will be convenient to prove this result in a slightly more general setting, in order to avoid confusion between the two appearances of $\Set^+_\Del$, as the domain and codomain of the functor $\Tw^+$. Let $\bS, \bT$ be symmetric monoidal model categories and let $\R: \bT \lrar \bS$ be a symmetric monoidal right Quillen functor, with left adjoint $\L$. Consider a commuting square
\begin{equation}\label{e:BS}
\vcenter{\xymatrix@C=3pc{
\N\ar[d]_{\pi}\ar[r]_{\G} & \M\ar[d]^{\rho}\ar@/_0.8pc/@{..>}[l]_{\F}\\
\bT \ar[r]_{\R} & \bS\ar@/_0.8pc/@{..>}[l]_{\L}
}}
\end{equation} 
where $\pi$ and $\rho$ are (co)Cartesian fibrations that exhibit $\M$ and $\N$ as relative model categories over $\bS$ and $\bT$. In particular, the fibers of $\pi$ and $\rho$ are model categories and an arrow $\alpha: s\lrar s'$ induces a Quillen pair $\alpha_!: \M_s\adj \M_{s'}: \alpha^*$ between the fibers (see \cite[Lemma 3.6]{part0}). Let us assume that all fibers $\M_s$ and $\N_t$ are combinatorial and that the square has the following properties:
\begin{enumerate}

\item\label{iv} $\G$ is a right Quillen functor with left adjoint $\F$ and the Beck-Chevalley map $\L\circ \rho \Rightarrow \pi\circ  \F$ is a natural isomorphism. 
\item\label{ii}
The category $\M$ is tensored over $\bS$ in such a way that tensoring with a fixed object preserves coCartesian edges and $\rho$ preserves the tensoring. In other words, each object $s\in \bS$ induces functors $s\otimes (-): \M_{s'}\lrar \M_{s\otimes s'}$ for every $s'\in \bS$ and these functors commutes with the various $\alp_!$. In addition, we require that each functor $s\otimes (-): \M_{s'}\lrar \M_{s\otimes s'}$ is a left Quillen functor which preserves weak equivalences and fibrant objects. Similarly, $\N$ is tensored over $\bT$, with the same properties.
\item\label{iii}
The functor $\G$ preserves the tensoring in the sense that we have natural isomorphisms $\R(t)\otimes\G(B) \x{\cong}{\lrar} \G(t\otimes B)$ for $t\in \bT$, $B\in \N$, which satisfy the usual compatibility conditions with respect to the monoidal structure of $\bT$.
\end{enumerate}

\begin{rem}\label{r:relativequillen}
Condition \eqref{iv} implies that $\G$ preserves relative limits and $\F$ preserves relative colimits. In particular, $\G$ preserves Cartesian edges (and $\F$ preserves coCartesian edges) and induces right (Quillen) functors $\G_t: \N_t\lrar \M_{\R(t)}$ on fibers. We will denote by $\F_t: \M_{\R(t)}\lrar \N_t$ the corresponding left adjoint, which first applies $\F$ and then changes between fibers along the counit map via $(\epsilon_t)_!:\N_{\L\R(t)}\lrar \N_{t}$. 
\end{rem}

\begin{rem}\label{r:lotsofconditions}
The square \eqref{e:fiberwise} indeed satisfies the above conditions, where the actions of $\bT=\Set_\Del^+$ on $\N=\T\Set_\Del^+$ and of $\bS=\Set^+_\Del$ on $\M=\LFib_{\Sp}$ are both given by the levelwise Cartesian product $S\otimes X_{\bullet\bullet}=S\times X_{\bullet\bullet}$. Note that \eqref{iv} holds because $\R=\Tw^+$ and $\G=\R^{\Sp}$ commute with the right adjoints of $\pi$ and $\rho$, which send $X\in \Set^+_\Del$ to the constant $(\NN\times\NN)_*$-diagram on $X$.
\end{rem}

Now suppose that $\I$ is a fibrant $\bT$-enriched category and let $\phi:\I\lrar \bT$ be an enriched functor: for every $i \in \I$ we have an associated object $\phi(i) \in \bT$ and for every $i,j \in \I$ we have a structure map $\phi(i,j): \Map_{\phi}(i,j) \otimes \phi(i) \lrar \phi(j)$ such that the usual compatibility conditions hold. Applying the functor $\R$, we obtain an $\bS$-enriched functor $\phi_{\R}: \I_{\R}\lrar \bS$. Here $\I_{\R}$ is the $\bS$-enriched category with the same objects as $\I$ and mapping spaces $\I_{\R}(i, j)=\R\big(\I(i, j)\big)$. The functor $\phi_\R$ is given on objects by $\phi_\R(i)=\R\big(\phi(i)\big)$ and with structure maps $\phi_{\R}(i,j)$ given by
$$
\I_{\R}(i, j)\otimes \phi_\R(i) = \R\big(\I(i, j)\big)\otimes \R(\phi(i)) \x{\simeq}{\lrar} \R\big(\I(i, j)\otimes \phi(i)\big) \x{\R(\phi(i,j))}{\lrar} \R(\phi(j)) = \phi_{\R}(j).$$ 
Let $\Lift^{\bT}_{\phi}(\I, \N)$ and $\Lift^{\bS}_{\phi_{\R}}(\I_{\R}, \M)$ be the categories of $\bT$-enriched (resp. $\bS$-enriched) lifts
$$\xymatrix@C=3pc{
& \N\ar[d]^\pi & & & \M\ar[d]^\rho\\
\I\ar[r]_\phi\ar@{..>}[ru] & \bT & & \I_\R\ar[r]_{\phi_\R} \ar@{..>}[ru] & \bS.
}$$
There is a functor $\G_*: \Lift^{\bT}_{\phi}(\I, \N)\lrar \Lift^{\bS}_{\phi_{\R}}(\I_{\R}, \M)$, which applies the functor $\G$ pointwise. More precisely, if $f: \I\lrar \M$ is a $\bT$-enriched lift of $\phi$, then $\G_*(f)(i)=\G\big(f(i)\big)$ and for any $i, j\in \I_{\R}$, the action of maps is given by
$$
\xymatrix@C=1.5cm{
\I_{\R}(i, j)\otimes \G_*(f)(i) = \R(\I(i, j))\otimes \G(f(i)) \ar[r]^-{\cong}_-{\eqref{iii}} & \G\big(\I(i, j)\otimes f(i)\big)\ar^-{\G(f(i,j))}[r] & \G(f(j)).\\
}$$ 
In particular, $\G_*$ fits into a commuting square
\begin{equation}\label{d:monadic}\vcenter{\xymatrix@C=2pc{
\Lift^{\bT}_{\phi}(\I, \N)\ar[rr]^-{\G_*}\ar[d]_{\ev_{\N}}  & & \Lift^{\bS}_{\phi_{\R}}(\I_{\R}, \M)\ar[d]^{\ev_{\M}}\\
\prod_{i\in \I} \N_{\phi(i)} \ar[rr]_-{\G'_{\ast}} & & \prod_{i\in \I} \M_{\phi_\R(i)}.
}}\end{equation}
where $\G'_{\ast}= \prod_{i\in \I} \G_{\phi(i)}$ is given by pointwise applying the corresponding functors $\G_t$ (see Remark \ref{r:relativequillen}). The functors $\ev_\N$ and $\ev_\M$ evaluate a section on the objects of $\I$.

\begin{lem}\label{l:projmod}
The category $\Lift^{\bT}_{\phi}(\I, \N)$ carries a combinatorial model structure (the \emph{projective} model structure) such that
$$\xymatrix{
\ev_\N: \Lift^{\bT}_{\phi}(\I, \N)\ar[r] & \prod_{i\in \I} \N_{\phi(i)}
}$$
is both a left and a right Quillen functor, which preserves and detects weak equivalences and fibrations. Similarly for $\Lift^{\bS}_{\phi_{\R}}(\I_{\R}, \M)$.
\end{lem}
\begin{proof}
The functor $\ev_\N$ can be identified with the functor that restricts a lift along the inclusion $\Ob(\I)\lrar \I$. Consequently, it admits both a left and a right adjoint, given by (enriched) left and right Kan extension relative to $\phi$. Let us denote the left adjoint by $\Free_\N$. 

To describe this left adjoint, let $i\in \I$, $a\in \N_{\phi(i)}$ and let us write $a_i\in \prod_{i\in \I} \N_{\phi(i)}$ for the tuple $(\dots, \emptyset, a, \emptyset, \dots)$ given by $a$ at $i$ and initial objects for all $j\neq i$. Then the lift $\Free_\N(a_i)$ is given by
\begin{equation}\label{e:freeevaluated}
\Free_\N(a_i)(j) = \phi(i, j)_!\big(\I(i, j)\otimes a\big)
\end{equation}
where $\phi(i, j)_!: \N_{\I(i, j)\otimes \phi(i)}\lrar \N_{\phi(j)}$. 

Note that the union of all maps $a_i\lrar b_i$ arising from generating (trivial) cofibrations $a\lrar b$ in some $\N_{\phi(i)}$ serve as generating (trivial) cofibrations in $\prod_{i\in \I} \N_{\phi(i)}$. Since the functors $\phi(i, j)_!$ and $\I(i, j)\otimes (-)$ are left Quillen (assumption \eqref{ii}), it follows that $\ev_{\N}\circ \Free_{\N}:\prod_{i\in \I} \N_{\phi(i)} \lrar \prod_{i\in \I} \N_{\phi(i)}$ preserves (trivial) cofibrations. The result now follows from the usual transfer argument.
\end{proof}
In light of Proposition \ref{p:fiberwise equivalence} and Remark \ref{r:lotsofconditions}, Proposition \ref{p:not-there-yet} is now a special case of the following assertion:
\begin{pro}\label{p:not-there-yet-abstract}
The functor 
$$
\G_*: \Lift^{\bT}_{\phi}(\I, \N)\lrar \Lift^{\bS}_{\phi_{\R}}(\I_{\R}, \M)
$$
is a right Quillen functor, where both sides are endowed with the projective model structure. Furthermore, if the Quillen adjunctions $\F_t\dashv \G_t$ are Quillen equivalences for all $t\in \bT$ of the form $\phi(i)$ or $\I(i, j)\otimes \phi(i)$, then $\G_*$ is a Quillen equivalence.
\end{pro}
\begin{proof}
Clearly $\G_*$ preserves fibrations and weak equivalences, since it is given pointwise by the right Quillen functors $\G_{t}$. Since $\G_*$ is accessible and preserves limits, the adjoint functor theorem provides a left adjoint $\F_*$, so that $\G_*$ is right Quillen. Furthermore, if all the $\G_t$ are Quillen equivalences, then the right derived functor $\RR\G_*$ detects weak equivalences (which are determined pointwise). It therefore suffices to show that the derived unit map $\id\lrar \RR\G_*\LL\F_*$ is an equivalence. 

Since the evaluation functor $\ev_\M: \Lift^{\bT}_{/\phi_{\R}}(\I_\R, \M)\lrar \prod_{i\in \I} \M_{\phi_{\R}(i)}$ detects weak equivalences, it suffices to show that the natural transformation
$$
\RR\ev_\M\lrar \RR\ev_\M\RR\G_*\LL\F_*
$$
is an equivalence. Let $\K$ be the class of objects $f$ in $\Lift^{\bS}_{\phi_{\R}}(\I_{\R}, \M)$ for which this map is an equivalence. Since $\RR \G_*$ and $\RR \ev_\M$ preserve homotopy colimits (which are computed pointwise by Lemma \ref{l:projmod}), the class $\K$ is closed under all homotopy colimits.

Since every object arises (up to weak equivalence and retracts) from a transfinite composition of homotopy pushouts of maps $\LL \Free_{\M}(a_i)\lrar \LL \Free_{\M}(b_i)$, for cofibrations $a\lrar b$ in various $\M_{\phi_{\R}(i)}$, it suffices to show that the class $\K$ contains all $\LL \Free_{\M}(a_i)$. Let $a\in \M_{\phi_\R(i)}$ be a cofibrant object and let $a_i\in\prod_{i\in \I} \M_{\phi_{\R}(i)}$ be the induced object. The square \eqref{d:monadic} induces a commuting square of left adjoints, so that there is an isomorphism of (cofibrant) lifts of $\phi$
$$
\F_*\big(\Free_\M(a_i)\big)\cong \Free_\N\big(\F'_{*}(a_i)\big)
$$
where $\F'_*$ is the left adjoint of $\G'_*$, given by pointwise applying $\F_{\phi(i)}$. Using formula \eqref{e:freeevaluated}, we have to verify that for every $j\in \I$, the map
$$\xymatrix{
\phi_{\R}(i, j)_!\big(\I_\R(i, j)\otimes a)\ar[r] & \RR\G_{\phi(j)}\left(\phi(i, j)_!\otimes \F_{\phi(i)}(a)\right)
}$$
is a weak equivalence. Let us denote $t:=\I(i, j)$, so that $\I_{\R}(i, j)=\R(t)$. Since $\G_{\phi(i)}$ is a Quillen equivalence, the above map is an equivalence if its derived adjoint
\begin{equation}\label{e:unitatpoint}\xymatrix{
\F_{\phi(j)}\left(\phi_{\R}(i, j)_!\big(\R(t)\otimes a)\right)\ar[r] & \phi(i, j)_!\big(t\otimes \F_{\phi(i)}(a)\big)
}\end{equation}
is an equivalence (note that all objects involved are cofibrant, since $a$ is cofibrant and $\R(t)\otimes (-)$ is left Quillen by assumption \eqref{ii}). It follows from Remark \ref{r:relativequillen} that
$$
\F_{\phi(j)}\circ \phi_{\R}(i, j)_! \cong \phi(i, j)_!\circ \F_{t\otimes \phi(i)}.
$$
Under this isomorphism, the map \eqref{e:unitatpoint} is the image under $\phi(i, j)_!$ of the map between cofibrant objects
$$\xymatrix{
\F_{t\otimes \phi(i)}\big(\R(t)\otimes a)\ar[r] & t\otimes \F_{\phi(i)}(a).
}$$
It therefore suffices to verify that this map is a weak equivalence in $\N_{t\otimes \phi(i)}$. Note that this is the Beck-Chevalley transformation of the square
$$ \xymatrix@C=3cm{
\N_{\phi(i)} \ar^{\G_{\phi(i)}}[r]\ar_{t \otimes (-)}[d] & \M_{\phi_\R(i)} \ar^{\R(t) \otimes (-)}[d] \\
\N_{\I(i,j) \otimes \phi(i)} \ar_-{\G_{t \otimes \phi(i)}}[r] & \M_{\R(t) \otimes \phi_\R(i)}. \\
}$$
Since $\F_{t\otimes \phi(i)}$ is a left Quillen equivalence, it suffices to verify that the derived adjoint map is a weak equivalence. Unwinding the definitions, this derived adjoint can be identified with the composite
\begin{equation}\label{e:derivedadj}\xymatrix{
\R(t)\otimes a\ar[r]^-{\R(t)\otimes \eta} & \R(t)\otimes \G_{\phi(i)}\left(\F_{\phi(i)}(a)^{\fib}\right) \ar[r]^-{\cong} & \G_{t\otimes \phi(i)}\left(t\otimes \F_{\phi(i)}(a)^{\fib}\right).
}\end{equation}
Note that the codomain of this map indeed computes $\RR\G_{t\otimes \phi(i)}\big(t\otimes \F_{\phi(i)}(a)\big)$, because $t\otimes (-)$ preserves fibrant objects (see \eqref{ii}). The second map is the isomorphism from \eqref{iii} and the first map is the image under $t\otimes (-)$ of the derived unit map $\eta$ of the Quillen equivalence $\F_{\phi(i)}\dashv \G_{\phi(i)}$. Since $t\otimes (-)$ preserves all weak equivalences (by \eqref{ii}), it follows that \eqref{e:derivedadj} is a weak equivalence, which concludes the proof.
\end{proof}

\subsection{Families of marked left fibrations}\label{s:markedfibfam}
Proposition \ref{p:not-there-yet} identifies the tangent model category $\T_{\CC}\Cat^+_\Del$ with a certain model category of enriched lifts of a diagram of marked simplicial sets against $\LFib_{\Sp}\lrar \Set_{\Del}^+$. Informally, one can think of an enriched lift of such a diagram $\F: \I\lrar \Set_\Del^+$ as a collection of $\infty$-functors $g_i:\F(i)\lrar \Sp$ for each $i\in \I$, together with a coherent family of natural transformations
$$\xymatrix{
\I(i, j)\times \F(i)\ar[r]^-{\pi_2}\ar[d] & \F(i)\ar[d]^{g_i} \ar@{=>}[0, 0]-<16pt, 14pt>;[1, -1]+<16pt, 14pt>\\
\F(j)\ar[r]_{g_j} & \Sp.
}$$
To prove Theorem \ref{t:main2}, we will show that the data of such a family of diagrams of spectra is equivalent to the data of a diagram of spectra over the unstraightening of $\F$. This section is devoted to a proof of a preliminary unstable analogue of this result:
\begin{pro}\label{p:sect}
Let $\I$ be a marked-simplicial category and let $\F: \I\lrar \Set_\Del^+$ be a projectively fibrant diagram. Then there is a Quillen equivalence
$$
\St^{\cov} : \Set_\Del^{\Un^{\sca}(\F)} \x{\simeq}{\adj} \Lift_{\F}\big(\I, \LFib\big): \Un^{\cov}
$$
between the marked covariant model structure over the scaled unstraightening of $\F$ and the projective model structure on enriched lifts of $\F$ against $\LFib\lrar \Set^+_\Del$, as in Lemma \ref{l:projmod}.
\end{pro}
Let us start by describing the projective model structure on $\Lift_{\F}\big(\I, \LFib\big)$ in a bit more detail. Since the projection $\LFib\lrar \Set_\Del^+$ is simply given by the codomain fibration $\ev_1: \big(\Set_\Del^+\big)^{[1]}\lrar \Set_\Del^+$ at the level of categories, there is an equivalence of categories
$$
\Lift_{\F}\big(\I, \LFib\big)\simeq \Fun^+(\I,\Set^+_\Del)_{/\F}
$$
between the category of lifts of $\F$ and the category of enriched functors $\I\lrar \Set^+_\Del$ over $\F$. If $f: \tilde{\F}\lrar \tilde{\F}'$ is a map of lifts of $\F$, then $f$ is a weak equivalence (fibration) if each $f_i: \tilde{\F}(i)\lrar \tilde{\F}'(i)$ is a weak equivalence (fibration) in the marked covariant model structure on $(\Set^+_{\Del})_{/\F(i)}$. Under the above equivalence of categories, the projective model structure therefore corresponds to the following model structure on $\Fun^+(\I,\Set^+_\Del)_{/\F}$:

\begin{define}\label{d:marked-left2}
Let $\F: \I\lrar \Set_\Del^+$ be a projectively fibrant enriched functor. We will denote by $\Fun^+(\I,\Set^+_\Del)^{\cov}_{/\F}$ the model category of enriched functors over $\F$, in which a map $\G\lrar \HH\lrar \F$ is a weak equivalence (fibration) if and only if each $\G(i)\lrar \HH(i)\lrar \F(i)$ is a weak equivalence (fibration) in the marked covariant model structure on $(\Set^+_\Del)_{/\F(i)}$. We note that $\Fun^+(\I,\Set^+_\Del)^{\cov}_{/\F}$ has the same trivial fibrations and more weak equivalences than $\Fun^+(\I,\Set^+_\Del)^{\proj}_{/\F}$, and is hence a left Bousfield localization of the latter.
\end{define}

Given a projectively fibrant functor $\F: \I\lrar \Set_\Del^+$, the straightening-unstraightening equivalence of \cite{goodwillie} (recalled in \S\ref{s:straightening}) induces a Quillen equivalence on slice model categories
\begin{equation}\label{e:useful-1}
\St^{\sca}: \left((\Set^+_\Del)^{\lcc}_{/\rN^{\sca}(\I)}\right)_{/\Un^{\sca}(\F)} \x{\simeq}{\adj} \Fun^+(\I, \Set^+_\Del)^{\proj}_{/\F}: \Un^{\sca}.
\end{equation}
It will be useful to describe the left hand side of~\eqref{e:useful-1} in terms of a suitable categorical pattern model structure. For this we will make use of the following general lemma concerning categorical pattern model structures:

\begin{lem}\label{l:joyal}
Let $\Beta = (\uline{S}, E_S,T_S)$ be as in \S\ref{s:straightening} and let $S$ be the marked simplicial set $(\uline{S}, E_S)$. For each $\Beta$-fibered object $p: X=(\uline{X}, E_X)\lrar S$, the slice model structure on $\left((\Set^+_{\Del})_{/\Beta}\right){}_{/X}$ gives a model structure on the equivalent category $(\Set^+_{\Del})_{/X}$. This model structure coincides with the model structure associated to the categorical pattern $p^*\Beta = (\uline{X}, E_X,p^{-1}(T_S))$ on $\uline{X}$. 
\end{lem}
\begin{proof}
Since both model structures have the same cofibrations, it suffices to show they have the same fibrant objects. In other words, we need to show that a map $q:Y \lrar X$ of marked simplicial sets over $S$ is a fibration in $(\Set^+_{\Del})_{/\Beta}$ if and only if is has the right lifting property with respect to all $p^*\Beta$-anodyne maps in $(\Set^+_{\Del})_{/X}$. By the construction of $p^*\Beta$ we see that a map is $p^*\Beta$-anodyne if and only if it forgets to a $\Beta$-anodyne map in $(\Set^+_{\Del})_{/S}$. It therefore suffices to show that $q$ is a fibration in $(\Set^+_{\Del})_{/\Beta}$ if and only if is has the right lifting property with respect to all $\Beta$-anodyne maps. 

One direction is clear, since every $\Beta$-anodyne map is a trivial cofibration in $(\Set^+_{\Del})_{/\Beta}$. To prove the other direction, assume that $q:Y \lrar X$ has the right lifting property with respect to all $\Beta$-anodyne maps. We wish to show that $q$ is a fibration in $(\Set^+_{\Del})_{/\Beta}$. Let $i: A \lrar B$ be a trivial cofibration in $(\Set^+_{\Del})_{/\Beta}$ and consider the diagram of mapping spaces
$$\xymatrix{
\Map^{\sharp}_S(B,Y) \ar[r]^-{\tau} & \Map^{\sharp}_S(B,X) \times_{\Map^{\sharp}_S(A,X)} \Map_S^{\sharp}(A,Y)\ar[r]^-{\pi_2} & \Map_S^{\sharp}(A,Y).
}$$
It suffices to verify that $\tau$ is a trivial Kan fibration. Since $X$ and $Y$ are both $\Beta$-fibered over $S$, it follows that the map $\pi_2$ and the composite $\pi_2\tau$ are trivial Kan fibrations. 

On the other hand, the map $\tau$ is a left fibration: indeed, this follows from the fact that for every left anodyne map $j: C\lrar D$, the map $j^\sharp: C^\sharp\lrar D^\sharp$ is $\Beta$-anodyne, so that the pushout-product of $i$ and $j^\sharp$ is $\Beta$-anodyne as well. Since $\pi_2$ is a trivial fibration, the fibers of $\tau$ are equivalent to the fibers of $\pi_2\tau$ and are hence contractible. We conclude that the left fibration $\tau$ is a trivial fibration.
\end{proof}

Using Lemma~\ref{l:joyal} we can reformulate~\eqref{e:useful-1} as follows.
Let $\Beta=(\Un^{\sca}(\F), E, T)$, where $E$ is the set of marked edges of $\Un^{\sca}(\F)$ and $T$ is the set of triangles which map to thin triangles in $\rN^{\sca}(\I)$. Combining~\eqref{e:useful-1} with Lemma~\ref{l:joyal} we then obtain a Quillen equivalence
\begin{equation}\label{e:useful-2}
\St^{\sca}: (\Set^+_{\Del})_{/\Beta} \x{\simeq}{\adj} \Fun^+(\I, \Set^+_\Del)^{\proj}_{/\F} :\Un^{\sca}.
\end{equation}
In light of the above discussion, Proposition \ref{p:sect} can now be reformulated as follows:
\begin{pro}\label{p:bousfield}
The Quillen equivalence~\eqref{e:useful-2} descends to a Quillen equivalence
$$
\St^{\cov} : \Set_\Del^{\Un^{\sca}(\F)}=(\Set^+_\Del)^{\cov}_{/\Un^{\sca}(\F)} \x{\simeq}{\adj} \Fun^+(\I, \Set^+_\Del)^{\cov}_{/\F}: \Un^{\cov}
$$
between the model categories of Definition \ref{d:lcc-left} and Definition \ref{d:marked-left2}.
\end{pro}
\begin{proof}
Both model structures are left Bousfield localizations of the slice model categories appearing in \eqref{e:useful-2}, by Remark \ref{r:local-2}. By \cite[Theorem 3.3.20]{Hir} it suffices to verify that a slice fibrant object $\G \lrar \F$ is local with respect to the left Bousfield localization on the right hand side if and only if $\Un^{\sca}(\G) \lrar \Un^{\sca}(\F)$ is local with respect to the left Bousfield localization on the left hand side.

Let us first show that a $\Beta$-fibered $p: Y\lrar \Un^{\sca}(\F)$ is fibrant in $(\Set^+_\Del)^{\cov}_{/\Un^{\sca}(\F)}$ if and only if $p_{\{i\}}:Y \times_{\rN^{\sca}(\I)} \{i\} \lrar \Un^{\sca}(\F) \times_{\rN^{\sca}(\I)} \{i\}$ is a marked left fibration for each $i \in \I$. 
Indeed, each $p_{\{i\}}$ is a marked left fibration if $p$ is. Conversely, if $p:Y \lrar \Un^{\sca}(\F)$ is $\Beta$-fibered and each $p_{\{i\}}$ is a marked left fibration, then $Y \lrar \Un^{\sca}(\F)$ is a locally coCartesian fibration by~\cite[Proposition 2.4.2.11]{Lur09}. In addition, the fibers of $p$ are Kan complexes, so that all edges are locally coCartesian and $p$ is a left fibration by~\cite[Proposition 2.4.2.8]{Lur09} and~\cite[Proposition 2.4.2.4]{Lur09}. 

Now, for each $i\in \I$, there is a commuting square (see~\cite[Remarks 3.5.16,~3.5.17]{goodwillie})
$$ \xymatrix{
\Fun(\I,\Set^+_\Del) \ar^{\Un^{\sca}}[r]\ar_{\G \mapsto \G(i)}[d] & (\Set^+_\Del)^{\lcc}_{/\rN^{\sca}(\I)} \ar^{Y \mapsto Y \times_{\rN^{\sca}(\{i\})} \{i\}}[d] \\
\Fun(\{i\},\Set^+_\Del) \ar^{\Un^{\sca}_{\{i\}}}[r] & (\Set^+_\Del)^{\lcc}_{/\rN^{\sca}(\{i\})}. \\
}$$
It follows from the previous paragraph that $\Un^{\sca}(\G)\lrar \Un^{\sca}(\F)$ is a marked left fibration if and only if $\Un^{\sca}_{\{i\}}(\G(i)) \lrar \Un^{\sca}_{\{i\}}(\F(i))$ is a marked left fibration for each $i$. It remains to verify that this is equivalent to $\G(i)\lrar \F(i)$ being a marked left fibration for each $i$. In particular, it suffices to prove the claim for the case $\I = \ast$. 

In this case we may identify both $\Fun(\ast,\Set^+_\Del)$ and $(\Set^+_\Del)_{/\rN^{\sca}(\ast)}$ with $\Set^+_\Del$ (equipped with the categorical model structure) and consider $\Un^{\sca}_{\ast}$ as a right Quillen functor from $\Set^+_\Del$ to itself. By \cite[Proposition 3.6.1]{goodwillie}, there is a natural transformation $\Id \Rightarrow \Un^{\sca}_{\ast}$ which is a weak equivalence on fibrant objects, so that every fibration $p:Y \lrar X$ between fibrant objects in $\Set^+_\Del$ fits into a commutative diagram
\begin{equation}\label{e:un}
\vcenter{\xymatrix{
Y \ar^-{\simeq}[r]\ar_{p}[d] & \Un^{\sca}_{\ast}(Y) \ar^{\Un^{\sca}_{\ast}(p)}[d] \\
X \ar^-{\simeq}[r] & \Un^{\sca}_{\ast}(X). \\
}}
\end{equation}
We can think of this map as a weak equivalence between fibrant objects in the arrow category $(\Set^+_\Del)^{[1]}$, so that $p$ is a local object in the left Bousfield localization $\LFib$ of Construction \ref{c:diagrams} if and only if $\Un^{\sca}_{\ast}(p)$ is a local object. The local objects of $\LFib$ are precisely the marked left fibrations over fibrant marked simplicial sets, so the result follows.
\end{proof}
\begin{cor}\label{c:last}
There is a right Quillen equivalence
$$ \uline{\Un}^{\cov}:\Lift_{\F}\big(\I, \LFib\big)\simeq \Fun^+(\I, \Set^+_\Del)^{\cov}_{/\F} \x{\simeq}{\lrar} (\Set_{\Del})^{\cov}_{/\uline{\Un}^{\sca}(\F)}. $$
\end{cor}
\begin{proof}
Compose the Quillen equivalences of Proposition~\ref{p:bousfield} and Lemma~\ref{l:lcc-left}.
\end{proof}

\subsection{Proof of the main theorem}\label{s:proof-main}
To conclude the proof of Theorem \ref{t:main2}, we need a spectral analogue of Corollary \ref{c:last}. More precisely, let $\F: \I\lrar \Set^+_\Del$ be an enriched functor and consider the category of enriched lifts of $\F$ against $\LFib_{\Sp}\lrar \Set_\Del^+$, endowed with the projective model structure of Lemma \ref{l:projmod}. Recall from Construction \ref{c:diagrams} that the underlying functor of $\LFib_{\Sp}\lrar \Set_\Del^+$ is given by the projection
$$
\ev_\ast: \left(\Set^+_\Del\right)^{(\NN\times \NN)_\ast}\lrar \Set^+_\Del.
$$
It follows that the category of enriched lifts of $\F$ is equivalent (as an ordinary category) to the category of enriched functors $\tilde{\F}: \I\times (\NN\times \NN)_\ast\lrar \Set^+_\Del$ whose restriction to $\I\times \{\ast\}$ is $\F$. In turn, this category is equivalent to the category of $\NN\times\NN$-diagrams in $\Fun^+(\I, \Set^+_\Del)_{\F//\F}$.
\begin{lem}\label{l:step2}
Let $\F: \I\lrar \Set^+_\Del$ be a projectively fibrant enriched functor. Then the equivalence of categories described above provides an identification
\begin{equation}\label{e:step2}
\Lift_{\F}\left(\I, \LFib_{\Sp}\right) \x{\simeq}{\lrar} \Sp\left(\Fun^+(\I, \Set^+_\Del)^{\cov}_{\F//\F}\right)
\end{equation}
between the projective model structure on lifts and the stabilization of the model structure of Definition \ref{d:marked-left2}.
\end{lem}
\begin{proof}
It suffices to show that both sides have the same trivial fibrations and fibrant objects. Let us represent an object in either of these two categories by a functor $\G: \I \times (\NN \times \NN)_{*} \lrar \Set^+_\Del$ whose restriction to $\I \times \{\ast\}$ coincides with $\F$, and let us denote the value of $\G$ at $(i,n,m)$ by $\G_{n,m}(i)$. Since trivial fibrations are unchanged by left Bousfield localization we have that a map $\HH \lrar \G$ between such functors is a trivial fibration in either the left or right hand side if and only if $\HH_{n,m}(i) \lrar \G_{n,m}(i)$ is a trivial fibration of marked simplicial sets for every $i\in \I$ and $n,m \in \NN$. 

Weak zero-objects and homotopy Cartesian squares in $\Fun^+(\I, \Set^+_\Del)^{\cov}_{\F//\F}$ are detected in $\Fun^+(\I,\Set^+_\Del)^{\proj}_{/\F}$. An object $\G$ is on the right hand side is therefore fibrant if and only if $\G_{\bullet\bullet}(i)$ is an $\Omega$-fibrant spectrum object over $\F(i)$ and each $\G_{n, m}(i)\lrar \F(i)$ is a marked left fibration. This means precisely that $\G$ is fibrant on the left hand side. 
\end{proof}

We are now ready to harness the above results to compute the tangent categories of $\Cat_{(\infty,2)} \simeq (\Cat^+_\Del)_{\infty}$. 

\begin{proof}[Proof of Theorem~\ref{t:main2}]
Let $\CC$ be a marked-simplicial category, let $\CC_{\Tw}$ be the marked-simplicial category obtained by applying $\Tw^+$ to its mapping objects and let 
$\Map_{\CC_{\Tw}}(-,-):\CC_{\Tw}^{\op}\times \CC^{}_{\Tw}\lrar \Set_\Del^+ $
be the mapping functor. 
Combining \cite[Corollary 3.1.16]{part2} with Proposition~\ref{p:not-there-yet}, Lemma~\ref{l:step2} and Corollary~\ref{c:last} (where $\F=\Map_{\CC_{\Tw}}$) we obtain a composable sequence of natural right Quillen functors, which are Quillen equivalences when $\CC$ is fibrant:
\begin{align}\label{m:all-equiv}
U_\CC:\T_{\CC}\Cat^+_{\Del} & \x[\text{\cite{part2}}]{}{\xrightarrow{\hspace{15pt}\simeq\hspace{15pt}}} \Lift_{\Map_{\CC}}\left(\CC^{\op} \times \CC, \T\Set^+_\Del\right)\nonumber\\
& \x[\ref{p:not-there-yet}]{}{\xrightarrow{\hspace{15pt}\simeq\hspace{15pt}}} \Lift_{\Map_{\CC_{\Tw}}}\left(\CC_{\Tw}^{\op} \times \CC^{}_{\Tw}, \LFib_{\Sp}\right)\nonumber\\
& \x[\ref{l:step2}]{}{\xrightarrow{\hspace{15pt}\simeq\hspace{15pt}}} \Sp\left(\Fun^+(\CC_{\Tw}^{\op} \times \CC^{}_{\Tw}, \Set^+_\Del)^{\cov}_{/\Map_{\CC_{\Tw}}}\right)\\
& \x[\ref{c:last}]{}{\xrightarrow{\hspace{15pt}\simeq\hspace{15pt}}} \Sp\left((\Set_{\Del})^{\cov}_{/\uline{\Un}^{\sca}(\Map_{\CC_{\Tw}})}\right) = 
\Sp^{\uline{\Tw}_2(\CC)}\nonumber
\end{align}
Unraveling the definitions, one sees that $U_{\CC}$ sends an $\NN\times\NN$-diagram $\CC\x{\iota}{\lrar} \DD_{\bullet\bullet}\lrar \CC$ to the $\NN\times\NN$-diagram
$$
\uline{\Un}^{\sca}\left(\R^+_{\Map_{\CC}(-, -)}\big(\Map_{\DD_{\bullet\bullet}}(\iota-, \iota-)\big)\right)
$$
where the scaled unstraightening is over $\CC_{\Tw}\times\CC_{\Tw}^{\op}$. Using the compatibility of the scaled unstraightening with restriction one finds that the Quillen equivalence $U_{\CC}$ depends naturally on $\CC$, as asserted.
\end{proof}
\begin{proof}[Proof of Theorem~\ref{t:main}]
Using the identification
$$ \left(\Sp^{\uline{\Tw}_2(\CC)}\right)_{\infty} \simeq \Sp\left(\Fun(\uline{\Tw}_2(\CC),\cS_{\ast}\right)) \simeq \Fun(\uline{\Tw}_2(\CC),\Sp) $$ we may conclude that the underlying $\infty$-category $(\T_{\CC}\Cat^+_\Del)_{\infty}$ is naturally equivalent to the $\infty$-category $\Fun(\uline{\Tw}_2(\CC),\Sp)$ of functors from $\uline{\Tw}_2(\CC)$ to spectra.
\end{proof}

\subsection{The cotangent complex of an $(\infty,2)$-category}
Theorem~\ref{t:main} identifies the tangent $\infty$-category to $\Cat^+_\Del$ at a marked-simplicial category $\CC$ with the $\infty$-category of spectrum-valued functors $\uline{\Tw}_2(\CC) \lrar \Sp$ from the twisted $2$-cell $\infty$-category of $\CC$. Our goal in this section is to identify the image of the cotangent complex $L_{\CC}$ of $\CC$ under this equivalence. 

Throughout, let us fix a fibrant model $\SS$ for the sphere spectrum in the model category $\Sp^\ast=\Sp((\Set_\Del)^{\cov}_{*//*})$, i.e.\ the stable model structure on $\NN\times\NN$-diagrams of pointed simplicial sets. In particular, $\SS_{n,n} \simeq \hocolim_k \Om^{k}S^{n+k}$. Let $r: \uline{\Tw}_2(\CC) \lrar \ast$ denote the terminal map. We then claim the following:
\begin{pro}\label{p:cotangent}
Under the equivalence of Theorem~\ref{t:main}, the cotangent complex $L_{\CC}$ corresponds to the constant diagram $\uline{\Tw}_2(\CC) \lrar \Sp$ on the twice desuspended sphere spectrum $\SS[-2]$. More precisely, there is a weak equivalence 
$$ \theta_\CC:r^*\SS[-2] \x{\simeq}{\lrar} \RR U_\CC(L_{\CC}) $$ 
in the model category $\Sp^{\uline{\Tw}_2(\CC)}$, where $U_{\CC}$ is the right Quillen equivalence of Theorem \ref{t:main2}.
\end{pro}

\begin{cor}\label{c:twisted-arrow-quillen}
Let $\F: \uline{\Tw}_2(\CC) \lrar \Sp$ be a functor and let $M_{\F} \in \T_{\CC}\Cat^+_\Del$ be the corresponding parameterized spectrum object under the equivalence of Theorem~\ref{t:main}. Then the $n$-th Quillen cohomology group can be identified as
$$
\rH^n_Q(\CC;M_{\F})\cong \pi_{-n-2}\big(\lim{}_{\uline{\Tw}_2(\CC)}\F\big).
$$
In particular, if $A: \uline{\Tw}_2(\CC) \lrar \Ab$ is a diagram of abelian groups, then $\rH^n_Q(\CC;M_{\H A})$ is naturally isomorphic to the $(n+2)$-th derived functor $\RR^{n+2}\lim_{\uline{\Tw}_2(\CC)}(A)$.
\end{cor}
\begin{proof}
By definition we have $\rH^n_Q(\CC;M_{\F}) = \big[L_{\CC},M_{\F}[n]\big]_{\T_{\CC}\Cat^+_\Del}$. By Theorem~\ref{t:main} this can be identified with
$$ \big[\ovl{\mathbb{S}}[-2],\F[n]\big]_{\Fun(\uline{\Tw}_2(\CC),\Sp)} \cong \big[\ovl{\mathbb{S}}[-n-2],\F\big]_{\Fun(\uline{\Tw}_2(\CC),\Sp)} \cong \pi_{-n-2}\left(\lim_{\uline{\Tw}_2(\CC)}\F\right) $$
where $\ovl{\mathbb{S}}$ denotes the constant diagram with value the sphere spectrum.
\end{proof}
\begin{proof}[Proof of Proposition \ref{p:cotangent}]
Let us start by treating the special case where $\CC=[0]$ is the terminal marked-simplicial category. In that case, $\uline{\Tw}_2([0])=\ast$ is terminal as well, and we can identify $\Sp^\ast$ with the stable model structure on $\NN\times \NN$-diagrams of pointed simplicial sets. Let us denote the terminal \emph{marked} simplicial set by $\Del^0$ (to avoid confusion with the terminal simplicial set $\ast$). In this case, the functor $U_{[0]}$ can be identified with the composite
$$\xymatrix{
U_{[0]}: \T_{[0]}\Cat_\Del^+ \ar[r]^-{\G^{\Sp}_{[0]}} & \T_{\Del^0}\Set^+_\Del\ar[r]_-{\ref{p:not-there-yet}}^-{\R^{\Sp}_{\Del^0}} & \Sp^{\Del^0} \ar[r]_-{\ref{c:last}}^-{\uline{\Un}^{\cov}} & \Sp^\ast=\Sp((\Set_\Del)^{\cov}_{*//*})=\Sp((\Set_\Del^{\KQ})_\ast).
}$$
Here the functor $\G^{\Sp}_{[0]}$ is the right Quillen equivalence of \cite[Corollary 3.1.16]{part2}, which sends an $\NN\times\NN$-diagram of pointed marked-simplicial categories $[0]\x{\ast}{\lrar} \DD_{\bullet\bullet}\lrar [0]$ to the diagram of pointed marked simplicial sets $\Map_{\DD_{\bullet\bullet}}(\ast, \ast)$. By~\cite[Proposition 3.2.1]{part2}, there is a weak equivalence
\begin{equation}\label{e:G} 
\eta:L_{\Del^0}[-1]\x{\simeq}{\lrar} \RR\G^{\Sp}_{[0]}(L_{[0]}) \in \T_{\Del^0}\Set^+_\Del
\end{equation}
between the (derived) image of the cotangent complex of $[0]$ and the desuspension of the cotangent complex of the marked simplicial set $\Del^0$. To compute this cotangent complex, recall from \S\ref{s:scaledsset} that the functor $(-)^\sharp: \Set_\Del^{\KQ}\lrar \Set_\Del^+$ is a left Quillen functor. Since left Quillen functors preserve cotangent complexes, we conclude that $L_{\Del^0}$ is the image of the cotangent complex of the point in $\Set_\Del^{\KQ}$, which is $\SS^\sharp$.

Since $\SS^\sharp$ is a fibrant object of $\T_{\Del^0}\Set^+_\Del$, we have that
\begin{equation}\label{e:R-1} 
\RR \R^{\Sp}_{\Del^0}(L_{\Del^0}) \simeq \R^{\Sp}_{\Del^0}(\SS^{\sharp}) \simeq \Tw^+(\SS^{\sharp}) \times_{\SS^{\sharp} \times \SS^{\sharp}} (\Del^0 \times \Del^0) \simeq \Om\SS^{\sharp} \simeq \SS^{\sharp}[-1]
\end{equation}
where the pullback and looping are computed degreewise. Finally, the unstraightening $\uline{\Un}^{\cov}: (\Set^+_\Del)^{\cov}\lrar \Set_\Del^{\KQ}$ is naturally equivalent to the functor forgetting the marked edges by \cite[Proposition 3.6.1]{goodwillie}. It follows that there is a weak equivalence
$$
\theta_{[0]}: \SS[-2] \x{\simeq}{\lrar} \uline{\Un}^{\cov} \R^{\Sp}_{\Del^0}(\SS^\sharp[-1]) \simeq \RR \uline{\Un}^{\cov}\RR\R^{\Sp}_{\Del^0}(L_{\Del^0}) \x{\simeq}{\lrar} \RR U_{[0]}(L_{[0]})
$$
where the last equivalence is induced by the equivalence $\eta$ of \eqref{e:G}.

For a general fibrant marked-simplicial category $\CC$, let $p: \CC \lrar [0]$, $q: \rN^{\sca}(\CC) \lrar \Del^0$ and $r: \uline{\Tw}_2(\CC)\lrar \ast$ be the terminal maps. We then obtain a commuting diagram of right Quillen functors
$$ \xymatrix{
\T_{\Del^0}\Set^{\sca}_{\Del} \ar[d]_-{q^*} & \T_{[0]}\Cat^+_{\Del}\ar[l]^-\simeq_-{\rN^{\sca}}\ar[r]_-{\simeq}^-{U_{[0]}}\ar[d]^-{p^*} & \Sp^\ast\ar[d]^-{r^*}\\
\T_{\rN^{\sca}(\CC)}\Set^{\sca}_{\Del} & \T_{\CC}\Cat^+_{\Del} \ar[l]_-\simeq^-{\rN^{\sca}}\ar[r]^-{\simeq}_-{U_{\CC}} & \Sp^{\uline{\Tw}_2(\CC)}.
}$$
All vertical functors take pullbacks of parameterized spectrum objects along the indicated maps. The horizontal functors are all right Quillen equivalences (the left horizontal functors take scaled nerves). By \cite[Lemma 4.2.6]{goodwillie}, the bicategorical model structure on $\Set^{\sca}_{\Del}$ is Cartesian closed, so that the functor $q^*: \Set^{\sca}_{\Del}\lrar (\Set^{\sca}_{\Del}){}_{/\rN^{\sca}(\CC)}$ is also a left Quillen functor. It follows that $q^*$ maps the cotangent complex of $\Del^0$ in $\T_{\Del^0}\Set^{\sca}_{\Del}$ to the cotangent complex of $\rN^{\sca}(\CC)$ in $\T_{\rN^{\sca}(\CC)}\Set^{\sca}_{\Del}$. Since $r^*$ is conjugate to $q^*$ via Quillen equivalences, it follows that $r^*$ sends $\RR U_{[0]}(L_{[0]})$ to $\RR U_{\CC}(L_{\CC})$. The desired equivalence therefore arises from the equivalence $\theta_{[0]}: \SS[-2]\lrar \RR U_{[0]}(L_{[0]})$.
\end{proof}

It will be useful to record the following enhanced version of Proposition~\ref{p:cotangent}, which allows one to compute relative cotangent complexes as well. Let $f: \CC \lrar \DD$ be a map of fibrant marked-simplicial categories and let $\vphi: \uline{\Tw}_2(\CC)\lrar \uline{\Tw}_2(\DD)$ be the induced functor on twisted 2-cell $\infty$-categories. Theorem~\ref{t:main2} gives a commutative square of Quillen adjunctions
\begin{equation}\label{e:square-tangent}
\vcenter{\xymatrix@C=3pc{
\T_\CC\Cat^+_\Del\ar@<-1ex>[d]_-{f_!}\ar@<-1ex>[r]_-{U_{\CC}}^-{\upvdash}  &
\Sp^{\uline{\Tw}_2(\CC)} \ar@<-1ex>[l]_-{\F_\CC}\ar@<-1ex>[d]_-{\vphi_!} \\
\T_\DD\Cat^+_\Del\ar@<-1ex>[u]^-{\dashv}_{f^*}\ar@<-1ex>[r]_-{U_\DD}^-{\upvdash} &
\Sp^{\uline{\Tw}_2(\DD)}\ar@<-1ex>[l]_-{\F_\DD}\ar@<-1ex>[u]^-{\dashv}_{\vphi^*} 
}}
\end{equation}
where the horizontal Quillen adjunctions are Quillen equivalences and the functors $f^*$ and $\vphi^*$ take the pullback of a parameterized spectrum (of marked simplicial categories, resp.\ left fibrations) along $f$ and $\vphi$. We then have the following:
\begin{cor}\label{c:enhanced}
Let $f: \CC \lrar \DD$ be a map of fibrant marked-simplicial categories and let $r: \uline{\Tw}_2(\CC) \lrar \ast$ denote the terminal map.
Then there is a natural weak equivalence in $\Sp^{\uline{\Tw}_2(\DD)}$
$$ \theta_{f}:\LL\vphi_!(r^*\SS[-2]) \x{\simeq}{\lrar} \RR U_{\DD}\LL f_!(L_{\CC}).$$
\end{cor}
\begin{rem}\label{r:conceptual}
Under the equivalences of $\infty$-categories 
$$
\big(\Sp^{\uline{\Tw}_2(\CC)}\big)_{\infty} \simeq \Fun(\uline{\Tw}_2(\CC),\Sp)\qquad\qquad \big(\Sp^{\uline{\Tw}_2(\DD)}\big)_{\infty} \simeq \Fun(\uline{\Tw}_2(\DD),\Sp)
$$
the functors $\vphi^*$ and $\vphi_!$ correspond to restriction and left Kan extension along $\vphi$. Corollary~\ref{c:enhanced} should hence be read as follows: given a map $f: \CC \lrar \DD$, the suspension spectrum of the object $\CC \in (\Cat^+_\Del)_{/\DD}$ corresponds, under the equivalence of Theorem~\ref{t:main}, to the left Kan extension of the constant functor $\SS[-2]:\uline{\Tw}_2(\CC) \lrar \Sp$ along the induced map $\vphi: \uline{\Tw}_2(\CC) \lrar \uline{\Tw}_2(\DD)$.
\end{rem}
\begin{proof} 
Proposition~\ref{p:cotangent} provides a natural weak equivalence $\theta_\CC:r^*\SS[-2] \x{\simeq}{\lrar} U_\CC(L_{\CC})$. Since $F_{\CC} \dashv U_{\CC}$ is a Quillen equivalence, this map is adjoint to a weak equivalence $\theta^{\ad}_{\CC}:\LL\F_{\CC}(r^*\SS) \x{\simeq}{\lrar} L_{\CC}[2]$. Using the commutativity of~\eqref{e:square-tangent} we obtain a natural weak equivalence
$$\xymatrix@C=3pc{ 
\LL\F_{\DD}\LL\vphi_!(r^*\SS) \simeq \LL f_!\LL\F_{\CC}(r^*\SS) \ar[r]^-{\LL f_!\theta^{\ad}_{\CC}}_-{\simeq} & \LL f_!(L_\CC[2])}.$$
The equivalence $\theta_{f}$ is the weak equivalence which is adjoint to this map under the Quillen equivalence $\F_{\DD} \dashv U_{\DD}$.
\end{proof}
\begin{cor}\label{c:relative-2}
Let $f: \CC \lrar \DD$ be a map of marked-simplicial categories.
Then there is a natural homotopy cofiber sequence in $\Sp^{\uline{\Tw}_2(\DD)}$
\begin{equation}\label{e:relative-2}
\uline{\Tw}_2(\CC) \times\SS \lrar \uline{\Tw}_2(\DD) \times \SS \lrar U_{\DD}(L_{\DD/\CC}[2]) 
\end{equation}
\end{cor}
\begin{proof}
By Corollary~\ref{c:enhanced} the left term of the above sequence can be identified with $U_\DD\LL f_!(L_{\CC}[2])$, while the middle term is given by $U_{\DD}(L_{\DD}[2])$ by Proposition \ref{p:cotangent}. This identifies the above sequence with the image of the cofiber sequence $\LL f_!(L_{\CC}[2])\lrar L_{\DD}[2] \lrar L_{\DD/\CC}[2]$ under $U_{\DD}$.
\end{proof}
The cofiber sequence~\eqref{e:relative-2} can also be rewritten as 
$$
\Sig^{\infty}_+(\vphi) \lrar \Sig^{\infty}_+(\Id_{\uline{\Tw}_2(\DD)}) \lrar U_{\DD}(L_{\DD/\CC}[2]) 
$$
Recall that a map $p: X \lrar Y$ of simplicial sets is said to be \textbf{coinitial} if $p^{\op}$ is cofinal, i.e., if $p$ is equivalent to the terminal object in $(\Set_{\Del})^{\cov}_{/Y}$ (cf.~\cite[Definition 4.1.1.1]{Lur09}). We may therefore conclude the following:
\begin{cor}\label{c:coinitial}
Let $f: \CC \lrar \DD$ be a map of fibrant marked-simplicial categories such that the induced map $\vphi:\uline{\Tw}_2(\CC) \lrar\uline{\Tw}_2(\DD)$ is coinitial. Then the relative cotangent complex of $f$ vanishes. In particular, for any coefficient system $M\in \T_{\DD}\Cat_\Del^+$, the relative Quillen cohomology groups vanish:
$$
\rH_{Q}^n(\CC, \DD; M) \cong 0.
$$
\end{cor}
\begin{rem}\label{r:coinitial}
The notion of coinital appears in the literature under various names, including \textit{right cofinal}, and \textit{initial}. By the $\infty$-categorical Quillen theorem A (see, e.g., \cite[Theorem 4.1.3.1]{Lur09}) a map $p: X \lrar Y$ from a simplicial set to an $\infty$-category is coinitial if and only if for every object $y \in Y$ the simplicial set $X \times_{Y} Y_{/y}$ is weakly contractible.
\end{rem}

\section{Application - the classification of adjunctions}\label{s:adj}
In this section we will demonstrate the above machinery on a particular example, by showing that the inclusion of $2$-categories $\iota:[1] \lrar \Adj$ has a trivial relative cotangent complex. Here $\Adj$ is the walking adjunction and $[1] = \bullet \lrar \bullet$ is considered as a $2$-category with no non-trivial $2$-cells. If $\CC$ is a fibrant marked-simplicial category then the data of a functor $\Adj \lrar \CC$ is equivalent to the data of a homotopy coherent adjunction in $\CC$, while functors $[1] \lrar \CC$ classify $1$-arrows in $\CC$. 

The triviality of the relative cotangent complex of $\iota:[1] \lrar \Adj$ means that the relative Quillen cohomology groups $\rH^n_Q(\Adj,[1];M)$ vanish for every coefficient object $M \in \T_{\Adj}\Cat^+_\Del$ (see \S\ref{s:recall}). By the obstruction theory mentioned in~\S\ref{s:intro} (see also~\cite[2.6]{part2} and~\cite{part4}) this means that a $1$-arrow $f$ in a fibrant marked-simplicial category $\CC$ extends to a homotopy coherent adjunction if and only if it extends to an adjunction in the homotopy $(3,2)$-category $\Ho_{\leq 3}(\CC)$. In fact, the space of derived lifts in the square
$$ \xymatrix{
[1] \ar[r]\ar[d] & \CC \ar[d] \\
\Adj \ar[r]\ar@{-->}[ur] & \Ho_{\leq 3}(\CC) \\
}$$
is weakly contractible. 
We note that the analogous contractibility statement for lifts of $[1] \lrar \Adj$ against $\CC \lrar \Ho_{\leq 2}(\CC)$ was established in~\cite{RV16} by using a somewhat elaborate combinatorial argument and an explicit cell decomposition of $\Adj$. As we hope to demonstrate below, the argument concerning the relative cotangent complex of $[1] \lrar \Adj$ is rather simple in comparison. 
Recall that $\Adj$ contains two objects $0,1 \in \Adj$, its $1$-morphisms are freely generated by a morphism $f: 0 \lrar 1$ (the left adjoint) and a morphism $g: 1 \lrar 0$ (the right adjoint) and its $2$-morphisms are generated (via both horizontal and vertical compositions) by a unit $2$-cell $u: \Id_0 \Rightarrow T := gf$ and counit $2$-cell $v:K := fg \Rightarrow \Id_1$ subject to the relations that the compositions 
$$
(vf)\circ (fu): f\Rightarrow fgf \Rightarrow f\qquad\qquad (gv)\circ (ug): g \Rightarrow gfg \Rightarrow g
$$
are equal to the identity $2$-cells. Our goal in this section is then to prove the following:
\begin{thm}\label{t:adj}
Let $\iota: [1] \lrar \Adj$ be the inclusion which sends the non-trivial morphism of $[1]$ to $f$. Then the map
$$ \iota_*:\uline{\Tw}_2([1]) \lrar \uline{\Tw}_2(\Adj) $$
induced by $\iota$ is coinitial. In particular (see Corollary~\ref{c:coinitial}), the relative cotangent complex of $\iota$ is trivial.
\end{thm}
Let us start by describing the mapping categories of $\Adj$ in more detail. 
\begin{define}\label{d:ordinals}
Let us denote the finite ordinal of size $n$ by $\<n\> = \{0,...,n-1\}$. For $x,y \in \{0,1\}$, let $\Del_{x,y}$ be the following category of \textbf{$(x, y)$-ordinals}:
\begin{itemize}
 \item objects given by finite ordinals with at least $\min(x, y)$ elements.
 \item maps given by order-preserving maps that preserve the initial $x$ elements and the final $y$ elements (i.e.\ no further condition when $x=y=0$).
\end{itemize}
For $x,y,z\in \{0,1\}$, consider the functor 
$$
\otimes_y:\Del_{x,y} \times \Del_{y,z} \lrar \Del_{x,z}; \hspace{4pt} \<n\> \otimes_y \<m\> = \<n-y+m\>
$$
which concatenates $\<n\>$ and $\<m\>$ and identifies the final element of $\<n\>$ with the initial element of $\<m\>$ if $y=1$. 
\end{define}

\begin{obs}[cf.\cite{RV16}]\label{c:identification}
There is a natural identification $\Map_{\Adj}(x,y) \cong \Del_{x,y}$ such that the composition functors $\Map_{\Adj}(x,y) \times \Map_{\Adj}(y,z) \lrar \Map_{\Adj}(x,z)$ are given by $\otimes_y$.
\end{obs}

Recall that $\Adj$ admits a natural duality functor $\Adj \lrar \Adj^{\coop}$, where the directions of $1$-morphisms and $2$-morphisms are reversed in $ \Adj^{\coop}$. This functor switches $0$ with $1$, $f$ with $g$ and $u$ with $v$. In terms of Definition \ref{d:ordinals}, this functor can be described as follows:
\begin{define}
Let $x,y \in \{0,1\}$ and let $\<n\> \in \Del_{x,y}$ be an $(x,y)$-ordinal. A \textbf{gap} in $\<n\>$ is a map of $(x,y)$-ordinals $g: \<n\> \lrar \<2\> = \{0,1\}$. 
We denote by $\widehat{\<n\>}$ the linear order of gaps in $\<n\>$, where $g\leq g'$ if $g^{-1}(0) \subseteq (g')^{-1}(0)$.
\end{define}
\begin{rem}
The notation $\widehat{\<n\>}$ is slightly abusive: it does not reflect the dependency of the notion of a gap in $x$ and $y$. 
\end{rem}

\begin{obs}\label{o:gap}
Let $x,y \in \{0,1\}$ be elements. Then the association $\<n\right> \mapsto \widehat{\<n\right>}$ maps $(x,y)$-ordinals contravariantly to $(1-x,1-y)$-ordinals and determines an equivalence of categories 
\begin{equation}\label{e:equiv}
\Del_{x,y} \x{\simeq}{\lrar} (\Del_{1-x,1-y})^{\op}.
\end{equation}
Under the identification of Observation \ref{c:identification}, these equivalences describe the canonical duality functor $\Adj\lrar\Adj^{\coop}$.
\end{obs}

By Proposition~\ref{p:2-cat} and Observation \ref{c:identification}, the twisted $2$-cell $\infty$-bicategory of $\Adj$ can be modeled by the Grothendieck construction
\begin{equation}\label{e:groth-tw} 
\displaystyle\mathop{\int}_{(x,y) \in \Adj^{\op}_{\Tw} \times \Adj_{\Tw}} \Map_{\Adj_{\Tw}}(x,y)\x{\ref{c:identification}}{=} \displaystyle\mathop{\int}_{(x,y) \in \Adj^{\op}_{\Tw} \times \Adj_{\Tw}} \Tw(\Del_{x, y}).
\end{equation}
For the remainder of this section we will therefore just take~\eqref{e:groth-tw} as the \textbf{definition} of $\Tw_2(\Adj)$. 
In particular, we may represent objects in $\Tw_2(\Adj)$ as tuples $(x,y,\sig)$ where $x,y$ are objects of $\Adj$ and $\sig \in \Tw(\Del_{x, y})$ is a map of $(x, y)$-ordinals $\sigma:\<n\> \lrar \<m\>$, describing a $2$-cell between two $1$-morphisms from $x$ to $y$. Since the mapping categories of $\Adj_{\Tw}$ are all of the form $\Tw(\Del_{x, y})$, the mapping categories in $\Tw_2(\Adj)$ are then given by the Grothendieck construction
\begin{equation}\label{e:groth}
\Map_{\Tw_2(\Adj)}((x,y,\sig),(x',y',\sig')) = \hspace{-4pt}\displaystyle\mathop{\int}_{\substack{\vphi \in \Tw(\Del_{x',x}) \\ \psi \in \Tw(\Del_{y,y'})} }\hspace{-4pt}
\Map_{\Tw(\Del_{x',y'})}(\vphi \otimes_x \sig \otimes_y \psi,\sig') .
\end{equation}
By Remark~\ref{r:freeoo1}, the twisted $2$-cell \textbf{$\infty$-category} of $\Adj$ 
is equivalent to (the coherent nerve of) the simplicial category obtained from $\Tw_2(\Adj)$ by replacing each mapping category with its classifying space. On the other hand, since $[1]$ is a $2$-category with no non-trivial $2$-cells it follows from Example~\ref{e:1-cat} that the twisted $2$-cell $\infty$-category of $[1]$ is equivalent to its ordinary twisted arrow category $\Tw([1]) = \bullet \lrar \bullet \llar \bullet$. 
Theorem~\ref{t:adj} then follows from the following weak contractibility statement:
\begin{pro}\label{p:adj-2}
Let $(x,y,\sig) \in \Tw_2(\Adj)$ be an object. Then the nerve of the $1$-category
\begin{equation}\label{e:cont}
\Tw([1])_{/(x,y,\sig)} := \int_{e \in \Tw([1])^{\op}} \Map_{\Tw_2(\Adj)}(\iota_*(e),(x,y,\sig)) 
\end{equation}
is weakly contractible.
\end{pro}
\begin{proof}[Proof of Theorem \ref{t:adj} assuming Proposition \ref{p:adj-2}]
Let us start with the following observation: let $\CC\x{f}{\lrar} \DD\x{\sim}{\lrar} \DD'$ be a diagram of simplicial categories, where $\CC$ is a discrete category and $\DD\lrar \DD'$ replaces each mapping object by a weakly equivalent Kan complex. Fix an object $d\in \DD$, which we can equivalently consider as an object $d\in \DD'$ or an object $d\in \rN(\DD')$. Then the classifying space $|\rN(\CC)\times_{\rN(\DD')} \rN(\DD')_{/d}|$ is a model for the colimit of the restriction of the representable functor $\Map_{\rN(\DD')}(-, d)$ along $\rN(\CC)\lrar \rN(\DD')$. By \cite[Theorem 4.2.4.1]{Lur09}, this space can be modeled by the homotopy colimit
$$
\hocolim_{c\in \CC^{\op}} \Map_{\DD}(f(-), d) \simeq \hocolim_{c\in \CC^{\op}} \Map_{\DD'}(f(-), d).
$$
Now consider the case where $\CC=\Tw[1]$ and $\DD=\Tw_2(\Adj)_\rN$ is obtained by taking the nerves of all mapping categories in $\Tw_2(\Adj)$. The $\infty$-functor $\rN(\CC)\lrar \rN(\DD')$ is then equivalent to the functor $\iota_*: \Tw([1])\lrar \uline{\Tw}_2(\Adj)$. The above homotopy colimit is equivalent to the nerve of the category \eqref{e:cont} and is hence contractible by Proposition \ref{p:adj-2}. By Remark \ref{r:coinitial}, the $\infty$-functor $\iota_*$ is now coinitial, so that Theorem \ref{t:adj} follows from Corollary \ref{c:coinitial}.
\end{proof}
The remainder of the section is devoted to the proof of Proposition~\ref{p:adj-2}. Fix $x,y \in \{0,1\} = \Obj(\Adj)$ and let $\sig \in \Tw(\Map_{\Adj}(x,y)) \cong \Tw(\Del_{x,y})$ be a map of $(x,y)$-ordinals $\sig: \<n\> \lrar \<m\>$. Consider the object $\Id_0:0 \lrar 0$ of $\Tw([1])$. By~\eqref{e:groth}, the mapping category $\Map_{\Tw_2(\Adj)}(\iota_*(\Id_0),\sig)$ can be identified with the Grothendieck construction
$$ \displaystyle\mathop{\int}_{\substack{\vphi \in \Tw(\Del_{x,0}) \\ \psi \in \Tw(\Del_{0,y})}}  
\Map_{\Tw(\Del_{x,y})}(\vphi \otimes_0 \psi,\sig) \cong \big(\Tw(\Del_{x,0}) \times \Tw(\Del_{0,y})\big) \times_{\Tw(\Del_{x,y})} \big(\Tw(\Del_{x,y})_{/\sig}\big).
$$
This is just the comma category of the concatenation functor $\otimes_0: \Tw(\Del_{x,0}) \times \Tw(\Del_{0,y}) \lrar \Tw(\Del_{x,y})$ over $\sig \in \Tw(\Del_{x,y})$. A similar unfolding shows that we can identify $\Map_{\Tw_2(\Adj)}(\iota_*(\Id_1),\sig)$ with the comma category of the functor $\otimes_1: \Tw(\Del_{x,1}) \times \Tw(\Del_{1,y}) \lrar \Tw(\Del_{x,y})$ over $\sig$. 

Finally, if $e:0 \lrar 1$ is the non-identity arrow of $[1]$ then the mapping category $\Map_{\Tw_2(\Adj)}(\iota_*(e),\sig)$ identifies with the comma category over $\sig$ of the functor $\Tw(\Del_{x,0}) \times \Tw(\Del_{1,y}) \lrar \Tw(\Del_{x,y})$ given by $(\<n\>,\<m\>) \mapsto \<n\> \otimes_0 \<1\> \otimes_1 \<m\> \cong \<n+m\>$. To describe these various products of twisted arrow categories concisely, let us introduce the following terminology:

\begin{define}\label{d:pgs}
Let $x,y \in \{0,1\}$ be fixed numbers. A \textbf{gapped ordinal} is an object of the over category $\Del_{\gp} := (\Del_{x,y})_{/\<2\>}$, i.e., a pair $(\left<n\right>,g)$ where $\<n\> \in \Del_{x,y}$ is an $(x,y)$-ordinal and $g: \<n\> \lrar \<2\>$ is a gap in $\<n\>$. A \textbf{pointed ordinal} is an object of the under category $\Del_{\pt} := (\Del_{x,y})_{\<x+1+y\>/}$, i.e., a pair $(\<n\>,i)$ where $\<n\> \in \Del_{x,y}$ is an $(x,y)$-ordinal and $i:\<x+1+y\> \lrar \<n\>$ can be identified with an element $i \in \<n\> = \{0,...,n-1\}$. Finally, a \textbf{split ordinal} is a triple $(\<n\>,g,i)$ where $(\<n\>,i) \in \Del_{\pt}$ is a pointed ordinal and $g \in \widehat{\<n\>}$ is a gap such that $i$ is a minimal element of $g^{-1}(1)$. The split ordinals form a full subcategory $\Del_{\spl} \subseteq \Del_{\gp} \times_{\Del_{x,y}} \Del_{\pt}$.
\end{define}

\begin{rem}\label{r:adjoint}
The forgetful functor $\Del_{\spl} \lrar \Del_{\gp}$ admits a \textbf{left adjoint} which sends a gapped ordinal $(\<n\>,g)$ to the split ordinal $(\<n\> \cup \{a\},a,g_a)$, where $\<n\> \cup \{a\}$ is the ordinal obtained by adding to $\<n\>$ a new element $a$ and setting the order to be such that $a$ is bigger then all the elements in $g^{-1}(0)$ and smaller than all the elements in $g^{-1}(1)$. The new gap $g_a: \<n\> \cup \{a\} \lrar \<2\>$ extends $g$ by setting $g_a(a) = 1$. Similarly, the forgetful functor $\Del_{\spl} \lrar \Del_{\pt}$ admits a \textbf{right adjoint} which sends a pointed ordinal $(\<m\>,j)$ to the split ordinal $(\<m\> \cup \{b\},b,g_b)$ where $\<m\> \cup \{b\}$ is obtained by adding to $\<m\>$ a new element $b$ and setting the order to be such that $b$ is the smallest element which is bigger than $j \in \<m\>$. The gap $g_b: \<m\> \cup \{b\} \lrar \<2\>$ is defined so that $b$ is the minimal element of $g_b^{-1}(1)$. 
\end{rem}
The types of gapped, pointed and split ordinals we will come across will mostly be of the following forms: 
\begin{cons}\label{c:point-gap}
Given two ordinals $\<n\> \in \Del_{x,0},\<m\> \in \Del_{0,y}$, the concatenation $\<n\> \otimes_0 \<m\> \in \Del_{x,y}$ comes equipped with a natural gap $g: \<n\> \otimes_0 \<m\> \lrar \<2\>$ which is obtained by applying the functor $\otimes_0$ to the terminal maps $\<n\> \lrar \<1\>$ and $\<m\> \lrar \<1\>$. Explicitly, $g$ sends the first $n$ elements of $\<n\> \otimes_0 \<m\>$ to $0$ and the last $m$ elements of $\<n\> \otimes_0 \<m\>$ to $1$. Similarly, for $\<n\> \in \Del_{x,1}, \<m\> \in \Del_{1,y}$ the ordinal $\<n\> \otimes_1 \<m\> \in \Del_{x,y}$ comes equipped with a distinguished base point: the map $\<x+1+y\> \lrar \<n\> \otimes_1 \<m\>$ obtained by applying the functor $\otimes_1$ to the initial maps $\<x+1\> \lrar \<n\>$ and $\<1+y\> \lrar \<m\>$. More explicitly, this base point is the element $n-1$ in $\<n\> \otimes_1 \<m\> = \{0,...,n+m-1\}$. Finally, if we take an object $\<n\> \in \Del_{x,0}$ and an object $\<m\> \in \Del_{1,y}$ then $\<n\> \otimes_0 \<1\> \otimes_1 \<m\>$ is naturally split. It contains both a natural base point induced from the initial maps $\<x\> \lrar \<n\>, \<1\> \lrar \<1\>$ and $\<1+y\> \lrar \<m\>$ and a natural gap $g: \<n\> \otimes_0 \<1\> \otimes_1 \<m\> \lrar \<2\>$ induced from the terminal maps $\<n\> \lrar \<1\>, \Id: \<1\> \lrar \<1\>$ and $\<m\> \lrar \<1\>$.
\end{cons}

\begin{lem}
The functors $\Del_{x,0} \times \Del_{0,y} \lrar \Del_{\gp}$, $\Del_{x,1} \times \Del_{1,y} \lrar \Del_{\pt}$ and $\Del_{x,0} \times \Del_{1,y} \lrar \Del_{\spl}$ described in Construction~\ref{c:point-gap} are equivalences of categories.
\end{lem}
\begin{proof}
The functor $(\<n\>,g) \mapsto (g^{-1}(0),g^{-1}(1))$ is inverse to the first functor, the functor $(\<n\>,i) \mapsto (\{j \in \<n\> | j \leq i\},\{j \in \<n\> | j \geq i\})$ is inverse to the second functor and the functor $(\left<n\>,g,i) \mapsto (g^{-1}(0),\{j \in \<n\> | j \geq i\})$ is inverse to the third.
\end{proof}

\begin{cor}\label{c:step-1}
Let $\sig: \<n\> \lrar \<m\>$ be a map of ordinals, considered as a $2$-cell in $\Adj$ from $\<n\> : x \lrar y$ to $\<m\>: x \lrar y$. Then we have natural equivalences of categories
$$ \Map_{\Tw_2(\Adj)}(\iota_*(\Id_0),\sig) \simeq \Tw(\Del_{\gp})_{/\sig} := \Tw(\Del_{\gp}) \times_{\Tw(\Del_{x,y})} \Tw(\Del_{x,y})_{/\sig}, $$
$$ \Map_{\Tw_2(\Adj)}(\iota_*(\Id_1),\sig) \simeq \Tw(\Del_{\pt})_{/\sig} := \Tw(\Del_{\pt}) \times_{\Tw(\Del_{x,y})} \Tw(\Del_{x,y})_{/\sig} $$
and
$$ \Map_{\Tw_2(\Adj)}(\iota_*(e),\sig) \simeq \Tw(\Del_{\spl})_{/\sig} := \Tw(\Del_{\spl}) \times_{\Tw(\Del_{x,y})} \Tw(\Del_{x,y})_{/\sig} .$$
\end{cor}

\begin{rem}
Under the equivalences of Corollary~\ref{c:step-1} the maps from $\Map_{\Tw_2(\Adj)}(\iota_*(e),\sig)$ to $\Map_{\Tw_2(\Adj)}(\iota_*(\Id_0),\sig)$ and $\Map_{\Tw_2(\Adj)}(\iota_*(\Id_1),\sig)$ obtained by restricting along the morphisms $\Id_0 \lrar e,\Id_1 \lrar e$ in $\Tw([1])$ correspond to the maps induced by the natural projections $\Del_{\spl} \lrar \Del_{\gp}$ and $\Del_{\spl} \lrar \Del_{\pt}$.
\end{rem}

Consider the forgetful functor $\Del_{\pt} = (\Del_{x,y})_{\<x+1+y\>/} \lrar \Del_{x,y}$. This is a left fibration, and the fiber $(\Del_{\pt})_{\<m\>}$ over the $(x,y)$-ordinal $\<m\>$ is the set of possible base points $\Map_{\Del_{x,y}}(\<x+1+y\>,\<m\>) = \{0,...,m-1\}$. 
Let $(\Del_{\pt})_{/\<m\>} := \Del_{\pt} \times_{\Del_{x,y}} (\Del_{x,y})_{/\<m\>}$ be the associated comma category. Then we have a natural functor $(\Del_{\pt})_{/\<m\>} \lrar (\Del_{\pt})_{\<m\>}$ which sends a pair $((\<k\>,i),\vphi:\<k\> \lrar \<m\>)$ to the element $\vphi(i) \in (\Del_{\pt})_{\<m\>}$. Similarly, $\Del_{\gp} \lrar \Del_{x,y}$ is a right fibration, the fiber $(\Del_{\gp})_{\<n\>}$ is the set $\widehat{\<n\>} = \Map_{\Del_{x,y}}(\<n\>,\<2\>)$ of gaps in $\<n\>$, and we have a natural functor $((\Del_{\gp})^{\op})_{/\<n\>} \lrar (\Del_{\gp})_{\<n\>}$ obtained by pulling back the gap.
 
\begin{define}\label{d:comp}
Let $\sig: \<n\> \lrar \<m\>$ be a map in $\Del_{x,y}$. We will say that an element $j \in \<m\>$ is \textbf{compatible} with a gap $g \in \widehat{\<n\>}$ if the following condition holds: for any $i \in \<n\>$ such that $\sig(i) < j$ we have $g(i) = 0$ and for any $i \in \<n\>$ such that $\sig(i) > j$ we have $g(i) = 1$. We will denote by 
$$ \E_{\sig} \subseteq \widehat{\<n\>} \times \<m\> $$ 
the subset consisting of those pairs $(g,j)$ such that $j$ is compatible with $g$.
\end{define}

The following proposition will play a key role in the proof of Proposition~\ref{p:adj-2}.
\begin{pro}\label{l:step-2}
Let $\sig: \<n\> \lrar \<m\>$ be a map in $\Del_{x,y}$. Then the following holds:
\begin{enumerate}
\item
The composed functor $\Tw(\Del_{\gp})_{/\sig} \lrar ((\Del_{\gp})^{\op})_{/\<n\>} \lrar (\Del_{\gp})_{\<n\>}$ induces a weak equivalence on nerves.
\item
The composed functor $\Tw(\Del_{\pt})_{/\sig} \lrar (\Del_{\pt})_{/\<m\>} \lrar (\Del_{\pt})_{\<m\>}$ induces a weak equivalence on nerves.
\item
The composed functor 
\begin{equation}\label{e:F} 
\Tw(\Del_{\spl})_{/\sig} \lrar \Tw(\Del_{\gp})_{/\sig} \times \Tw(\Del_{\pt})_{/\sig} \lrar (\Del_{\gp})_{\<n\>} \times (\Del_{\pt})_{\<m\>} = \widehat{\<n\>} \times \<m\> 
\end{equation}
induces a weak equivalence $\rN(\Tw(\Del_{\spl})_{/\sig}) \x{\simeq}{\lrar} \E_{\sig} \subseteq \widehat{\<n\>} \times \<m\>$.
\end{enumerate}
\end{pro}
\begin{proof}
Let us begin with Claim (1). We will depict objects of $\Tw(\Del_{\gp})_{/\sig}$ as commutative diagram
\begin{equation}\label{e:twisted-gap} 
\vcenter{\xymatrix@R=1.8pc{
(\<l\>,g) \ar[d]_{\tau} & \<n\> \ar[l]_-{\vphi}\ar^{\sig}[d] \\
(\<k\>,h) \ar[r]^-{\psi} & \<m\> \\
}}
\end{equation}
where the horizontal arrows indicate maps which are defined just on the underlying ungapped sets.
Let $\A \subseteq \Tw(\Del_{\gp})_{/\sig}$ be the full subcategory spanned by those objects as in\eqref{e:twisted-gap} such that $\vphi:\<n\> \lrar \<l\>$ is an isomorphism. Then the inclusion $\A \subseteq \Tw(\Del_{\gp})_{/\sig}$ admits a left adjoint $\Tw(\Del_{\gp})_{/\sig} \lrar \A$ which sends an object $\Psi$ as in~\eqref{e:twisted-gap} to the object
\begin{equation}\label{e:twisted-gap-2} 
\vcenter{\xymatrix@R=1.8pc{
(\<n\>,\vphi^*g) \ar[d]_{\tau\circ \vphi} & \<n\> \ar[l]_-{\Id_{\<n\>}}\ar^{\sig}[d] \\
(\<k\>,h) \ar[r]^-{\psi} & \<m\> \\
}}
\end{equation}
It then follows that the inclusion of $\A$ induces a weak equivalence $\rN(\A) \x{\simeq}{\lrar} \rN(\Tw(\Del_{\gp})_{/\sig})$ on nerves. We now observe that the category $\A$ decomposes as the disjoint union
$$ \A \cong \coprod_{g' \in \widehat{\<n\>}} \A_{g'} $$
where $\A_{g'}$ is the full subcategory containing those objects as in~\eqref{e:twisted-gap-2} such that $\vphi^*(g) = g'$. 
The restriction of the map $\Tw(\Del_{\gp})_{/\sig} \lrar (\Del_{\gp})_{\<n\>}$ to $\A$ sends $\A_{g'}$ to the gap $g' \in (\Del_{\gp})_{\<n\>} = \widehat{\<n\>}$. It will hence suffice to show that each $\A_{g'}$ is weakly contractible. But this now holds simply because $\A_{g'}$ has an initial object, corresponding to the diagram
\begin{equation}\label{e:twisted-gap-3} 
\vcenter{\xymatrix@R=1.8pc{
(\<n\>,g') \ar[d]_{\Id} & \<n\> \ar[l]_-{\Id}\ar^{\sig}[d] \\
(\<n\>,g') \ar[r]^-{\sig} & \<m\> \\
}}
\end{equation}
Let us now prove Claim (2). The proof is essentially dual to the proof of (1). We will depict objects of $\Tw(\Del_{\pt})_{/\sig}$ as commutative diagrams
\begin{equation}\label{e:twisted-pt} 
\vcenter{\xymatrix@R=1.8pc{
(\<l\>,i) \ar[d]_{\tau} & \<n\> \ar[l]_-{\phi}\ar^{\sig}[d] \\
(\<k\>,j) \ar[r]^-{\psi} & \<m\> \\
}}
\end{equation}
Let $\B \subseteq \Tw(\Del_{\pt})_{/\sig}$ be the full subcategory spanned by those objects as in\eqref{e:twisted-pt} such that $\psi:\<k\> \lrar \<m\>$ is an isomorphism. As in the case of Claim (1) the inclusion $\B \subseteq \Tw(\Del_{\pt})_{/\sig}$ admits a left adjoint $\Tw(\Del_{\pt})_{/\sig} \lrar \B$,  and so induces a weak equivalence $\rN(\B) \x{\simeq}{\lrar} \rN(\Tw(\Del_{\pt})_{/\sig})$ on nerves. We now observe that the category $\B$ decomposes as the disjoint union
$$ \B \cong \coprod_{j' \in \<m\>} \B_{j'} $$
where $\B_{j'}$ is the full subcategory containing those objects 
such that $\psi(j) = j'$, 
and the restriction of the map $\Tw(\Del_{\pt})_{/\sig} \lrar (\Del_{\pt})_{\<m\>}$ to $\B$ sends $\B_{j'}$ to the element $j' \in (\Del_{\pt})_{\<m\>} = \<m\>$. 
Finally, each $\B_{j'}$ has an initial object and is hence weakly contractible. 

We shall now proceed to prove Claim (3). We will depict objects of $\Tw(\Del_{\spl})_{/\sig}$ as commutative diagrams
\begin{equation}\label{e:twisted-sp} 
\vcenter{\xymatrix@R=1.8pc{
(\<l\>,g,i) \ar[d]_{\tau} & \<n\> \ar[l]_-{\vphi}\ar^{\sig}[d] \\
(\<k\>,h,j) \ar[r]^-{\psi} & \<m\> \\
}}
\end{equation}
where the horizontal arrows indicate maps which are defined just on the underlying unpointed ungapped sets. 
Here $(\<l\>,g,i)$ and $(\<k\>,h,j)$ are split ordinals (see Definition~\ref{d:pgs}). In particular, $i$ is the minimal element of $g^{-1}(1)$, and similarly $j$ is the minimal element of $h^{-1}(1)$. The functor~\eqref{e:F} sends a diagram as in~\eqref{e:twisted-sp} to the pair $(\vphi^*g,\psi(j))$. Now the element $\psi(j) \in \<m\>$ is compatible with the gap $\vphi^*g \in \widehat{\<n\>}$ in the sense of Definition~\ref{d:comp}: indeed, if $i' \in \<n\>$ is such that $\sig(i') < \psi(j)$ then necessarily $\vphi(i') < i$ and so $\vphi^*g(i') = g(\vphi(i')) = 0$. Similarly, if $i' \in \<n\>$ is such that $\sig(i') > \psi(j)$ then necessarily $\vphi(i') > i$ and so $\vphi^*g(i') = g(\vphi(i')) = 1$. In particular, the image of~\eqref{e:F} is contained in $\E_\sig$. We now observe that the category $\Tw(\Del_{\spl})_{/\sig}$ splits as a disjoint union
$$ \Tw(\Del_{\spl})_{/\sig} = \coprod_{(g',j') \in \E_{\sig}} \C_{(g',j')} $$
where $\C_{(g',j')}$ denote the full subcategory spanned by those objects as in~\eqref{e:twisted-sp} such that $(\vphi^*g,\psi(j)) = (g',j')$. It will hence suffice to show that each $\C_{(g',j')}$ is weakly contractible. For this we will show that each $\C_{(g',j')}$ has a terminal object. Given $(g',j') \in \E_{\sig} \subseteq \widehat{\<n\>} \times \<m\>$ let $\Psi_{(g',j')} \in \C_{(g',j')}$ be the object corresponding to the diagram
\begin{equation}\label{e:twisted-sp-2} 
\vcenter{\xymatrix@R=1.9pc{
(\<n\> \cup \{a\},a,g_a) \ar[d]_{\tau_0} & \<n\> \ar[l]_-{\vphi_0}\ar^{\sig}[d] \\
(\<m\> \cup \{b\},b,g_b) \ar[r]^-{\psi_0} & \<m\> \\
}}
\end{equation}
where $(\<n\> \cup \{a\},a,g_a)$ and $(\<m\> \cup \{b\},b,g_b)$ are obtained by applying the adjoint functors of Remark~\ref{r:adjoint} to $(\<n\>,g')$ and $(\<m\>,j')$ respectively.
The map $\vphi_0: \<n\> \hrar \<n\> \cup \{a\}$ is the natural embedding 
and the map $\psi_0: \<m\> \cup \{b\} \lrar \<m\>$ is the identity when restricted to $\<m\>$ and sends $b$ to $j'$. 
Finally, the map $\tau_0: \<n\> \cup \{a\} \lrar \<m\> \cup \{b\}$ is uniquely determined by universal mapping properties insured by Remark~\ref{r:adjoint}. More explicitly, $\tau_0$ sends $a$ to $b$, identifies with $\sig$ on $\{i \in \<n\> | \sig(i) \neq j'\} \cup (g')^{-1}(0)$, and sends every $i \in \sig^{-1}(j') \cap (g')^{-1}(1)$ to $b$. 
It is then clear that $\Psi_{(g',j')}$ is an object of $\Tw(\Del_{\spl})_{/\sig}$ which maps to $(g',j') \in \E_{\sig}$, and is hence contained in $\C_{(g',j')}$. We now claim that $\Psi_{(g',j')}$ is terminal in $\C_{(g',j')}$. Indeed, suppose that $\Psi \in \C$ is an object as in~\eqref{e:twisted-sp} such that $(\vphi^*g,\psi(j)) = (g',j')$. Then maps $\Psi \lrar \Psi_{(g',j')}$ in $\C_{(g',j')}$ correspond to diagrams of the form
\begin{equation}\label{e:twisted-sp-3} 
\vcenter{\xymatrix@R=1.9pc{
(\<l\>,i,g) \ar[d]_-{\tau}  & (\<n\> \cup \{a\},a,g_a) \ar^-{\sig_0}[d]\ar_-{\vphi'}[l] & \<n\> \ar[l]_-{\vphi_0}\ar^{\sig}[d] \\
(\<k\>,j,h) \ar[r]^-{\psi'} & (\<m\> \cup \{b\},b,g_b) \ar^-{\psi_0}[r] & \<m\> \\
}}
\end{equation}
with $\vphi',\psi'$ maps of split $(x,y)$-ordinals and such that the external rectangle identifies with~\eqref{e:twisted-sp}. The existence of a unique such pair $\vphi',\psi'$ now follows from the universal mapping properties of $(\<n\> \cup \{a\},a,g_a)$ and $(\<m\> \cup \{b\},b,g_b)$ provided by Remark~\ref{r:adjoint}. 
\end{proof}

\begin{proof}[Proof of Proposition~\ref{p:adj-2}]
By Corollary~\ref{c:step-1} and Lemma~\ref{l:step-2} it will suffice to prove that the homotopy pushout $\widehat{\<n\>} \coprod^h_{\E_{\sig}} \<m\>$
is weakly contractible. 
Since $\widehat{\<n\>}, \<m\>$ and $\E_{\sig}$ are all discrete sets this homotopy pushout is equivalent to the underlying space of a bipartite graph $G$ whose set of vertices is $\widehat{\<n\>} \coprod \<m\>$ and such that $(g,j) \in \widehat{\<n\>} \times \<m\>$ is an edge if and only if $j$ is compatible with $g$ in the sense of Definition~\ref{d:comp}. Let us show that $G$ is connected. Let $j \in \<m\>$ be an element. If $j > 0$ then we may consider the gap $g_-: \<n\> \lrar \<2\>$ given by $g_-(i) = 0 \Leftrightarrow \sig(i) < j$. Then both $j$ and $j-1$ are compatible with $g_-$ and so $j$ is connected to $j-1$ in $G$. It then follows that all of $\<m\>$ lies in a single component of $G$. Similarly, if $g: \<n\> \lrar \<2\>$ is a gap such that $g^{-1}(0)$ is non-empty and we set $i_{\max} = \max(g^{-1}(0))$ then $g$ is compatible with $\sig(i_{\max})$. On the other hand, the gap $g':\<n\> \lrar \<2\>$ given by $g'(i)=0 \Leftrightarrow i<i_{\max}$ is also compatible with $\sig(i_{\max})$, and so $g$ and $g'$ are connected in $G$. We hence get that all of $\widehat{\<n\>}$ lies in the same component. Finally, since there are edges connecting $\widehat{\<n\>}$ and $\<m\>$ we may conclude that $G$ is connected.  

To show that $G$ is weakly contractible it will hence suffice to show that the number of edges is equal to the number of vertices minus $1$. But this just follows from the direct observation that the valency of the vertex corresponding to $j \in \<m\>$ is equal to $|\sig^{-1}(j)|+1$ if $x \leq j \leq n-1-y$, equal to $|\sig^{-1}(j)|-1$ if $j=x=y=m=1$ and is equal to $|\sig^{-1}(j)|$ in all other cases. This means that the total number of edges is $m+n-x-y$, while the total number of vertices is $|\left<m\right>| + |\widehat{\<n\>}| = m+n+1-x-y$.
\end{proof}

\section{Scaled unstraightening and the Grothendieck construction}\label{s:discrete}
In this section we give a proof of Proposition \ref{p:grothendieck}, which compares the $\infty$-categorical Grothendieck construction of a 2-functor $\F: \CC\lrar \Cat_1$ (realized by the scaled unstraightening functor) to its $2$-categorical Grothendieck construction.
Let us start by recalling the following generalization of the Grothendieck construction mentioned in \S\ref{s:straightening}, which applies to (strict) 2-functors $\F: \CC\lrar \Cat_2$ from a 2-category to the 2-category of (strict) 2-categories  (see \cite{Buc14}):

\begin{define}\label{d:2-groth}
Let $\CC$ be a $2$-category and let $\F: \CC \lrar \Cat_2$ be a $2$-functor. The \textbf{Grothendieck construction} $\int_{\CC}\F$ is is the $2$-category defined as follows:
\begin{itemize}
\item
The objects of $\int_{\CC}\F$ are pairs $(A,X)$ with $A \in \CC$ and $X \in \F(A)$.
\item
The $1$-morphisms from $(A,X)$ to $(B,Y)$ are given by pairs $(f,\vphi)$, where $f: A \lrar B$ is a $1$-morphism in $\CC$ and $\vphi: f_!X \lrar Y$ is a morphism in $\F(B)$ (here $f_!=\F(f)$).
\item
If $(f,\vphi),(g,\psi)$ are two $1$-morphisms from $(A,X)$ to $(B,Y)$ then the $2$-morphisms from $(f,\vphi)$ to $(g,\psi)$ are given by pairs $(\sig,\Sigma)$ where $\sig: f \Rightarrow g$ is a $2$-morphism in $\CC$ and $\Sigma:\vphi \Rightarrow \psi \circ \sig_!$ is a $2$-cell in the diagram
\begin{equation}\label{e:sigma}
\vcenter{\xymatrix{
f_!X \ar^{\sig_!X}[rr]\ar_{\vphi}[dr] && g_!X \ar^{\psi}[dl]\\
& Y. \twocell{ur}{\Sigma}{0.3}{1.7}{0.15} &\\
}}
\end{equation}
\end{itemize}
The various compositions of $1$-morphisms and $2$-morphisms are defined in a straightforward way, see~\cite{Buc14}. The projection $(A,X) \mapsto A$ determines a canonical functor $\pi:\int_{\CC} \F \lrar \CC$.
\end{define}
\begin{rem}\label{r:groth-base}
The Grothendieck construction is evidently compatible with base change: given $2$-functors $g: \CC \lrar \CC'$ and $\F: \CC' \lrar \Cat_2$, there is a natural isomorphism $\int_{\CC} g^*\F \cong \CC \times_{\CC'} \int_{\CC'} \F$.
\end{rem}
Let $\Fun_2(\CC,\Cat_2)$ denote the $1$-category of $2$-functors $\CC \lrar \Cat_2$. The $2$-categorical Grothendieck construction described above can then be promoted to a functor $\Fun_2(\CC,\Cat_2) \lrar \Cat_2/\CC$ (of $1$-categories) and the Grothendieck construction described in \S\ref{s:straightening} is the restriction
\begin{equation}\label{e:restrictedgroth}\xymatrix{
\Fun_2(\CC,\Cat_1) \ar[r] &  \Fun_2(\CC,\Cat_2) \ar[r]^-\int &  \Cat_2/\CC.
}\end{equation}
Let us start by describing the image of the functor \eqref{e:restrictedgroth}.
\begin{define}
Let $p: \DD \lrar \CC$ be a $2$-functor. We will say that a $1$-morphism $e: x \lrar y$ is \textbf{$p$-coCartesian} if for every object $z \in \DD$ the diagram
\begin{equation}\label{e:cocart}\vcenter{\xymatrix{
\Map_{\DD}(y,z) \ar^-{e^*}[r]\ar[d] & \Map_{\DD}(x,z) \ar[d] \\
\Map_{\CC}(p(y),p(z)) \ar^-{p(e)^*}[r] & \Map_{\CC}(p(x),p(z)) \\
}}\end{equation}
is homotopy Cartesian.
\end{define}
\begin{rem}
When all vertical arrows in \eqref{e:cocart} are right (or left) fibrations, the condition that $e: x \lrar y$ is $p$-coCartesian can be checked locally in the following sense: for every $1$-morphism $g: p(y) \lrar p(z)$ in $\CC$ one needs to verify that the induced functor
$$ \Map_{\DD}(y,z)_{g} \x{e^*}{\lrar} \Map_{\DD}(x,z)_{g \circ p(e)} $$
is an equivalence. Here $\Map_{\DD}(y,z)_{g}$ denotes the homotopy fiber of $\Map_{\DD}(y,z) \lrar \Map_{\CC}(p(y),p(z))$ over $g$ and similarly for $\Map_{\DD}(x,z)_{g \circ p(e)}$.
\end{rem}
\begin{define}\label{d:2-fib}
Let $p: \DD \lrar \CC$ be a $2$-functor. We will say that $p$ is \textbf{opfibered in categories} if the following conditions are satisfied:
\begin{enumerate}
\item
For every $x,y \in \DD$ the functor $\Map_{\DD}(x,y) \lrar \Map_{\CC}(p(x),p(y))$ is a right fibration whose fibers are sets (i.e., fibered in sets in the sense of Grothendieck).
\item
For every $x \in \DD$ and $1$-morphism $f: p(x) \lrar y$ in $\CC$ there exists a $p$-coCartesian $1$-morphism $e: x \lrar y'$ in $\CC$ such that $p(e) = f$.
\end{enumerate}
\end{define}
If $p: \DD \lrar \CC$ is opfibered in categories, then $p^{\op}:\DD^{\op} \lrar \CC^{\op}$ is in particular a $2$-fibration in the sense of~\cite{Buc14}. By \cite[Theorem 2.2.11]{Buc14}, such a $2$-fibration is an unstraightened model of a $2$-functor $\CC^{\coop} \lrar \Cat_2$, whose value at an object $C$ is the fiber of $p^{\op}$ over $C$. On the other hand, if $p^{\op}$ is a $2$-fibration, then $p$ is opfibered in categories if and only if the fibers of $p$ are $1$-categories, i.e., the corresponding $2$-functor $\CC^{\coop} \lrar \Cat_2$ lands in $\Cat_1$.
The following is then a special case of~\cite[Proposition 3.3.4]{Buc14}:
\begin{pro}[\cite{Buc14}]\label{c:fact}
Let $\CC$ be a $2$-category and $\F:\CC \lrar \Cat_1$ a $2$-functor. Then the map $\int_{\CC}\F \lrar \CC$ is opfibered in categories.
\end{pro}
Recall from \S\ref{s:scaledsset} that the $2$-nerve $\rN_2(\CC)$ of a strict 2-category $\CC$ is an $\infty$-bicategory, i.e.\ a fibrant scaled simplicial set. We will write $\underline{\rN}_2(\CC)$ for the underlying simplicial set of $\rN_2(\CC)$. 
\begin{lem}\label{l:compatible}
Let $p:\DD \lrar \CC$ be a $2$-functor which is opfibered in categories. Then the induced map $\rN_2(\DD) \lrar \rN_2(\CC)$ is a scaled coCartesian fibration in the sense of Definition~\ref{d:2-fib-a}.
\end{lem}
\begin{proof}
Let us first show that the underlying map of simplicial sets $\underline{\rN}_2(\DD) \lrar \underline{\rN}_2(\CC)$ is an inner fibration. Given an inner horn inclusion $\iota:\Lam^n_i \hrar \Del^n$, the associated functor $\iota_*:\fC^{\sca}(\Lam^n_i) \lrar \fC^{\sca}(\Del^n)$ induces a bijection on objects and an isomorphism $\Map_{\fC^{\sca}(\Lam^n_i)}(j,j') \lrar \Map_{\fC^{\sca}(\Del^n)}(j,j')$ for all $(j,j') \neq (0,n)$. On the other hand, recall that $\Map_{\fC^{\sca}(\Del^n)}(0,n)\cong (\Del^1)^{\{1, ..., n-1\}}$ is an $(n-1)$-cube. If we denote by
$$
K = \partial(\Del^1)^{\{1,...,i-1,i+1,...,n-1\}} \lrar (\Del^1)^{\{1,...,i-1,i+1,...,n-1\}}=L
$$
the inclusion of the boundary of the $(n-2)$-cube obtained by forgetting the $i$-th coordinate, then $\Map_{\fC^{\sca}(\Lam^n_i)}(0,n) \lrar \Map_{\fC^{\sca}(\Del^n)}(0,n)$ can be identified with
$$ L\times\Del^{\{1\}} \coprod_{K\times\Del^{\{1\}}} K\times\Del^1 \subseteq L\times \Del^1 = (\Del^1)^{\{1,...,n-1\}}.$$
This map is right anodyne, being the pushout-product of the right anodyne map $\Del^{\{1\}} \hrar \Del^1$ and the inclusion $K\lrar L$. It follows from Condition (i) of Definition~\ref{d:2-fib} that $\DD_{\rN^+} \lrar \CC_{\rN^+}$ has the right lifting property with respect to $\iota_*:\fC^{\sca}(\Lam^n_i) \lrar \fC^{\sca}(\Del^n)$. Consequently, $\underline{\rN}_2(\DD) \lrar \underline{\rN}_2(\CC)$ has the right lifting property with respect to $\iota:\Lam^n_i \hrar \Del^n$, as desired.

Next we claim that if $\sig: \Del^2 \lrar \uline{\rN}_2(\CC)$ is a thin triangle, then $\sig^*f: \underline{\rN}_2(\DD) \times_{\underline{\rN}_2(\CC)} \Del^2 \lrar \Del^2$ is a coCartesian fibration. Indeed, in this case $\sigma$ determines a map $\BDel^2 \lrar \CC$ with values in the maximal sub-$(2,1)$-category of $\CC$, so we may reduce to the case where $\CC$ is a $(2,1)$-category. Condition (i) of Definition~\ref{d:2-fib} now implies that $\DD$ is a $(2,1)$-category as well, so that $\CC_{\rN^+}$ and $\DD_{\rN^+}$ are fibrant marked-simplicial categories whose mapping objects have all edges marked. The desired result now follows by applying~\cite[2.4.1.10]{Lur09} to the underlying simplicial categories of $\CC_{\rN^+}$ and $\DD_{\rN^+}$ respectively.

We conclude that $p: \underline{\rN}_2(\DD) \lrar \underline{\rN}_2(\CC)$ is a $T$-locally coCartesian fibration, 
where $T$ is the collection of thin triangles in $\rN_2(\CC)$. To finish the proof we have to show that the thin triangles in $\rN_2(\DD)$ are exactly those triangles whose image in $\rN_2(\CC)$ is thin. This is a direct consequence of Condition (i) of Definition~\ref{d:2-fib}, since right fibrations detect isomorphisms. 
\end{proof}
We can now consider two different ways to ``unstraighten'' a $2$-functor $\F: \CC \lrar \Cat_1$ into a map of scaled simplicial sets. On the one hand, we can take the Grothendieck construction $\int_{\CC}\F \lrar \CC$ and apply the $2$-nerve functor $\rN_2$ to obtain a map $\rN_2(\int_{\CC}\F) \lrar \rN_2(\CC)$. On the other hand, we can form the associated enriched functor $\rN^+\F: \CC_{\rN^+} \lrar \Set^+_{\Del}$ (obtained by applying $\rN^+$ to the values of $\F$ as well as to the action maps $\Map_{\CC}(c,d) \times \F(c) \lrar \F(d)$) and take the scaled unstraightening $\wtl{\Un}^{\sca}(\rN^+\F) \lrar \rN^{\sca}(\CC_{\rN^+}) \cong \rN_2(\CC)$ (see Notation~\ref{n:variants} and Notation~\ref{n:counit}). 
We now claim the following:
\begin{pro}\label{p:construction}
For $\F:\CC \lrar \Cat_1$ there exists a natural map 
\begin{equation}\label{e:theta_F}
\Theta_{\CC}(\F): \rN_2\left(\int_{\CC}\F\right) \lrar \wtl{\Un}^{\sca}(\rN^+\F)
\end{equation}
of scaled simplicial sets over $\rN_2(\CC)$ with the following properties:
\begin{enumerate}
\item
$\Theta_{\CC}(\F)$ preserves locally coCartesian edges over $\rN_2(\CC)$.
\item
For every $2$-functor $g: \CC \lrar \CC'$ and every $\F: \CC' \lrar \Cat_1$ the diagram
$$ \xymatrix@C=4pc{
\rN_2\left(\int_{\CC'}g^*\F\right) \ar^-{\Theta_{\CC}(\F)}[r]\ar[d] & \wtl{\Un}^{\sca}(\rN^+g^*\F) \ar[d] \\
\rN_2\left(\int_{\CC'}\F\right) \ar_-{\Theta_{\CC'}(\F)}[r] & \wtl{\Un}^{\sca}(\rN^+\F) \\
}$$
commutes.
\end{enumerate}
\end{pro}
We will construct \eqref{e:theta_F} from a natural transformation between the associated left adjoint functors. To this end, observe that the sequence of functors \eqref{e:restrictedgroth} gives rise to a sequence of left adjoints
$$ 
\LL_1: \Cat_2/\CC \x{\LL}{\lrar} \Fun_2(\CC,\Cat_2) \x{|-|_1}{\lrar} \Fun_2(\CC,\Cat_1). 
$$ 
The functor $|-|_1$ is given pointwise by sending a 2-category $\DD$ to the 1-category $|\DD|_1$ with the same objects and hom-sets $\Hom_{|\DD|_1}(x,y) = \pi_0|\Map_{\DD}(x,y)|$ (see \eqref{e:gpdcomp}). The left adjoint $\LL$ to the 2-categorical Grothendieck construction exists by the adjoint functor theorem, but can also be described explicitly as follows (cf.\ \cite[\S 4.2]{Buc14}). Given a $2$-functor $f:\DD \lrar \CC$, let $\DD_{/c}$ be the $2$-category where 
\begin{itemize}
\item an object is a pair $(d,\alp)$, where $d$ is an object of $\DD$ and $\alp: f(d) \lrar c$ is a morphism in $\CC$.
\item a $1$-morphism is a pair $(\bet,\tau):(d,\alp) \lrar (d',\alp')$, where $\bet: d \lrar d'$ is a $1$-morphism in $\DD$ and $\tau: \alp \Rightarrow \alp' \circ f(\bet)$ is a $2$-morphism in $\CC$.
\item a $2$-cell $(\bet,\tau) \Rightarrow (\bet',\tau')$ is a $2$-cell $\sig: \bet \Rightarrow \bet'$ such that the diagram
\begin{equation}\label{e:sigma-2}
\vcenter{\xymatrix{
\alp' \circ f(\bet) \ar@{=>}^{f(\sig)}[rr] && \alp' \circ f(\bet')\\
& \alp \ar@{=>}_{\tau'}[ur] \ar@{=>}^{\tau}[ul] &\\
}}
\end{equation}
commutes in $\Map_{\CC}(\alp,\alp' \circ f(\bet'))$.
\end{itemize}
The left adjoint $\LL$ to the Grothendieck construction $\int: \Fun_2(\CC,\Cat_2)\lrar \Cat_2/\CC$ then sends $f: \DD\lrar \CC$ to the 2-functor
$$
\LL(f): \CC\lrar \Cat_2; \hspace{5pt} c \mapsto \DD_{/c}.
$$
\begin{rem}
The analogous description of the left adjoint to the $1$-categorical Grothendieck construction is well-known (see, e.g., \cite[Proposition 3.1.2]{Mal05}). The above $2$-categorical analogue can be proven in a similar fashion, by explicitly describing the unit and counit. More precisely, the unit $u: \DD\lrar \int_{c\in\CC} \DD_{/c}$ sends $d$ to the tuple $(f(d),(d,\Id_{f(d)}))$ and the counit $\nu: \LL(\int_{\CC} \F) \Rightarrow \F$ sends $(x, \alpha: c'\lrar c)$ in $(\int_{\CC} \F)_{/c}$ to $\alpha_!(x)$ in $\F(c)$.
\end{rem}
\begin{rem}\label{r:left-kan}
Remark \ref{r:groth-base} implies, by passing to left adjoints, that $\LL_1$ is compatible with ($\Cat_1$-enriched) left Kan extensions: if $f: \DD \lrar \CC$ and $g: \CC \lrar \CC'$ are $2$-functors then there is a natural isomorphism $\LL_1(gf) \cong \Lan_g(\LL_1(f))$ of functors $\CC' \lrar \Cat_1$. 
\end{rem}
We conclude that the composite left adjoint $\LL_1: (\Cat_2)_{/\CC} \lrar \Fun_2(\CC,\Cat)$ sends $f: \DD \lrar \CC$ to the functor $\LL_1(f): \CC \lrar \Cat; \hspace{2pt}c\mapsto |\DD_{/c}|_1$. We will prove Proposition~\ref{p:construction} by relating this left adjoint $\LL_1$ to the scaled straightening functor of~\cite[\S 3.5]{goodwillie}. To do this, it will be convenient to describe $\LL_1$ in terms of \textbf{lax cones}.
\begin{define}\label{d:cone}
Let $\DD$ be a $2$-category. The \textbf{lax cone} $\LaxCone(\DD)$ on $\DD$ is the $2$-category with object set $\{\ast\} \cup \Ob(\DD)$ and mapping categories 
$$
\Map_{\LaxCone(\DD)}(x,y) = \Map_{\DD}(x,y) \qquad \Map_{\LaxCone(\DD)}(x,\ast) = \varnothing \qquad \Map_{\LaxCone(\DD)}(\ast,x) = |\DD_{/x}|_1
$$
for $x, y\in \DD$. The composition is defined using the functorial dependence of $|\DD_{/x}|_1$ on $x \in \DD$. Similarly, if $f: \DD \lrar \CC$ is a $2$-functor, then the lax cone of $f$ is the 2-category $\LaxCone(f) := \LaxCone(\DD) \coprod_{\DD} \CC$.
\end{define}
\begin{rem}
The reason for the terminology of Definition~\ref{d:cone} is that for any $2$-category $\EE$ the data of a $2$-functor $\LaxCone(\DD) \lrar \EE$ is equivalent to the data of a $2$-functor $p: \DD \lrar \EE$ together with a \textbf{lax natural transformation} from a constant diagram to $p$ (see~\cite[Theorem 11]{St76}).
\end{rem}
For every 2-functor $f: \DD \lrar \CC$, there is a natural isomorphism of functors $\CC \lrar \Cat_1$
$$ \LL_1(f) \cong \Map_{\LaxCone(f)}(\ast,-). $$
Indeed, when $f$ is the identity map this holds by construction. For more general functors $f$, it follows from the universal property of pushouts that $\Map_{\LaxCone(f)}(\ast,-)$ is the ($\Cat_1$-enriched) left Kan extension of $\Map_{\LaxCone(\DD)}(\ast,-)= \LL_1(\Id_\D)$ along $f$, which can be identified with $\LL_1(f)$ by Remark \ref{r:left-kan}.

Now recall that the scaled straightening functor $\St^{\sca}$ of \cite{goodwillie} is also defined in terms of a suitable cone construction: for a marked simplicial set $X = (\uline{X},E_X)$, the \textbf{scaled cone} of $X$ (see~\cite[Definition 3.5.1]{goodwillie}) is given by 
$$ \Cone(X) = (\uline{X} \times \Del^1,T) \coprod_{(\uline{X} \times \{0\})_{\flat}}\{\ast\} ,$$
where $T$ is the collection of those triangles $(\sig,\tau): \Del^2 \lrar \uline{X} \times \Del^1$ such that $\sig$ is degenerate and such that either $\sig|_{\Del^{\{0,1\}}}$ belongs to $E_X$ or $\tau|_{\Del^{\{1,2\}}}$ is degenerate. Given a marked-simplicial category $\CC$, the scaled unstraightening functor $\St^{\sca}:(\Set^+_\Del)_{/\rN^{\sca}(\CC)} \lrar \Fun^+(\CC,\Set^+_\Del)$ is then given by
$$ \St^{\sca}(X) = \Map_{\fC^{\sca}(\Cone(X))\coprod_{\fC^{\sca}(\uline{X}_{\flat})}\CC}(\ast,-) .$$

\begin{lem}\label{l:truncatedcone}
Let $\fC_2: \Set_\Del^{\sca}\lrar \Cat_2$ be the left adjoint to the 2-nerve $\rN_2$ (see Remark \ref{r:nerve}). Then there is a natural transformation of simplicial objects in the category $(\Cat_2)_\ast$ of pointed 2-categories
\begin{equation}\label{e:sig_n}
\Psi_\bullet: \fC_2\big(\Cone((\Del^{\bullet})^\flat)\big) \lrar \LaxCone(\BDel^\bullet).
\end{equation}
\end{lem}
\begin{rem}\label{r:hocat}
Let $\Ho_{\leq 1}: \Set^+_\Del \lrar \Cat_1$ denote the left adjoint of the marked nerve $\rN^+$, which sends a marked simplicial set $(\uline{S}, E_S)$ to the category freely generated by the simplicial set $\uline{S}$, localized at the arrows from $E_S$. If $X$ is a scaled simplicial set, then $\fC_2(X)$ is the 2-category obtained from the marked-simplicial category $\fC^{\sca}(X)$ by applying $\Ho_{\leq 1}$ to the mapping objects.
\end{rem}
\begin{proof}
Let us start by describing the $2$-category $\LaxCone(\BDel^n)$ more explicitly. For $i, j\in [n]$, the mapping category $\Map_{\LaxCone(\BDel^n)}(i, i')$ is the poset of chains $C\subseteq [n]$ starting at $i$ and ending at $i'$, ordered by inclusion. To describe the category of maps $\ast\lrar i$, observe that $\BDel^n_{/i}$ can be identified with the $2$-category whose objects are chains $C \subseteq [n]$ ending at $i$: such a chain determines a map $\min(C)\lrar i$ in $\BDel^n$. If $C$ and $C'$ are two such chains, then
$$
\Map_{\BDel^n_{/i}}(C, C') = \big\{D\subseteq [n] : \min(D)=\min(C), \max(D)=\min(C'), C\subseteq D\cup C'\big\}
$$
is a subposet of chains in $[n]$, ordered by inclusion. In particular, $\Map_{\BDel^n_{/i}}(C, C')$ is nonempty if and only if $\min(C) \leq \min(C')$ and each $j\in C$ is contained in $C'$ as soon as $j\geq \min(C')$. In that case, the poset contains a maximal chain, namely the interval $[\min(C), \min(C')]$. It follows that the associated 1-category (see \eqref{e:gpdcomp})
$$
\Map_{\LaxCone(\BDel^n)}(\ast, i) = |\BDel^n_{/i}|_1
$$
is the poset of chains $C\subseteq [n]$ ending at $i$, where $C\leq C'$ if $\min(C)\leq \min(C')$ and if each $j\in C$ with $j\geq \min(C')$ is also contained in $C'$.

To describe $\fC_2\big(\Cone((\Del^n)^\flat)\big)$, let us start by identifying $\fC^{\sca}(\Del^n\times \Del^1, T)$, where the scaling $T$ is described above Lemma \ref{l:truncatedcone}. By \cite[Remark 3.7.5]{goodwillie}, this marked-simplicial category has objects $(i,\eps) \in [n] \times [1]$, while $\Map_{\fC(\Del^n \times \Del^1)}((i,\eps),(i',\eps'))$ is the nerve of the poset of chains $C\subseteq [n]\times [1]$ starting at $(i,\eps)$ and ending at $(i',\eps')$. When $\eps=\eps'$, this is simply a poset of chains in $[n]=[n]\times\{\eps\}$. 

On the other hand, let us denote by $\P_{i,i'}$ the poset of chains from $(i,0)$ to $(i',1)$ and for each such chain $C$, let $C_0=C\cap ([n]\times \{0\})$ and $C_1=C\cap ([n]\times \{1\})$ be the associated two chains in $[n]$. Examining the scaling $T$, we see that all the marked edges $W$ lie in these $\P_{i, i'}$: an inclusion $C\subseteq C'$ determines a marked edge in $\Map_{\fC(\Del^n\times \Del^1, T)}((i, 0), (i', 1))$ if and only if $C_0=C'_0$ and $C_1'= C_1 \cup \{\max(C_0)\}$. Using Remark \ref{r:hocat}, we therefore conclude that $\fC_2(\Del^n\times \Del^1, T)$ is the 2-category with objects $(i, \eps)$ and mapping categories 
\begin{align*}
\Map_{\fC_2(\Del^n\times \Del^1, T)}((i, 0), (i', 0))&=\Map_{\fC_2(\Del^n\times \Del^1, T)}((i, 1), (i', 1))=\BDel^n(i, i')\\
\Map_{\fC_2(\Del^n\times \Del^1, T)}((i, 0), (i', 1))&= \P_{i, i'}[W^{-1}].
\end{align*}
Composition proceeds by concatenation of chains. Since the functor $\fC_2$ is a left adjoint and $\fC_2(\ast)=\ast$, there is a natural isomorphism
$$
\fC_2\big(\Cone((\Del^n)^\flat)\big) \cong \fC_2\big(\Del^n\times \Del^1, T\big)\coprod_{\fC_2(\Del^n\times\{0\})} \ast.
$$
By the above isomorphism, the natural transformation $\Psi_\bullet$ of \eqref{e:sig_n} is determined uniquely by natural functors $\Psi_n: \fC_2(\Del^n\times \Del^1, T)\lrar \LaxCone(\BDel^n)$ collapsing $\fC_2(\Del^n\times\{0\})$ to $\ast$. We simply define these functors by
\begin{itemize}
\item $\Psi_n$ sends $\fC_2(\Del^n\times\{0\})$ to $\ast\in\LaxCone(\BDel^n)$.
\item $\Psi_n$ sends $\fC_2(\Del^n\times\{1\})$ isomorphically to $\BDel^n=\fC_2(\Del^n)\subseteq \LaxCone(\BDel^n)$.
\item $\Psi_n((i, 0), (i', 1)): \P_{i, i'}[W^{-1}]\lrar \|\BDel^n_{/i'}\|_1$ arises from the functor $\P_{i, i'}\lrar \|\BDel^n_{/i'}\|_1$ sending $C\mapsto \{\max(C_0)\}\cup C_1$, which indeed sends marked edges to identities.
\end{itemize}
This determines the desired natural transformation $\Psi_\bullet$ as in \eqref{e:sig_n}.
\end{proof}

\begin{proof}[Proof of Proposition~\ref{p:construction}]
It will suffice to define $\Theta_{\CC}(\F)$ on the underlying simplicial sets since the thin triangles on both sides of~\eqref{e:theta_F} are exactly those triangles whose image in $\rN_2(\CC)$ is thin. In particular, we need to construct a natural transformation $\uline{\rN}_2\int_\CC(-) \Rightarrow \uline{\Un}^{\sca}(-)$ between two functors $\Fun_2(\CC,\Cat_1) \lrar \Set_\Del$ which is compatible with base change. 

To do this, let us consider, for each simplicial set $X$, the natural map of pointed $2$-categories
\begin{equation}\label{e:sig-3}
\Psi(X): \fC_2(\Cone(X^{\flat})) \Rightarrow \LaxCone(\fC_2(X_{\flat}))
\end{equation}
defined as follows: since both sides of~\eqref{e:sig-3} are functors on $\Set_\Del$ which commute with colimits, the natural transformation $\Psi(-)$ is uniquely determined by its value on simplices, which we take to be the natural transformation $\Psi_\bullet$  of Lemma~\ref{l:truncatedcone}. For each 2-category $\CC$, this determines a natural transformation of functors $(\Set_\Del)_{/\uline{\rN}_2(\CC)}\lrar (\Cat_2)_{\ast\coprod \CC/}$
\begin{equation}\label{e:sig-2}
\Psi_{\CC}(X): \fC_2(\Cone(X^{\flat}))\coprod_{\fC_2(X_{\flat})}\CC \Rightarrow \LaxCone(\fC_2(X_{\flat})) \coprod_{\fC_2(X_{\flat})}\CC.
\end{equation}
This natural transformation $\Psi_{\CC}(-)$ is also natural in $\CC$. Taking mapping categories out of the basepoint $\ast$, we obtain a natural transformation of functors $(\Set_\Del)_{/\uline{\rN}_2(\CC)}\lrar \Fun(\CC, \Cat_1)$
\begin{equation}\label{e:sig}
\Sig_{\CC}(X): \Ho_{\leq 1}\St^{\sca}(X^{\flat}) \Rightarrow \LL_1(\fC_2(X_{\flat}))
\end{equation}
where $\Ho_{\leq 1}$ is the functor from Remark \ref{r:hocat}. Since $\Psi_{\CC}$ depends naturally on $\CC$, the natural transformation $\Sig_\CC(X)$ is compatible with $\Cat_1$-enriched left Kan extensions along functors $\CC\lrar \CC'$. The natural transformation $\Sig_{\CC}$ is therefore adjoint to a natural transformation of functors $\Fun(\CC, \Cat_1)\lrar (\Set_\Del)_{/\uline{\rN}_2(\CC)}$
$$
\Theta_{\CC}(\F): \uline{\rN}_2\left(\int_{\CC}\F\right) \lrar \uline{\Un}^{\sca}(\rN^+\F)
$$
which is compatible with base change, as desired.

It remains to be shown that this $\Theta_\CC(\F)$ preserves coCartesian edges. In light of the compatibility with base change (ii), it will suffice to work over $\CC = [1] = \bullet \lrar \bullet$.  
Unwinding the definitions, we see that $\LL_1([1]): [1] \lrar \Cat_1$ is the diagram of categories $\{0\} \hrar [1]$. 
A natural transformation $\sig:\LL_1([1]) \Rightarrow \F$ is adjoint to a coCartesian edge of $\int_{[1]}\F \lrar [1]$ if and only if $\sig(1)$ maps $\LL_1([1])(1) = [1]$ to an isomorphism in $\F(1)$.
On the other hand, $\Ho_{\leq 1}\St^{\sca}(\Del^1):[1] \lrar \Cat_1$ is the functor depicted by the diagram 
$$ \Ho_{\leq 1}\Del^{\{1\}} \hrar \Ho_{\leq 1}\Big(\Lam^2_0 \coprod_{\Del^{\{0,1\}}} (\Del^{\{0,1\}})^{\sharp}\Big) .$$ 
A natural transformation $\tau:\Ho_{\leq 1}\St^{\sca}(\Del^1) \Rightarrow \F$ is adjoint to a marked edge of $\Un^{\sca}_{\Del^1}(\rN^+\F)$ if and only if it factors through $\Ho_{\leq 1}\St^{\sca}((\Del^1)^{\sharp}) = \Ho_{\leq 1}(\St^{\sca}(\Del^1)^{\sharp})$, i.e., if $\tau(1)$ sends $\Del^{\{0,2\}} \subseteq \Lam^2_0$ 
to an isomorphism in $\F(1)$.  
The desired result now follows from the fact that $\Sig_{[1]}(\Del^1)(1): \Ho_{\leq 1}(\Lam^2_0 \coprod_{\Del^{\{0,1\}}} (\Del^{\{0,1\}})^{\sharp}) \lrar [1]$ maps the edge corresponding to $\Del^{\{0,2\}}$ onto $[1]$. 
\end{proof}
Proposition \ref{p:grothendieck} now follows from the following:
\begin{pro}
The map $\Theta_{\CC}(\F)$ \eqref{e:theta_F}
constructed above is a bicategorical equivalence of scaled simplicial sets over $\rN_2(\CC)$.
\end{pro}
\begin{proof}
By Lemma~\ref{l:compatible} and Proposition~\ref{p:construction}(i) we know that $\Theta_{\CC}(\F)$ is a map between two scaled coCartesian fibrations over $\rN_2(\CC)$ which preserves locally coCartesian edges. We may hence promote it to a natural map in the model category $(\Set^+_\Del)^{\lcc}_{/\rN_2(\CC)}$
\begin{equation}\label{e:theta-3} 
\Theta^+_{\CC}(\F): \underline{\rN}_2\left(\int_{\CC}\F\right)^{\natural} \lrar \Un^{\sca}\big(\rN^+(\F)\big).
\end{equation}
By Lemma~\ref{l:same} we see that $\Theta_{\CC}(\F)$ \eqref{e:theta_F} is an equivalence of scaled simplicial sets if~\eqref{e:theta-3} is an equivalence in $(\Set^+_\Del)^{\lcc}_{/\rN_2(\CC)}$. To show the latter it will suffice to show that for every $x \in \rN_2(\CC)$ the induced map
$$ \underline{\rN}_2\left(\int_{\CC}\F\right)^{\natural} \times_{\rN_2(\CC)}\{x\} \lrar \Un^{\sca}(\rN^+\F)\times_{\rN_2(\CC)}\{x\} $$
is a categorical equivalence of marked simplicial sets. Since $\Theta_{\CC}(\F)$ is compatible with base change we see that we now just need to prove the proposition in the case $\CC = \ast$. In this case the data of $\F$ is just a category $\C$ and~\eqref{e:theta-3} becomes a natural transformation of the form
\begin{equation}\label{e:theta-4}
\Theta^+_{\ast}(\C):\underline{\rN}(\C)^{\natural} = \rN^+(\C) \lrar \Un^{\sca}_{\ast}(\rN^+(\C))
\end{equation}
The restriction of this natural transformation to $\Del\subseteq \Cat_1$, corresponds under the adjunction $\St^{\sca}_{\ast}\dashv\Un^{\sca}_{\ast}$, to a natural transformation of cosimplicial objects in $\Set^+_\Del$
$$
\St^{\sca}_\ast((\Del^\bullet)^\flat)\lrar (\Del^\bullet)^\flat
$$
and hence extends to a natural transformation of left Quillen functors $\alpha: \St^{\sca}_{\ast}\Rightarrow \Id_{\Set^+_\Del}$. As explained in the beginning of~\cite[\S 3.6]{goodwillie}, there is \textbf{only one} such natural transformation $\alpha$, which is an equivalence by \cite[Proposition 3.6.1]{goodwillie}. Since $\rN^+$ is fully faithful, the map $\Theta^+_{\ast}(\C)$ is the component of the adjoint natural transformation $\alpha^{\ad}:\Id_{\Set^+_\Del}\Rightarrow \Un^{\sca}_{\ast}$ at $\rN^+(\C)$. Since $\rN^+(\C)$ is fibrant, we conclude that $\Theta^+_{\ast}(\C)$ is an equivalence.
\end{proof}

\end{document}